\documentclass[12pt,oneside]{amsart}

\usepackage{geometry}
\usepackage[style=alphabetic,sorting=nyt,backend=biber,maxnames = 99,maxalphanames=99]{biblatex} 
\usepackage{graphicx}
\usepackage{todonotes}
\usepackage{amsmath}
\usepackage{amsfonts}
\usepackage{amssymb}
\usepackage{tikz-cd}
\usepackage{comment}
\usepackage{bbm}
\usepackage{stmaryrd}
\usepackage{fullpage}
\usepackage{epstopdf}
\usepackage{amsthm}
\usepackage{mathrsfs} 
\usepackage{tikz}
\usepackage{dsfont}
\usetikzlibrary{matrix}
\usepackage{mathtools}
\usepackage[hidelinks]{hyperref}
\usepackage{cleveref}
\usepackage{xcolor}
\usepackage{enumitem}
\usepackage{tensor}
\usepackage[utf8]{inputenc}
\usepackage{csquotes}
\usepackage{mymacros}

\usepackage[T1]{fontenc}
\DeclareFontFamily{U}{min}{}
\DeclareFontShape{U}{min}{m}{n}{<-> udmj30}{}
\newcommand{\yo}{\!\text{\usefont{U}{min}{m}{n}\symbol{'210}}\!}

\hypersetup{
	colorlinks,
	linkcolor={red!50!black},
	citecolor={blue!50!black},
	urlcolor={blue!80!black}
}

\tikzset{
	rot90/.style={anchor=south, rotate=90, inner sep=.5mm}
}
\tikzset{
	rot45/.style={anchor=south, rotate=-45, inner sep=.5mm}
}

\theoremstyle{definition}

\newtheorem{nul}{}[section]
\newtheorem{dfn}[nul]{Definition}

\newtheorem{rmk}[nul]{Remark}

\newtheorem{cnstr}[nul]{Construction}

\newtheorem{ntn}[nul]{Notation}
\newtheorem{exm}[nul]{Example}

\newtheorem{wrn}[nul]{Warning}
\newtheorem{qst}[nul]{Question}

\newtheorem*{dfn*}{Definition}
\newtheorem*{axm*}{Axiom}
\newtheorem*{ntn*}{Notation}
\newtheorem*{exm*}{Example}
\newtheorem*{exr*}{Exercise}
\newtheorem*{int*}{Intuition}
\newtheorem*{qst*}{Question}
\newtheorem*{rmk*}{Remark}

\theoremstyle{plain}

\newtheorem{thm}[nul]{Theorem}
\newtheorem{prop}[nul]{Proposition}

\newtheorem{lem}[nul]{Lemma}

\newtheorem{cor}[nul]{Corollary}

\newtheorem*{thm*}{Theorem}
\newtheorem*{prop*}{Proposition}
\newtheorem*{cor*}{Corollary}
\newtheorem*{lem*}{Lemma}
\newtheorem*{cnj*}{Conjecture}

\newcommand\xqed[1]{%
	\leavevmode\unskip\penalty9999 \hbox{}\nobreak\hfill
	\quad\hbox{#1}}
\newcommand\tqed{\xqed{$\triangleleft$}}

\title{\texorpdfstring{$c$}{c}-structures and trace methods beyond connective rings}

\author{Ishan Levy}
\address{Department of Mathematics, Institute of Advanced Studies, USA}
\email{ishanl@ias.edu}

\author{Vladimir Sosnilo}
\address{
RIKEN iTHEMS, Wako, Saitama 351-0198, Japan
}
\email{\href{mailto:vsosnilo@gmail.com}{vsosnilo@gmail.com}}

\newtoggle{draft}
\togglefalse{draft}

\iftoggle{draft} {
	\usepackage[notcite]{showkeys}
	\newcommand{\NB}[1]{\todo[color=gray!40]{#1}}
	\newcommand{\TODO}[1]{\todo[color=red]{#1}}
}{ 
	\newcommand{\NB}[1]{}
	\newcommand{\TODO}[1]{}	
	\renewcommand{\todo}[1]{}
	\renewcommand{\todo}[1]{}
}
\date{\today}

\begin{document}

\begin{abstract}
	We introduce the notion of a $c$-category, which is a kind of category whose behaviour is controlled by connective ring spectra. More precisely, 
	any $c$-category admits a finite step resolution by categories of compact modules over 
	connective ring spectra. We introduce nilpotent extensions of $c$-categories, and show that they induce isomorphisms on truncating invariants, such as the fiber of the cyclotomic trace map. We show that for many stacks, the category of perfect complexes is naturally a $c$-category and 
	deduce a generalization of the Dundas--Goodwillie--McCarthy theorem to such stacks. 
\end{abstract}

\maketitle{}
\tableofcontents

\section*{Introduction}

There are many results in homological algebra and homotopy theory that require connectivity hypotheses in order to work. 
As an example, we can consider the Dundas--Goodwillie--McCarthy theorem, a fundamental result in the algebraic K-theory of ring spectra, which asserts in particular that the natural square
\begin{center}
	\begin{tikzcd}
		\K(R) \ar[r]\ar[d] &\TC(R) \ar[d]\\
		\ar[r] \K(S)& \TC(S)
	\end{tikzcd}
\end{center}
is pullback for a map of connective ring spectra $R \to S$ that is a nilpotent extension, i.e. it is surjective on 
$\pi_0$ with nilpotent kernel. Because $\TC$ is often much more accessible a priori than $K$-theory, using the nilpotent extension $R \to \pi_0R$, this essentially 
reduces the study of $\K$-theory of $R$ to that of $\pi_0R$ and the study of $\TC$. 

Many categories which one would like to understand do not fit into the framework of connective rings. For example, quasi-coherent sheaves on a scheme or algebraic stack that isn't an affine scheme do not, and many categories in stable homotopy theory such as the $K(n)$-local category also do not. 
In this paper, we introduce the notion of a \textit{$c$-structure} on a stable category 
$\cC$ (Definition~\ref{dfn:c_structure}), which is a type of structure that encodes a weak connectivity\footnote{$c$ stands 
for connectivity.} property of $\cC$ that allows one to extend many results that are true for connective rings to $\cC$. 
We call a stable category with a $c$-structure a \textit{$c$-category}. We list some properties of $c$-categories. 
\begin{enumerate}
	\item Any bounded weighted category in the sense of Bondarko is a $c$-category.
	\item For a $c$-category $\cC$, $\Ind(\cC)$ admits a weight structure, which does not necessarily restrict to $\cC$. 
	Moreover, every object of $\cC$ is bounded in this weight structure, i.e. we have the inclusion $\cC \subset \Ind(\cC)^{\mathrm{b}}$. 
	\item For large enough cardinals $\kappa$ we also have the embedding $\cC \to \Ind(\cC)^{\mathrm{b},\kappa}$ into $\kappa$-small bounded objects, and the quotient 
	\[
 		\frac{\Ind(\cC)^{\mathrm{b},\kappa}}{\cC}
	\]
	is again a $c$-category. Repeating this procedure finitely many times produces a bounded weighted category. 
\end{enumerate}
The number of steps needed for the procedure above to produce a bounded weighted category is called the \textit{width} of a 
$c$-category.

We now explain how trace methods can be extended to $c$-categories. Recall that a localizing invariant $E$ is truncating if it 
induces an equivalence $E(R) \to E(S)$ for $R \to S$ a nilpotent extension of connective rings. For example, the 
Dundas--Goodwillie--McCarthy theorem is equivalent to the fact that the functor $\K^{\inv} = \fib(\K \to \TC)$ is a truncating 
invariant. 
Recall from \cite[Definition~5.1.1]{nilpotentextns} that a functor $f:\cC \to \cD$ of bounded weighted categories is a 
nilpotent extension if it is surjective on equivalence classes of objects in the weight heart, morphisms between them, 
and the kernel ideal of morphisms is nilpotent. 
\begin{thm}[{\ref{thm:truncating_invariants}}]\label{thm:nilpintro}
	If $f:\cC \to \cD$ is a map of $c$-categories such that $\Ind(\cC)^{\mathrm{b}} \to \Ind(\cD)^{\mathrm{b}}$ is a nilpotent extension, then $f$ induces an isomorphism on all truncating invariants.
\end{thm}
In particular, the theorem above shows that $\K(\cC)$ can be understood in terms of $\K(\cD)$ and $\TC$.

A typical example of a nilpotent extension of $c$-categories is as follows. Let $G$ be a pro-$p$-group of finite cohomological 
dimension, and consider the map $(\ZZ/p^n)^{hG} \to \FF_p^{hG}$ of bounded below rings. On perfect module categories, this 
induces a nilpotent extension of $c$-categories of width given by the cohomological dimension of $G$. \Cref{thm:nilpintro} 
then implies that there is a pullback square
\begin{center}
	\begin{tikzcd}
		\K((\ZZ/p^n)^{hG}) \ar[r]\ar[d] & \K(\FF_p^{hG})\ar[d]\\
	 \TC((\ZZ/p^n)^{hG})	\ar[r] & \TC(\FF_p^{hG})
	\end{tikzcd}
\end{center}
Using the regularity of $\FF_p$, $\K(\FF_p^{hG})$ can be identified with $\K(\FF_p)$ (see \cite{kcoconn}), reducing the 
understanding of $\K((\ZZ/p^n)^{hG})$ to a $\TC$ computation. This example illustrates how $c$-categories allow one to access 
structural information about categories that are `bounded below' but not quite coming from connective rings. 

Our results are naturally viewed as a generalization of both the results of the first author \cite{levy2022algebraic} and 
Elmanto and the second author \cite{nilpotentextns}. Indeed, the former shows that trace methods extend to $-1$-connective 
rings, which can be interpreted in the context of this paper as saying that perfect module categories of $-1$-connective rings 
naturally have a $c$-structure (see \Cref{exm:minusone}). The latter shows that trace methods extend to bounded weighted 
categories, which corresponds to the fact that these naturally have $c$-structures of width $0$. We will also show that a 
stable category $\cC$, for which $\Ind(\cC)$ is weakly approximable in the sense of Neeman 
(see \cite{neeman2018triangulated}), admits a $c$-structure. The axioms in the definition of a weakly approximable category 
are in fact very close to that of a $c$-structure. However, note that different choices of a generator give rise to different 
$c$-structures.

The key ingredient to \Cref{thm:nilpintro} is the following categorical statement which relates $c$-categories to 
perfect module categories of connective rings, whose proof is the direct application of the procedure introduced above.
\begin{thm}[\Cref{cor:resolution_full}]\label{thm:catresintro}
	Fix $n\geq1$ and an uncountable regular cardinal $\kappa$.
	Let $\cC$ be a $\kappa$-small $c$-category of width $n$.
	There exists a natural exact complex of stable categories 
	\[
		0 \to \cC \to \cC_0 \to \cC_1 \to \dots \to \cC_{n-1} \to \cC_n \to 0
	\]
	such that every localizing invariant vanishes on $\cC_i$ for $i<n$, and each $\cC_i$ is a equivalent to the perfect module 
	category of a connective ring.
\end{thm}
The above theorem gives a natural isomorphism that $E(\cC) \cong \Sigma^{-n}E(\cC_n)$ for each localizing invariant $E$, and 
$\cC_n$ is the perfect module category of a connective ring spectrum. To obtain \Cref{thm:nilpintro} from this, one just needs 
to show that given a nilpotent extension $\cC \to \cC'$ of $c$-categories, the induced maps $\cC_i \to \cD_i$ come 
from nilpotent extensions of connective rings. 

The main applications we discuss in this paper are to the theory of algebraic stacks (Section \Cref{sec:stacks}). We show that 
for a good enough quasi-geometric stack $X$ the category of perfect complexes $\mathrm{Perf}(X)$ admits a $c$-structure. 
Moreover, for an affine map of such stacks $X \to Y$ the pullback map $\mathrm{Perf}(Y) \to \mathrm{Perf}(X)$ is a
morphism of $c$-categories and for a nilpotent extension it is, furthermore, a nilpotent extension of $c$-categories. 
This, combined with \Cref{thm:nilpintro}, automatically yields fundamental results about $\K$-theory of algebraic stacks, 
previously unknown in this generality. 

Theorem~\ref{thm:catresintro} provides `cohomological' obstructions to both existence of $c$-structures and weak 
approximability, hence allows us to answer positvely the question of Canonaco, Neeman, and Stellari about the necessity of 
both axioms in the definition of weak approximability \cite{canonaco2024}.

The results of this paper should be applicable to understand structural information about many other categories of interest. 
For example, the $\K$-theory of many interesting categories and ring spectra in chromatic homotopy theory can be accessed 
using our methods. In \cite{levy2022algebraic}, the algebraic $\K$-theory of the $L_1$-local and $K(1)$-local categories were 
accessed using a version \Cref{thm:nilpintro} for $-1$-connective ring spectra, essentially reducing it to the 
study of $\TC(j_{\zeta})$ where $j_{\zeta}$ is an $\EE_{\infty}$-ring that is the $-1$-connective cover of the $K(1)$-local 
sphere. Being able to access these $\K$-theories using $\TC$ of $-1$-connective rings is one of the key ingredients in the 
work of the first author with Burklund, Hahn, and Schlank \cite{telescope} in order to disprove the telescope conjecture at 
heights $\geq2$. Forthcoming work of the first author will use the tools of this paper to give integral formulas for the 
algebraic $\K$-theory of certain $K(n)$-local categories and $L_n$-local categories in terms of $\TC$.

The paper is structured as follows. 
We begin by reviewing weighted categories in \Cref{sec:weights}, and then introduce fcd heart categories in \Cref{sec:fcdht}, 
a slightly stronger notion than that of a $c$-category. In \Cref{sec:cstr} we study $c$-structures, and prove 
\Cref{thm:catresintro}. In \Cref{sec:morphisms}, we study special kinds of morphisms of $c$-categories such as 
nilpotent extensions, 
and show these are suitably compatible with the categorical resolution of \Cref{thm:catresintro}. 
In \Cref{sec:localizing_invariants_of_c_categories}, we obtain applications to localizing invariants, and 
in \Cref{sec:stacks}, we study $c$-structures on algebraic stacks. In \Cref{sec:counterexamples}, we provide examples to 
indicate subtleties in notions surrounding that of $c$-categories. 
We end with two appendicies. The first discusses the relation between $c$-categories and weak approximability in the sense 
of \cite{neeman2018triangulated}, and the second provides a characterization of those categories that admit resolutions such 
as those provided in \Cref{thm:catresintro}.

\addtocontents{toc}{\protect\setcounter{tocdepth}{1}}
\subsection*{Acknowledgments}

The authors are grateful to Ko Aoki, Robert Burklund, Sasha Efimov, Jeremy Hahn, Mike Hopkins, Akhil Mathew, Thomas Nikolaus, Piotr Pstragowski, Maxime Ramzi, Victor Saunier, Tomer Schlank, Christoph Winges, Yuchen Wu for helpful conversations related to this paper. We thank Victor Saunier for comments on an earlier draft of the paper.
We are also thankful to the organizers of the 2023 IHES summer school ``Recent Advances in Algebraic K-theory'', during which 
the project started. 

\subsection*{Conventions}

Some basic notions to be introduced:
\begin{enumerate}
	\item We use the word `category' to mean category.\footnote{or more precisely, $(\infty,1)$-category.} 
	We denote the category of small categories by $\Cat$, the category of presentable categories is denoted with $\PrL$.
	\item For a category $\cC$ we denote its ind-completion by $\Ind(\cC)$. A finite colimit-preserving functor 
	$f:\cC \to \cD$ induces the Kan extension functor $f_!:\Ind(\cC) \to \Ind(\cD)$ and the restriction functor 
	$f^*:\Ind(\cD) \to \Ind(\cC)$.
	\item A subcategory $\cC \subset \cC'$ is extension closed if for any fiber 
	sequence $x \to y \to z$ in $\cC'$ with $x,z\in \cC$, $y$ also belongs to $\cC'$. 
	\item The idempotent completion of $\cC$ is denoted by 
	$\cC^{\mathrm{ic}}$.
	\item The mapping space between any two objects in a category is denoted with $\Map(-,-)$. In case there is a canonical 
	spectral enrichment (e.g. when the category is stable), the mapping spectrum is denoted with $\map(-,-)$.
	\item We say that a sequence of functors $\cA \stackrel{i}\to \cB \stackrel{p}\to \cC$ between stable categories is a localization sequence if $i$ is fully faithful, $p\circ i \simeq 0$ and the natural induced map 
	$(\cB/\cA)^{\mathrm{ic}} \to (\cC)^{\mathrm{ic}}$ is an equivalence.
\end{enumerate}

\addtocontents{toc}{\protect\setcounter{tocdepth}{2}}
\section{Weighted categories}\label{sec:weights}

Here we recall some preliminary results about weight structures that we use throughout the paper. 
\begin{dfn}\label{dfn:weight} A {\bf weight structure} on a stable category $\cC$ is the data of two idempotent closed subcategories $(\cC_{w\geq 0}, \cC_{w\leq 0})$ of connective and, respectively, coconnective objects such that:
	\begin{enumerate}
		\item $\cC_{w\geq0}$ is closed under suspension, and $\cC_{w\leq0}$ is closed under desuspensions.
		\item $\map(x,y)$ is connective if $x \in \cC_{w\leq0}, y \in \cC_{w\geq0}$.
		\item for any object $x \in \cC$ there exists a cofiber sequence
		\[
		x_{\leq 0} \rightarrow x \rightarrow x_{\geq 1},
		\]
		where $x_{\leq 0} \in \cC_{w \leq 0}$ and $\Sigma^{-1}x_{\geq 1} \in  \cC_{w \geq 0}.$ We call these {\bf weight truncations} of $x$.
	\end{enumerate} 
	We warn the reader that despite the notation, the weight truncations of $x$ are not uniquely determined by $x$.
By a {\bf weighted category} we mean a stable category endowed with a weight structure. 
A functor between weighted stable categories 
\[
F: \cC \to \cD
\]
is said to have \textbf{weight amplitude $[a,b]$} if it sends connective objects to $a$-connective objects and coconnective objects to $b$-coconnective objects. It is said to be of \textbf{bounded weight amplitude} if it is weight amplitude $[a,b]$ for some $a,b \in \mathbf{Z}$. It is said to be \textbf{weight exact} if it is weight amplitude $[0,0]$.
The subcategory $\cC^{w\heartsuit} \subset \cC_{w\ge 0} \cap \cC_{w\le 0}$ is called the {\bf heart} of the weight structure. 
We write $\cC_{w \geq n} := \Sigma^n\cC_{w \geq 0}$ and $\cC_{w \leq k} := \Sigma^k\cC_{w \leq 0}.$

We say that the weight structure on $\cC$ restricts to a subcategory $\cD \subset \cC$ if 
$(\cD \cap \cC_{w\le 0},\cD \cap \cC_{w\ge 0})$ defines a weight structure on the subcategory.
\end{dfn}

Let $\cC$ be a stable category endowed with a weight structure.
Every object $x\in \cC$ admits shifted weight decompositions 
$w_{\le n} x \to X \to w_{\ge n+1} x$. We freely use the elementary fact 
that that these can be made compatible for any map $x \to y$ (though they are not functorial).

\begin{rmk}\label{rmk:orthogonality}
The notion of a weight structure is self dual, in that a weight structure on $\cC$ induces one on $\cC^{\mathrm{op}}$ by 
swapping the roles of $\cC_{\geq0}$ and $\cC_{\leq0}$. Additionally, the subcategories $\cC_{\geq0}$, $\cC_{\leq0}$ 
determine one another via the orthogonality axiom (2) of \Cref{dfn:weight}.
\tqed \end{rmk}

\begin{dfn}\label{dfn:smallweightedcat}
We say that the weight structure on $\cC$ is {\bf bounded} if 
$\cC = \bigcup\limits_n \left(\cC_{w \ge -n}\cap \cC_{w \le n}\right)$. We call a pair of a small stable category $\cC$ equipped with a bounded weight structure a  \textbf{bounded weighted category}.  
Given any weighted category $\cC$, we define $\cC^\mathrm{b} = \bigcup\limits_n \left(\cC_{w \ge -n}\cap \cC_{w \le n}\right)$, 
that is the maximal bounded weighted subcategory.
We denote the category of bounded weighted categories and weight exact functors by $\Cat^\mathrm{bw}$. We use 
$\Cat^\mathrm{bw}_\mathrm{b}$ to denote the cateory of bounded weighted categories and bounded weight amplitude functors.
\end{dfn}

\begin{rmk}\label{rmk:generating_bdd_ws}
Any generating subcategory $N \subset \cC$, for which the mapping spectra $\map(x,y)$ are 
connective for all $x,y \in N$, gives rise to a bounded weight structure on $\cC$ with 
\[
\cC_{w\ge 0} = \{\text{retracts of finite colimits of } N\},
\]
\[
\cC_{w\le 0} = \{\text{retracts of finite limits of } N\}.
\]
The heart in this case is the retract-closure of $N$ in $\cC$ (see \cite[Theorem 4.3.2(II)]{Bondarko_2010}). 
\tqed \end{rmk}

\begin{exm}\label{exm:weight_conn_ring}
Let $R$ be a connective associative ring spectrum. 
The category of perfect complexes $\mathrm{Perf}(R)$ admits a generating subcategory $\{R\}$ which satisfies the 
assumptions of Remark~\ref{rmk:generating_bdd_ws}. 
In particular, $\mathrm{Perf}(R)$ admits a bounded weight structure with 
$\mathrm{Perf}(R)_{w\ge 0}$ being the class of connective perfect complexes and $\mathrm{Perf}(R)_{w\le 0}$ being the 
minimal retract-closed extension-closed subcategory containing $\Sigma^n R$ for $n\le 0$. 
\tqed \end{exm}

\begin{exm}\label{exm:psigma}
For a small additive category $\cA$ consider the initial additive functor into a stable category 
$\cA \to \cA^{\mathrm{fin}}$ (see \cite[Definition~2.1.17]{nilpotentextns}). 
There is then a canonical weight structure 
$w$ on $\cA^{\mathrm{fin}}$ with $\cA$ as the heart (see \cite[Construction 2.2.7]{nilpotentextns}).
\tqed \end{exm}
The next theorem shows that in fact all bounded weighted categories take the form of Example~\ref{exm:psigma}.

\begin{thm}[{\cite[Corollary~3.4]{Sosnilo_2019}, \cite[Theorem~2.2.9]{nilpotentextns}}]\label{thm:vova_heart}
We have an adjoint pair
\[
((-)^{\mathrm{fin}},w): \mathbf{Cat}^{\mathrm{add}} \rightleftarrows \mathbf{Cat}^{\mathrm{bw}}: (-)^{w\heartsuit}.
\]
where the right adjoint functor of taking the heart $(-)^{w\heartsuit}$ is fully faithful. 
The adjoint pair restricts to an equivalence of categories of idempotent complete categories on both 
sides.
\end{thm}


The weight structures that are relevant to us will be generated by a collection of compact objects. To make sense of what 
it means for a weight structure to be generated, we have the following lemma, which explains the operations under which 
$\cC_{\leq0}$ is closed:

\begin{lem}\label{lem:closureofneg}
Suppose that $\cC$ is equipped with a weight structure. Then $\cC_{\leq0}$ is closed under infinite direct sums, 
finite limits, extensions, retracts, and countable filtered colimits of maps whose cofiber is in $\cC_{\leq0}$.
\end{lem}

\begin{proof}
Most of these are immediate from the characterization of $\cC_{\leq0}$ via condition (2) of \Cref{dfn:weight}. For 
countable filtered colimits of maps whose cofiber is in $\cC_{\leq0}$, the mapping spectrum to an object in $\cC_{\geq0}$ 
is a countable inverse limit of connective spectra under maps which are extensions of connective spectra. In particular, 
each map is a surjection on $\pi_0$, so the $\lim^1$-term computing $\pi_{-1}$ of the inverse limit vanishes, so the 
inverse limit is connective.
\end{proof}


\begin{dfn}\label{dfn:generation}
We say that a weight structure on $\cC$ is {\bf generated by a collection of objects} $S\subset \cC_{\leq0}$ if $S$ generates 
$\cC_{\leq0}$ under the operations in \Cref{lem:closureofneg}. 
A stable presentable category $\cC$ with a weight structure is called a {\bf compactly generated weighted category} if it 
admits a set of compact generators that also generates the weight structure on $\cC$.

We denote the category of presentable categories with a compactly generated weight structure and weight exact 
strongly cocontinuous functors by $\Cat^{\mathrm{cgw}}$.
\end{dfn}

\begin{rmk}\label{rmk:generated_ws}
In case $\cC_{\leq0}$ is generated by $S$, Remark~\ref{rmk:orthogonality} yields that we have 
\[
\cC_{\ge 0} = \{ x \in \cC : \map(y,x) \text{ is connective for all } y \in S\}.
\]
In particular, there is a unique weight structure generated by any set of objects.
\tqed \end{rmk}

\begin{dfn}\label{dfn:size}
The \textit{size} of a small category $\cC$ is the cardinality of the coproduct of all of the homotopy groups of all of the mapping 
spaces between all pairs of isomorphism classes of objects of $\cC$.
\end{dfn}

\begin{thm}[{\cite[Theorem~5]{Pauksztello}, \cite[Theorem~2.3.4]{Bondarko_2021}}]\label{thm:generated_ws}
Let $\cC$ be a compactly generated stable presentable category. 
Then for any set of compact objects $S \subset \cC$, there exists a weight structure on $\cC$ generated by $S$. 
If the size of $S$ is $<\kappa$ for some uncountable regular cardinal $\kappa$, then this weight structure restricts to 
$\kappa$-compact objects.
\end{thm}

%

\begin{exm}\label{exm:from_small_to_big_ws}
Let $\cC$ be a bounded weighted category. Then Theorem~\ref{thm:generated_ws} yields a weight structure on 
$\Ind(\cC)$ generated by $\cC_{w\le 0}$. In this case we have $\Ind(\cC_{w\ge 0}) \subset \Ind(\cC)_{w\ge 0}$ because 
$\cC_{w\ge 0}$ satisfies the necessary orthogonality condition. In fact, the opposite inclusion also holds. 
Indeed, by Theorem~\ref{thm:vova_heart} we have equivalence
\[
\Ind(\cC) \simeq \Fun_\mathrm{ex}(\cC^{\mathrm{op}}, \Sp) \simeq \Fun_\mathrm{add}((\cC^{w\heartsuit})^{\mathrm{op}}, \Sp).
\]
sending $x\in \Ind(\cC)$ to the presheaf $y \in \cC^{w\heartsuit} \mapsto \map(y,x)$. 
Under this equivalence we may rewrite 
\[
\Ind(\cC)_{w\ge 0} \simeq \Fun_\mathrm{add}((\cC^{w\heartsuit})^{\mathrm{op}}, \Sp_{\ge 0}).
\]
The latter category is compactly generated since representable presheaves are compact, hence we have 
$\Ind(\cC)_{w\ge 0} = \Ind(\cC_{w\ge 0})$.

The heart in this case consists of retracts of infinite sums of objects of $\cC^{w\heartsuit}$. To see this  
for $x \in \Ind(\cC)^{w\heartsuit}$ take the direct sum of all maps from objects of $\cC^{w\heartsuit}$ 
\[
\bigoplus\limits_i x_i \to x.
\]
The cofiber of this map $x'$ belongs to $\Ind(\cC)_{w\ge 1}$, so the map is a split epimorphism. 
For more details on this example see \cite[Section~4.5]{Bondarko_2010}.
\tqed \end{exm}

\begin{rmk}\label{rmk:full_embedding}
The construction described in Example~\ref{exm:from_small_to_big_ws} is functorial: for a weight exact functor 
$\cC \to \cD$ the functor $\Ind(\cC) \to \Ind(\cD)$ is also weight exact. In particular, this defines a functor
\[
\Cat^{\mathrm{bw}} \to \Cat^{\mathrm{cgw}}
\]
which is fully faithful when restricted to idempotent complete bounded weighted categories.
\tqed \end{rmk}


\begin{dfn}\label{dfn:hypercomplete}
We say that the weight structure $w$ on $\cC$ is {\bf hypercomplete} if $\bigcap\limits_n \cC_{w \ge n} = 0$. 
\end{dfn}

\begin{rmk}\label{rmk:generated-hypercomplete}
Let $S$ be a set of objects in a presentable category $\cC$ that generates a weight structure on $\cC$ under the 
operations of Lemma~\ref{lem:closureofneg}. This weight structure is hypercomplete if and only if $S$ generates $\cC$ as a 
localizing subcategory. Indeed, $\bigcap\limits_n \cC_{w \ge n}$ can be identified with the right orthogonal to the 
localizing subcategory generated by $S$.
\tqed \end{rmk}


\begin{prop}\label{prop:hearts_gen}
If $\cC$ is a stable presentable category that admits a hypercomplete weight structure such that $\cC_{w\ge 0}$ 
closed under colimits, then $\cC_{w\ge 0}$ is generated under small colimits by the objects of the heart.
\end{prop}
\begin{proof}
Denote by $\cC'$ the subcategory of $\cC$ generated by the heart $\cC^{w\heartsuit}$ under taking small colimits. It suffices 
to show that $\cC_{w\ge 0} \subset \cC'$. 
Objects of $\cC_{w\ge 0} \cap \cC_{w\le n}$ are finite extensions of suspensions of objects of the heart, so they all 
belong to $\cC'$. 
On the other hand, for $x \in \cC_{w\ge 0}$ the weight structure gives a filtration
$w_{\le *} x$ which fits into a fiber sequence of filtered objects 
\[
w_{\le *} x \to x \to w_{\ge *+1} x.
\]
The colimit $\colim w_{\ge *+1} x$ belongs to $\bigcap\limits_n \cC_{w\ge n}$ for all $n$, so is $0$ by hypercompleteness. Hence, we have $\colim w_{\le *} x \simeq x \in \cC'$.
\end{proof}

\begin{prop}\label{prop:hearts_cogen}
Let $\cC$ be a stable presentable category that admits a weight structure, generated by a set of 
weight bounded objects $S$ and assume that the weight structure restricts to $\cC^{\kappa}$ for 
some cardinal $\kappa$. Then the weight structure is generated by $\cC^{\kappa,w\heartsuit}$. 

If the weight structure is hypercomplete then $\cC^{\kappa,w\heartsuit}$ generates $\cC$ under 
small colimits and shifts. 
\end{prop}
\begin{proof}
Denote by $\cC'$ the subcategory of $\cC$ generated by $\cC^{\kappa,w\heartsuit}$ under the operations of 
Lemma~\ref{lem:closureofneg}. Certainly, $\cC'$ is contained in $\cC_{w\le 0}$, so it suffices to show that $S$ is contained 
in $\cC'$. This follows since any $x\in S$ is weight bounded, so it is a finite limit of objects of 
$\cC^{\kappa,w\heartsuit}$. 

The claim about hypercomplete weight structures follows since by 
Remark~\ref{rmk:generated-hypercomplete} $S$ generates $\cC$ as a localizing subcategory and $S$ belongs to $\cC'$. 

To show the moreover part note that the connectivity of $\map(-,x)$ is a property that is closed under the operations of 
Lemma~\ref{lem:closureofneg}. 
\end{proof}

\section{Finite cohomological dimension heart categories}\label{sec:fcdht}
\subsection{Basics on heart categories}

\begin{dfn}\label{dfn:heart-structure}
Let $\cC$ be a stable category. A {\bf heart structure} on $\cC$ consists of a pair of full subcategories 
$(\cC_{\le 0},\cC_{\ge 0})$ satisfying the following axioms:
\begin{enumerate}
\item $\cC_{\le 0}$ and $\cC_{\ge 0}$ are closed under extensions, 
$\cC_{\ge 0}$ under suspensions and $\cC_{\le 0}$ under desuspensions;
\item For any $x \in \cC$ there exists a fiber sequence called a decomposition 
\[
x_{\le 0} \to x \to x_{\ge 1}
\]
such that $x_{\le 0} \in \cC_{\le 0}$ and $x_{\ge 1} \in \Sigma\cC_{\ge 0}$. 
\end{enumerate} 

We will employ the following notation: $\cC_{\ge n} = \Sigma^n\cC_{\le 0}$, $\cC_{\le n} = \Sigma^n\cC_{\le 0}$. 
We say that the heart structure is {\bf bounded} if $\cC = \bigcup\limits_{n\in\NN} \cC_{\le n}\cap \cC_{\ge -n}$. 
The heart of $\cC$ is the subcategory spanned by $\cC_{\ge 0} \cap \cC_{w\le 0}$, we denote it by $\cC^{\heartsuit}$. 

We call a stable category endowed with a heart structure a {\bf heart category}. We say that a functor between 
heart categories $\cC \to \cD$ is {\bf heart exact} if it sends $\cC_{\le 0}$ to $\cD_{\le 0}$ and $\cC_{\ge 0}$ to 
$\cD_{\ge 0}$. The category of small heart categories is denoted by 
$\Cat^{\heartsuit}$. The full subcategory on bounded heart-categories is denoted by 
$\Cat^{\heartsuit, b}$. 
\end{dfn}


\begin{exm}\label{exm:derived_category}
An exact category $\cE$ in the sense of \cite{barwick2015exact} gives rise to the stable category 
$\mathcal{D}^\mathrm{b}(\cE)$, called its stable hull (see \cite{Klemenc_2022}). There $\mathcal{D}^\mathrm{b}(\cE)_{\ge 0}$ 
is set to be the extension-closure of $\bigcup\limits_{n\in\NN} \Sigma^n\cE$ and $\mathcal{D}^\mathrm{b}(\cE)_{\ge 0}$ is
$\bigcup\limits_{n\in\NN} \Omega^n\cE$.
If $\cE$ is a classical exact category, $\mathcal{D}^\mathrm{b}(\cE)$ coincides with the derived category in the classical sense. 
It turns out that, conversely, any bounded heart category arises from an exact category in this way. In fact, we have 
an equivalence 
\[
	\mathcal{D}^\mathrm{b}(\cC^{\heartsuit}) \simeq \cC
\]
by \cite[Theorem~1.6]{saunier-winges}.
\tqed \end{exm}

\begin{exm}
As a consequence of the previous example, we see that the category of perfect complexes over a quasi-projective scheme $X$ 
is naturally endowed with a heart structure, since we have
\[
	\mathrm{Perf}(X) \simeq \mathcal{D}^\mathrm{b}(\mathrm{Vect}(X)).
\]
\tqed \end{exm}

\begin{exm}\label{exm:localizations}
Let $\cC$ be a weighted category and consider an exact localization functor $F: \cC \to \cD$. 
Define $\cD_{\le 0}$ and $\cD_{\ge 0}$ to be the extension closures of $F(\cC_{w \le 0})$ and $F(\cC_{w \ge 0})$,  
respectively. 
Since every object of $\cD$ lifts to $\cC$, it is clear that the decompositions exist and that 
$(\cD_{\le 0}, \cD_{\ge 0})$ satisfies the axioms of a heart structure. 
Moreover, it is a bounded heart structure if $(\cC_{w\le 0}, \cC_{w\ge 0})$ is a bounded weight structure.

The above construction works similarly if $\cC$ is a heart category and $\cC \to \cD$ is any essentially surjective functor. 
\tqed \end{exm}

\begin{dfn}\label{dfn:fcd_heart_category}
Let $\cC$ be a bounded heart category. We say that $\cC$ is of cohomological dimension $\le n$ if there exits 
$n \in \mathbb{N}$ such that for any $x \in \cC_{\le 0}$ and $y\in \cC_{\ge 0}$ we have that $\map(x,y)$ is 
$-n$-connective. If the natural number is unspecified, we say that $\cC$ is of {\bf finite cohomological dimension} or 
{\bf fcd}.
\end{dfn}

\begin{rmk}\label{rmk:equiv-fcd-dfn}
It suffices to the condition in Definition~\ref{dfn:fcd_heart_category} for any $x, y \in \cC^\heartsuit$. To see this, note 
that $\cC_{\le 0}$ and $\cC_{\ge 0}$ are the extension-closures of $\bigcup\limits_{n\in \NN} \Sigma^{-n}\cC$ and 
$\bigcup\limits_{n\in \NN} \Sigma^n\cC$, respectively.
\tqed \end{rmk}

\begin{exm}
Let $\cA$ be a small abelian category. Then $\mathcal{D}^\mathrm{b}(\cA)$ is an fcd heart category if and only if all 
ext-groups in $\cA$ vanish above some $n$. 
\tqed \end{exm}

\begin{cnstr}\label{cnstr:t-str-and-weight-str}
Let $\cC$ be a bounded heart category. Using Proposition~\ref{thm:generated_ws} we may endow $\Ind(\cC)$ with a 
canonical weight structure so that $\Ind(\cC)_{w\ge 0}$ consists of objects $x\in \cC$ with 
$\map_{\cC}(y, \Sigma x)$ connective for all $y\in \cC_{\le 0}$. Since $\cC_{\le 0}$ generates $\Ind(\cC)$ it is 
furthermore hypercomplete (see Remark~\ref{rmk:generated-hypercomplete}).

If $\cC$ is an fcd heart category, then every object of $\cC$ is bounded in the sense of the weight structure on $\Ind(\cC)$. 
In particular, we have an exact embedding 
\[
	\cC \to \Ind(\cC)^\mathrm{b}.
\]
By Example~\ref{exm:localizations}, the quotient category 
\[
	\frac{\Ind(\cC)^\mathrm{b}}{\cC}
\]
is also a heart category. It turns out that it is also an fcd heart 
category\footnote{it is possible to prove directly, but we 
do not include a proof here, as it will be discussed in next section.} 
and, unless $\cC$ is a weighted category, the cohomological dimension of the quotient is strictly smaller than that of $\cC$. 
In particular, in finitely many steps we may resolve $\cC$ by bounded weighted categories. 

One may try to repeat the above construction for any weight structure on $\Ind(\cC)$ such that $\cC$ is bounded in it. It is 
then possible that the quotient becomes a heart category even if $\cC$ is not a heart category. 
A $c$-structure on $\cC$, defined in the next section, is precisely what guarantees that to happen. 
\end{cnstr}

\begin{wrn}
We may also define an accessible t-structure on $\Ind(\cC)$ with $\Ind(\cC)_{t \ge 0}$ generated under colimits by 
$\cC_{\ge 0}$ (see \cite[Proposition~1.4.4.11(1)]{HA}). We note that $\Ind(\cC)_{w\ge 0} \neq \Ind(\cC)_{t\ge 0}$ unless $\cC$ 
is a bounded weighted category. However, by Corollary~\ref{cor:weight_heart_relation} below, that there is always an inclusion 
in one direction.
\end{wrn}

\subsection{Tools for heart categories}

This subsection is to include several technical tools we need to work with heart structures.

\begin{dfn}\label{dfn:left_special}
Let $\cD$ be an exact category. 
A full extension closed subcategory $\cE \subset \cD$ is called {\bf left special} if for every exact epimorphism 
$u \twoheadrightarrow y$ in $\cD$ with $y\in \cE$ there exists a map $x \to u$ with $x\in \cE$ such that the composite 
$x \to u \to y$ is an epimorphism in $\cE$. 
\end{dfn}

\begin{exm}\label{exm:special_subcategory}
For any exact category $\cE$, the embedding $\cE \subset \mathcal{D}^{\mathrm{b}}(\cE)_{\ge 0}$ is left special 
(see \cite[Lemma~3.12]{saunier-winges}). 
\tqed \end{exm}

\begin{prop}\label{prop:speciality}
Let $\cE$ be an extension closed subcategory of a prestable category $\cC_{\ge 0}$. 
And assume that $\cC_{\ge 0}$ is generated by $\cE$ under finite colimits. 
The following conditions are equivalent:
\begin{enumerate}
\item $\cE$ is left special in $\cC_{\ge 0}$;
\item for any $x\in \cE$ and $y \in \Sigma \cC_{\ge 0}$ any morphism $x\to y$ decomposes as 
\[
x \to \Sigma u \to y
\]
for some $u\in \cE$;
\item for any $x,y \in \cE$ and $n>0$ any morphism $x \to \Sigma^n y$ decomposes as a composite
\[
x \to \Sigma u \to \Sigma^n y
\]
for some $u \in \cE$.
\end{enumerate}
\end{prop}
\begin{proof} 
First we prove (2) assuming (1). 
Consider $x\in \cE$, $y \in \Sigma \cC_{\ge 0}$ and a morphism $\phi:x\to y$. 
The fiber of $\phi$ belongs to $\cC_{\ge 0}$ and the fiber of the map $\fib(\phi) \to x$
is also in $\cC_{\ge 0}$. Hence, $\fib(\phi) \to x$ is an exact epimorphism in $\cC_{\ge 0}$. 
Left specialty implies that there exists $x' \in \cE$ and a map $x' \to \fib(\phi)$ such that 
the composite 
\[
p: x' \to \fib(\phi) \to x
\]
is an epimorphism. Now, $\fib(p) \in \cE$ and we see that the map $x \to y$ factors through 
$\cof(p)= \Sigma \fib(p)$.

The condition (3) is a special case of (2), so it suffices to prove that (3) implies (1). 
Consider $x \in \cE$ and an exact epimorphism $\phi:w \to x \in \cC_{\ge 0}$. Then $\cof(\phi) \in \Sigma\cC_{\ge 0}$. 
To show left specialty it suffices to construct an exact epimorphism 
$x' \twoheadrightarrow x \in \cE$ such that the composite $x' \to x \to \cof(\phi)$ is zero. 
If $\cof(\phi) \in \bigcup\limits_{n>0} \Sigma^n\cE$, then by assumption we may choose a factorization 
\[
x \to \Sigma u \to \cof(\phi)
\]
and take $x'$ to be the fiber of the map $x\to \Sigma u$. Now 
consider the subcategory $\cD \subset \Sigma\cC_{\ge 0}$ consisting of objects $y$ such that for any $v\in \cE$ 
and any map $v \to y$ there exists an exact epimorphism $v' \twoheadrightarrow v \in \cE$ such that the composite 
$v' \to v \to y$ is zero. 
To finish the proof it suffices to show that $\cD$ is extension closed. 

Consider a map $\psi:v \to y$ such that $y$ fits into an exact sequence
\[
y' \stackrel{g}\to y \stackrel{f}\to y''
\] 
with $y', y'' \in \cD$. By assumption there exists an exact epimorphism $p:v' \twoheadrightarrow v$ in 
$\cE$ such that $f\circ \psi \circ p$ is 
zero. In particular, the map $v' \to y$ factors through $y'$. Using the assumption on $y'$, there exists an exact 
epimorphism $q:v'' \twoheadrightarrow v'$ such that the composite 
\[
v'' \to v' \to y' \to y
\] 
is zero. Thus the composite $v'' \twoheadrightarrow v' \twoheadrightarrow v$ is zero after composed with $\psi$ and 
so $y$ belongs to $\cD.$ 
\end{proof}

\begin{thm}[Keller's criterion {\cite[Theorem~4.5]{saunier-winges}}]\label{thm:Kellers_criterion}
Let $\cE \subset \cD$ be a left special subcategory of an exact category $\cD$. 
Then the induced functor 
\[
\tau:\mathcal{D}^\mathrm{b}(\cE) \to \mathcal{D}^\mathrm{b}(\cD)
\]
is fully faithful. 
\end{thm}

\begin{cor}\label{cor:derivedcat_heartcat}
Let $\cC$ be a bounded heart category. Then the embedding $\cC^{\heartsuit} \to \cC$ is left special and the natural functor
\[
	\mathcal{D}^\mathrm{b}(\cC^{\heartsuit}) \to \cC
\]
is an equivalence. 
\end{cor}
\begin{proof}
The embedding $\cC^{\heartsuit} \to \cC_{\ge 0}$ is left special, since for any exact epimorphism 
\[
	x' \twoheadrightarrow x 
\]
with $x \in \cC^{\heart}, x' \in \cC_{\ge 0}$, we may consider the composite $x'_{\le 0} \to x' \to x$. 
On the other hand, we have $\mathcal{D}^\mathrm{b}(\cC_{\ge 0})\simeq \mathrm{SW}(\cC_{\ge 0}) \simeq \cC$ since $\cC_{\ge 0}$ 
generates $\cC$, so the result follows from \Cref{thm:Kellers_criterion}. 
\end{proof}

\begin{lem}\label{lem:wtheartforheartstr}	
Let $\cC$ be a bounded heart category. Then every object $c$ of the weight-heart in $\Ind(\cC)$ can be written as a retract 
of an object of the form $c_{\infty}=\colim_ic_i$ where $c_0=0$, and $c_i/c_{i-1}$ is a possibly infinite sum of objects in 
$\cC^{\heart}$.
\end{lem}
\begin{proof}
Let $c$ be in the weight heart. Let $c_1$ be the direct sum of all homotopy classes of objects in $\cC^{\heart}$ 
equipped with a map to $c$. Given $c_i$, define $c_{i+1}$ to be the cofiber of the direct sum of all homotopy classes 
of objects in $\Sigma^{-1}\cC^{\heart}$ equipped with a map to $c_i$. We claim that $c_{\infty}$ is in the weight 
heart. To see this, we need to show that any map $x \to c_{\infty}$ with $x \in \Sigma^{-i}\cC^{\heart}$ for $i>0$ is 
null. Since $c_{\infty}$ is a filtered colimit of objects in $\cC^{\heart}$, 
it follows from Proposition~\ref{prop:speciality} (with $\cE = \cC^{\heart}$) and \Cref{cor:derivedcat_heartcat} that we can reduce to the case 
$i=1$. But now the map $x \to c_{\infty}$ factors through $c_n$ for some $n$, and is thus null, because it is zero after 
composition with $c_n \to c_{n+1}$ by construction. 
	
The map $c_1 \to c$ extends to $c_{\infty}$ since $c$ is in the weight heart, so the obstructions vanish. We claim that the 
cofiber is $\geq1$ in the weight structure. This claim finishes the proof, since it shows that the map $c_{\infty} \to c$ 
splits, so that $c$ is a summand of $c_{\infty}$. To see the claim, the cofiber $d$ is clearly $\geq0$ in the weight 
structure, so we just need to show that there are no nonzero homotopy classes of maps $x \to d$ for $x \in \cC^{\heart}$. 
If $f:x \to d$ is a map, it is null upon composing with the map $d \to \Sigma c_{\infty}$, so the map lifts to the fiber, 
$c$. But then by construction of $c_1$, the map lifts to $c_1$, so the original map is null, since it factors through the 
maps $c_1 \to c_{\infty} \to c \to d$, and the composite $c_{\infty} \to c \to d$ is null.
\end{proof}

\begin{cor}\label{cor:weight_heart_relation}
Let $\cC$ be a bounded heart category. 
Then the inclusion $\Ind(\cC)_{w\ge 0} \subset \Ind(\cC)_{t\ge 0}$ holds. 
\end{cor}
\begin{proof}
By Proposition~\ref{prop:hearts_gen} any object of $\Ind(\cC)_{w\ge 0}$ is a colimit of objects of the weight-heart. 
Since $\Ind(\cC)_{t\ge 0} = \Ind(\cC_{\ge 0})$ is closed under colimits, it suffices to observe that objects of the 
weight-heart belong there, which follows from Lemma~\ref{lem:wtheartforheartstr}. 
\end{proof}

\begin{prop}\label{prop:Karoubi-completion}
Let $\cC$ be a bounded heart category, then the following holds:
\begin{enumerate}
\item $(\cC_{\le 0}^{\mathrm{ic}}, \cC_{\ge 0}^{\mathrm{ic}})$ defines a bounded heart structure on $\cC^{\mathrm{ic}}$,
\item the functor $\cC^{\heartsuit} \to (\cC^{\mathrm{ic}})^{\heartsuit}$ is the idempotent completion map.
\end{enumerate} 
%
%
%
%
\end{prop}
\begin{proof} 
See \cite[Proposition~2.28]{winges2025presentable}.
\end{proof}

\begin{exm}\label{exm:minusone}
Let $R$ be a $-1$-connective ring spectrum $R$. It follows from 
\cite[Proposition~2.1.3(1)]{noncommutative_localizations} and \Cref{prop:Karoubi-completion} that $\mathrm{Perf}(R)$ has a heart structure of cohomological dimension 1 such that $(\mathrm{Perf}(R))_{\le 0}$ and $(\mathrm{Perf}(R))_{\ge 0}$ being the retract closure of the extension closures of the 
subcategories 
$\{ \Sigma^{-n} R  \text{ for all } n\ge 0 \}$ and $\{ \Sigma^n R  \text{ for all } n\ge 0 \}$. 
\tqed \end{exm}

\begin{rmk}
	$\mathrm{Perf}(R)$ for an arbitrary $-n$-connective ring spectrum $R$ for $n>1$ may not come from a fcd heart category, or 
	even the weaker notion of a $c$-category to be introduced in the next section (see \Cref{thm:localizingsubcat}). Using 
	\Cref{thm:Kellers_criterion} and \Cref{prop:speciality}, it is possible to see that the subcategory $\cE$ generated under 
	extensions and retracts by $R$ forms the heart of a heart structure iff for any map $a \to \Sigma^nc$ with $a,c \in \cE$, 
	there is a factorization $a \to \Sigma b \to \Sigma^nc$ with $b \in \cE$.
\tqed \end{rmk}

\section{\texorpdfstring{$c$-structures}{c-structures}}\label{sec:cstr}

\begin{dfn}\label{dfn:preccategory}
Let $\cC$ be an idempotent complete small stable category with a choice of a generating subcategory 
$\cC_{\le 0}\subset \cC$, closed under finite limits, extensions and retracts. This gives rise to a hypercomplete compactly 
generated weight structure on $\Ind(\cC)$ in the sense of Definition~\ref{dfn:generation}. We say that $\cC_{\le 0}$ is a 
{\bf pre-$c$-structure} for $\cC$ if every object of $\cC$ is bounded in the weight structure. In this case we say that $\cC$ 
(with the choice of $\cC_{w\le 0}$) is a {\bf pre-$c$-category}. 

A morphism of pre-$c$-categories $(\cC, \cC_{\le 0}) \to (\cD, \cD_{\le 0})$ is an exact functor $\cC \to \cD$ that induces 
a weight exact functor $\Ind(\cC) \to \Ind(\cD)$. They form a category $\Cat^{\mathrm{prec}}$. 
We say that a pre-$c$-category $\cC$ is $\kappa$-small if the underlying stable category is $\kappa$-small. 
We denote the subcategory of $\kappa$-small pre-$c$-categories by $\Cat^{\mathrm{prec},\kappa}$.
\end{dfn}

\begin{lem}\label{lem:weight0_morphism}
Let $C$ be a weighted category and $f:x\to y \in C$ be a morphism. 
The following conditions are equivalent:
\begin{enumerate}
\item $f$ can be decomposed as $x \to x' \to y$ where $x'\in C_{w\ge k+1}$, 
\item there exists a weight decomposition of $x$ such that the composite $w_{\le k} x \to x \to y$ is zero,
\item for any $t \in C_{w\le k}$ and any map $t \to x$ the composite $t \to x \to y$ is zero.
\end{enumerate}
\end{lem}
\begin{proof}
(1) implies (2) and (3), since already the composites $t \to x \to x'$ and $w_{\le k}x \to x \to x'$ are zero by 
orthogonality. Furtheremore, (2) implies (1) since under the assumption, the map $x\to y$ factors through $w_{\ge k+1}x$. 
\end{proof}

\begin{dfn}[\cite{Bondarko2022}]\label{dfn:weight0_morphism}
Let $C$ be a weighted category and $f:x\to y \in C$ be any morphism. 
We say that $f$ {\bf kills weights $\le k$} if the 
equivalent conditions of Lemma~\ref{lem:weight0_morphism} are satisfied.
\end{dfn}

\begin{dfn}\label{dfn:c_structure}
Let $\cC$ be a pre-$c$-category and let $n$ be a natural number. 
We say that the pre-$c$-structure on $\cC$ is a $c$-structure of width $\le n$ if 
any $x\in \Ind(\cC)_{w\ge 0}$ can be presented as a filtered colimit of a diagram of $\cC$ such 
that for every map $x_{\alpha} \to x_{\alpha'}$ in the diagram, it can be postcomposed with a map $x_{\alpha'} \to x_{\alpha''}$ such that $x_{\alpha} \to x_{\alpha''}$ kill weights $\le -n-1$; 

The full subcategory of $\Cat^{prec}$ spanned by $c$-categories of width $\le n$ is denoted by $\Cat^{\mathrm{conn}}_n$. 
We denote the subcategory of $\kappa$-small $c$-categories by $\Cat^{\mathrm{conn},\kappa}_n$. 
\end{dfn}

\begin{exm}\label{exm:from_heart_to_c}
If $\cC$ is a bounded heart category of cohomological dimension $\le n$, then $(\cC, \cC_{\le 0})$ is a $c$-category. 
Indeed, we at least have an inclusion $\cC_{\ge 0} \subset \Ind(\cC)_{w\ge -n}$. In particular, $\cC_{\ge -k}\cap \cC_{\le k}$ 
is contained inside $\Ind(\cC)_{w\ge -n-k} \cap \Ind(\cC)_{w\le k}$, so at least this yields a pre-$c$-structure on $\cC$. 
Moreover, by Corollary~\ref{cor:weight_heart_relation} we have a series of inclusions
\[
\Ind(\cC)_{w\ge 0} \subset \Ind(\cC_{\ge 0}) \subset \Ind(\Ind(\cC)_{w\ge -n} \cap \cC)
\]
which shows that it is actually a $c$-structure of width $\le n$.
\tqed \end{exm}

\begin{rmk}\label{rmk:abelian_weight_heart}
Let $A$ be an abelian category. Then $\mathcal{D}^\mathrm{b}(A)$ admits a heart structure of 
Example~\ref{exm:derived_category} and in particular, there is an induced weight structure on 
$\Ind(\mathcal{D}^\mathrm{b}(A))$. Recall that an object $I \in \Ind(A)$ is injective if and only if  
$\mathrm{Ext}^1_{\Ind(A)}(x/u, I)$ is zero for any $x\in A$ and $u \hookrightarrow x$ a subobject in $\Ind(A)$. By 
Example~\ref{exm:special_subcategory} and Proposition~\ref{prop:speciality} any injective object also satisfies 
\[
	\map(x,I)
\]
is connective for any $x\in A$. In particular, any positively graded complex of injectives is connective. 

If we also assume that $A$ is noetherian, then also the opposite inclusion holds: any object of the weight heart is injective. 
Indeed, in this case both $u$ and $x/u$ belong in $A$, and so the condition $\mathrm{Ext}^1_{\Ind(-)}(x/u, I)$ is 
automatically satisfied for any $I \in \Ind(\mathcal{D}^\mathrm{b}(A))^{w\heart}$. 

From this we see that the class $\Ind(\mathcal{D}^\mathrm{b}(A))_{w\ge 0}$ is equal to the class of positively graded 
complexes of injectives, and in particular $\Ind(\mathcal{D}^\mathrm{b}(A))^{w\heart}$ is the class of injective objects of 
$\Ind(A)$. This is false in the non-noetherian case, since objects of the weight heart are closed under direct sums, while 
injectives are in general not (see \cite[Theorem~1.1]{Bass1962}).
\tqed \end{rmk}

\begin{wrn}\label{wrn:preserving_injectives}
The $c$-structure constructed in Example~\ref{exm:from_heart_to_c} does not recover the heart structure, even if we 
ignore potential retract problems: in fact, under the assumptions above $(\cC_{\le 0}, \cC \cap \Ind(\cC)_{w\ge -n})$ 
defines a heart structure on the same category (see Proposition~\ref{prop:cstr_heartstructure}), which is typically different 
from the original one but has the same underlying $c$-category. 

Furthermore, this assignment is not functorial: a heart exact functor $\cC \to \cD$ does not necessarily induce 
a weight exact functor $\Ind(\cC) \to \Ind(\cD)$. For instance, an exact functor of abelian categories $f:A \to B$ induces a 
heart exact functor $\mathcal{D}^\mathrm{b}(A) \to \mathcal{D}^\mathrm{b}(B)$. However, if we assume $A$ and $B$ to be 
noetherian, by Remark~\ref{rmk:abelian_weight_heart} the weight exactness of 
\[
\Ind(\mathcal{D}^\mathrm{b}(A)) \to \Ind(\mathcal{D}^\mathrm{b}(B))
\]
is equivalent to $f_!:\Ind(A) \to \Ind(B)$ preserving injectives. This condition neither implies nor follows from exactness. 
\end{wrn}

\begin{prop}\label{prop:cstr_heartstructure}
Suppose that $\cC$ is a $c$-category of width $n$ such that every object in $\Ind(\cC)^{w\heartsuit}$ is a filtered colimit of 
compact objects which are $[-n,0]$ in the weight structure. Then $\cC$ comes from a bounded heart category of cohomological 
dimension $\le n$ via the construction of \Cref{exm:from_heart_to_c}. More precisely, it arises from a heart structure on 
$\cC$ where $\cC_{\leq 0}$ is the compact objects which are $\leq 0$ in the weight structure, and $\cC_{\geq 0}$ is the 
compact objects which are $\geq-n$ in the weight structure.
\end{prop}
\begin{proof}
The only condition in the definition of a heart structure that is not immediate is the existence of decompositions. 
Since each object in $\cC$ is bounded, it is enough to show that for an object $x \in\cC$ that is $\leq0$, it admits a 
decomposition $x_{\leq -1} \to x \to x_{\geq0}$. Choose a weight decomposition $y \to x \to z$ so that $z$ is in the weight 
heart and $y$ is weight $-1$-coconnective. By our assumptions, $x \to z$ factors through some $x_{\geq0}$ which is compact and 
$[-n,0]$ in the weight structure. The fiber of $x \to x_{\geq0}$ is compact, so it suffices to show that it is $\leq -1$ in 
the weight structure. This fiber is an extension of $y$ by $\Sigma^{-1} \fib(x_{\geq0} \to z)$. But the latter is $\leq -1$ 
since it is an extension of $\Sigma^{-1} x_{\geq 0}$ and $\Sigma^{-2} z$.
\end{proof}

\subsection{Resolving \texorpdfstring{$c$-categories}{c-categories}}

A bounded weight structure gives rise to a $c$-structure of width $\le 0$ (see Example~\ref{exm:from_heart_to_c}). 
While there are lots of examples of $c$-structures that do not appear in this way, the main theorem of this section is that one 
can always present a category $\cC$ with a $c$-structure as a kernel of a localization functor from a bounded weighted 
category to a bounded heart category (of cohomological dimension smaller than the width of $\cC$). 
This allows one to eventually resolve $\cC$ by bounded weighted categories. 

\begin{rmk}\label{remark:characterization}
	We do not know (and suspect it is false) that being the kernel of a heart exact localization from a bounded weighted category to an fcd heart category characterizes $c$-categories. See \Cref{question:kerlocccat} and \Cref{question:width0}, and see \Cref{thm:weakimpliesstrong} for a characterization of such kernels.
\end{rmk}

\begin{cnstr}\label{cnstr:resolution}
Let $\cC$ be a pre-$c$-category that is $\kappa$-small for some uncountable regular cardinal $\kappa$. 
The way we set up the theory so far, using we get an embedding of $\cC$ into a bounded weighted category: 
\[
	\yo:\cC \to \Ind(\cC)^{\kappa,\mathrm{b}}
\]
and we may attempt to do the same as in Construction~\ref{cnstr:t-str-and-weight-str}. 

This embedding is natural in the sense that it induces a natural transformation\footnote{
The functor $\yo$ does not have to be weight exact, so we do not have a natural transformation of functors 
$\Cat^{\mathrm{prec},\kappa} \to \Cat^{\mathrm{prec}}$.} of functors
\[
\Cat^{\mathrm{prec},\kappa} \to \Cat^{\mathrm{st}}.
\] 
The localization sequence
\[
\cC \to \Ind(\cC)^{\kappa,\mathrm{b}} \to \Ind(\cC)^{\kappa,\mathrm{b}}/\cC =: \cC^{(1)}
\]
is also functorial in $\cC$, and the category on the right is naturally a bounded heart category in a way that makes the 
localization functor heart exact. 
\end{cnstr}

It turns out that if we restrict the above construction to $c$-categories, then $\cC^{(1)}$ is actually an fcd heart category, 
and its cohomological dimension is either smaller than the width of $\cC$, or zero.

\begin{thm}\label{thm:resolution}
Let $\kappa$ be a regular uncountable cardinal and $n$ be a natural number. Then the following statements are valid: 
\begin{enumerate}
\item
for a $\kappa$-small $c$-category $\cC$ of width $\le n$ the quotient $\cC^{(1)}$ is 
a heart structure of cohomological dimension $\le \mathrm{max}(n-1,0)$, and, in particular, a $c$-category;
\item 
for a morphism of $c$-categories $\cC \to \cD$, the functor $\cC^{(1)} \to \cD^{(1)}$ is a morphism of 
$c$-categories again. 
\end{enumerate}
\end{thm}

Before its proof we will need several intermediate lemmas.

\begin{lem}\label{lem:ccategoryequiv}
Let $\cC$ be a pre-$c$-category of size smaller than an uncountable regular cardinal $\kappa$ and let $n$ be a natural number. Equip $\Ind(\cC)^{\kappa}$ with the weight structure from \Cref{thm:generated_ws}, and equip $\Ind(\Ind(\cC)^{\kappa,b})$ with the compactly generated weight structure generated by $\Ind(\cC)^{\kappa,b}_{\leq0}$.
Then the following conditions are equivalent:
\begin{enumerate}
\item $\cC$ is a $c$-category of width $\le n$; 
\item for any $\kappa$ bigger than the size of $\cC$ the functor 
$\yo_!\yo^*: \Ind(\cC)^{\kappa,\mathrm{b}} \to \Ind(\Ind(\cC)^{\kappa,\mathrm{b}})$ sends connective objects to weight 
$-n$-connective objects;
\item any $x\in \Ind(\cC)^{w\heartsuit}$ is a filtered colimit of objects of $\cC$ such 
that for every $x_\alpha$ in the diagram all maps out of $x_{\alpha}$ eventually kill weights $\le -n-1$; 
\item for any $\kappa$ bigger than the size of $\cC$ the functor 
$\yo_!\yo^*: \Ind(\cC)^{\kappa,\mathrm{b}} \to \Ind(\Ind(\cC)^{\kappa,\mathrm{b}})$ sends weight heart to weight 
$-n$-connective objects.
\end{enumerate}
\end{lem}
\begin{proof}
This is a special case of Lemma~\ref{lem:eqcondnconnemb}. 
\end{proof}

\begin{lem}\label{lem:bcres}
Let $f:\cC \to \cD$ be a weight exact functor of $c$-categories. Then the associated square
\[
\begin{tikzcd}
\Ind(\Ind(\cC)^{\kappa,\mathrm{b}}) \arrow[r,"L_{\cC,!}"]\arrow[d, "{\bar{f}_!}"] & \Ind(\cC^{(1)})\arrow[d, "{f^{(1)}_!}"]\\
\Ind(\Ind(\cD)^{\kappa,\mathrm{b}}) \arrow[r, "L_{\cD,!}"] & \Ind(\cD^{(1)})
\end{tikzcd}
\]
is horizontally right adjointable, in the sense that $\bar{f}_! L^*_\cC \simeq L_\cD^* f^{(1)}_!$.
\end{lem}

\begin{proof}
To check right adjointability, it suffices to check it on generators of $\cC^{(1)}$, which is of the form $L_{\cC,!}(x)$
for $x \in \Ind(\cC)^{\kappa,\mathrm{b}}$. Viewing $x$ as a filtered colimit\footnote{here the quotes mean the colimit is taken in 
	$\Ind(\Ind(\cC)^{\kappa,\mathrm{b}})$ rather than $\Ind(\cC)^{\kappa,\mathrm{b}}$} 
\[
	``\colim\text{''} x_{\alpha}
\] of compact objects in $\cC$, 
$L_\cC^*L_{\cC,!}(x)$ is $\cof(``\colim\text{''} x_{\alpha} \to x)$, since this object is local and the fiber 
of the map from $x$ is generated under filtered colimits by objects of $\cC$. Thus, we have 
\[
\bar{f}_!L_\cC^*L_{\cC,!}(x) \simeq \cof(``\colim\text{''} \bar{f}_!(x_{\alpha}) \to \bar{f}_!(x)).
\]
Similarly, in $\Ind(\cD)^{\kappa,\mathrm{b}}$, the formula $\cof(``\colim\text{''} \bar{f}_!(x_{\alpha}) \to \bar{f}_!(x))$ computes 
$L_\cD^*L_{\cD,!}(\bar{f}_!(x))$, so we learn: 
\[
\bar{f}_!L_\cC^*L_{\cC,!}(x)  \simeq \cof(``\colim\text{''} f_!(x_{\alpha}) \to \bar{f}_!(x)) \simeq L_\cD^*L_{\cD,!} \bar{f}_!(x) \simeq L_\cD^* f^{(1)}_! L_!(x).
\]
\end{proof}

\subsubsection{Proof of {Theorem~\ref{thm:resolution}}}
Set $N= \mathrm{max}(n-1,0)$. 
To prove (1), it suffices to check that $\map(Lx,Ly)$ is $-N$-connective for any $x \in \Ind(\cC)^{\kappa,\mathrm{b}}_{\le 0}$ 
and $y \in \Ind(\cC)^{\kappa,\mathrm{b}}_{\ge 0}$. 
Taking the ind-completion of the localization sequence in Construction~\ref{cnstr:resolution} we get a recollement diagram 
\[
\begin{tikzcd}
\Ind(\cC) \arrow[r, "{\yo_!}",shift left] & \Ind(\Ind(\cC)^{\kappa,\mathrm{b}})\arrow[l,"{\yo^*}", shift left]\arrow[r, "{L_!}", shift left] & \Ind(\cC^{(1)}).\arrow[l, "{L^*}", shift left]
\end{tikzcd}
\]
The category in the middle admits a weight structure as described in Example~\ref{exm:from_small_to_big_ws}. 
By adjunction, to prove the claim it suffices to show that $L^*L_!y$ is $-N$-connective for 
$y \in \Ind(\cC)^{\kappa,\mathrm{b}}_{\ge 0}$.
We have a fiber sequence 
\[
\yo_!\yo^*y \to y \to L^*L_!y
\]
for any $y \in \Ind(\cC)^{\kappa,\mathrm{b}}$, so the result follows from the fact that $\yo_!\yo^*y$
is $-n$-connective for $y \in \Ind(\cC)^{\kappa,\mathrm{b}}_{\ge 0}$ (see Lemma~\ref{lem:ccategoryequiv}(1,2)). 

Now we prove (2). 
It suffices to show that the functor $f^{(1)}_!:\Ind(\cC^{(1)}) \to \Ind(\cD^{(1)})$ is weight exact. 
By construction of the weight structure on $\Ind(-)$ of a heart category (see Construction~\ref{cnstr:t-str-and-weight-str}), 
$f^{(1)}_!$ preserves the coconnective part of the weight structure. So it suffices to 
show that for $x \in \Ind(\cC^{(1)})_{w \ge 0}$ we have that $f^{(1)}_!(x)$ is also connective. 
Recall that we have a commutative diagram 
\[
\begin{tikzcd}
\Ind(\Ind(\cC)^{\kappa,\mathrm{b}}) \arrow[r,"L_{\cC,!}"]\arrow[d, "{\bar{f}_!}"] & \Ind(\cC^{(1)})\arrow[d, "{f^{(1)}_!}"]\\
\Ind(\Ind(\cD)^{\kappa,\mathrm{b}}) \arrow[r, "L_{\cD,!}"] & \Ind(\cD^{(1)})
\end{tikzcd}
\]
in $\PrL$. The right adjoint $L_{\cC}^* : \Ind(\cC^{(1)}) \to \Ind(\Ind(\cC)^{\kappa,\mathrm{b}})$ preserves weight 
connectivity, and $\bar{f}_!$ preserves weight connectivity since $\bar{f}$ is a weight exact functor of bounded weighted 
categories (see Remark~\ref{rmk:full_embedding}). So it suffices to show that the commutative diagram above is right 
adjointable, i.e. that the map 
\[
\bar{f}_! L^*_\cC \to L_\cD^* f^{(1)}_!
\]
is an equivalence. This follows from Lemma~\ref{lem:bcres}.\qed

\subsection{Exact complexes of stable categories}

Recall from \cite[Section~2.2]{christ2024complexesstableinftycategories} that a complex of stable categories is a sequence of 
stable categories $\cC_n$ with exact functors $d:\cC_n \to \cC_{n+1}$ satisfying $d^2 \simeq 0$. 

\begin{dfn}\label{dfn:exact_complex}
We say that a complex of stable categories 
\[
	\dots \to \cC_{n-1} \stackrel{d}\to \cC_n \stackrel{d}\to \cC_{n+1} \to \cdots
\]
is exact if for every $n$ the natural map 
\[
	\frac{\cC_n}{\mathrm{Im}(\cC_{n-1}\stackrel{d}\to \cC_n)} \to \mathrm{Ker}(d:\cC_n \to \cC_{n+1})
\]
becomes an equivalence after taking the idempotent completion. Here $\mathrm{Im}$ denotes the stable category generated by the image. 
\end{dfn}

\begin{exm}
A localization sequence of stable categories 
\[
0 \to \cC_0 \to \cC_1 \to \cC_2 \to 0
\]
may be viewed as an exact complex by setting $\cC_i = 0$ for $i\neq 0,1,2$. 
\tqed \end{exm}

\begin{lem}\label{lem:singlegenerator}
Let $\cC$ be a $\kappa$-small pre-$c$-category. Then there exists $\kappa'$ such that $\Ind(\cC)^{\kappa',\heart w}$ is generated by a single 
object under retracts.
\end{lem}
\begin{proof}
By assumption, the weight structure on $\Ind(\cC)$ restricts to 
$\Ind(\cC)^{\kappa,\mathrm{b}}$. Suppose that there are $\leq \kappa'$ many isomorphism classes of objects in the heart, where WLOG we may assume $\kappa' \geq \kappa$, and that $\kappa'$ has cardinality the successor of $\kappa''$. Then 
take a $\kappa''$-fold direct sum of all the isomorphism classes of these objects, which is $\kappa'$-small. Then every object of the weight heart of $\Ind(\cC)^{\kappa',\mathrm{b}}$ is a retract of this.
\end{proof}

\begin{cor}\label{cor:resolution_full}
Let $n$ be a natural number and $\kappa$ be a regular uncountable cardinal. Set $N=\mathrm{max}(n,1)$. 
For all sequences of large enough (depending on $\kappa$) regular uncountable cardinals $\kappa_0,\dots,\kappa_{N-1}$, 
there exists
\begin{enumerate} 
\item 
a sequence of functors $(-)^{(i)}:\Cat^{\mathrm{conn},\kappa}_n \to \Cat^{\mathrm{conn},\kappa_i}_{N-i}$
for all $1\le i< N$,
\item 
a functor $(-)^{(N)}:\Cat^{\mathrm{conn},\kappa}_n \to \Cat^{\mathrm{bw}}$,
\item
an exact complex of stable categories 
\[
0 \to \cC \to \Ind(\cC)^{\kappa_0,\mathrm{b}} \to \Ind(\cC^{(1)})^{\kappa_1,\mathrm{b}} \to \dots \to \Ind(\cC^{(N-1)})^{\kappa_{n-1},\mathrm{b}} \to \cC^{(N)} \to 0
\]
functorial in $\cC$.
\end{enumerate}
Furthermore, each category in the sequence except possibly $\cC$ is a bounded weighted category with a single generator.
\end{cor}
\begin{proof}
Follows via inductive application of Theorem~\ref{thm:resolution} and Lemma~\ref{lem:singlegenerator}.
\end{proof}


\section{Morphisms of \texorpdfstring{$c$-categories}{c-categories}}

\label{sec:morphisms}

\subsection{0-affine functors}

One of the main results of the previous section is the existence of a functorial resolution of a morphism of 
$c$-categories $\cC \to \cD$ by weight exact functors of the form 
\[
	\mathrm{Perf}(R) \to \mathrm{Perf}(S).
\]
for connective rings $R,S$. 
However, it may happen that this functor is not induced by a pullback along a map of rings $f:R\to S$. 
A simple example of this is taking the map of bounded weighted categories given by the inclusion of the thick subcategory in 
$\Mod(\ZZ)$ generated by $\ZZ$ into the thick subcategory generated by $\oplus_0^{\infty}\ZZ$. This is an inclusion of 
bounded weighted categories, but does not come from a map of connective rings, even after resolving once.

It is true, however, that for {\it 0-affine} maps the induced map after resolving comes from one of rings.

\begin{dfn}\label{dfn:affine}
We say that an exact functor of stable categories $f:\cC \to \cD$ is {\bf 0-affine}\footnote{the name is compatible with the notion of (weakly) 0-affine stacks in the sense of Gaitsgory \cite{gaitsgory2015sheaves}} if it generates the codomain under finite 
colimits, shifts and retracts. Equivalently, this means that the functor $f^*:\Ind(\cD) \to \Ind(\cC)$ is conservative.
\end{dfn}

\begin{lem}\label{lem:generation}
Let $(\cC,w)$ be a weighted category and let $S \subset \cC^{w\heartsuit}$ be a subset. 
\begin{enumerate}
\item Assume $w$ is bounded and $S$ generates $\cC$ by taking finite colimits, shifts and retracts. 
Then every object of $\cC^{w\heartsuit}$ is a retract of an object of $S$.
\item Assume $w$ is compactly generated and $S$ generates $\cC$ under small colimits and shifts. Then 
every object of $\cC^{w\heartsuit}$ is a retract of a small coproduct of objects of $S$. 
\end{enumerate}
\end{lem}
\begin{proof}
1. By Remark~\ref{rmk:generating_bdd_ws} there is a unique weight structure $w'$ on $\cC$ whose heart is the retract-closure 
of $S$. We have inclusions $\cC_{w'\le 0} \subset \cC_{w\le 0}$ and $\cC_{w'\ge 0} \subset \cC_{w\ge 0}$, and by 
orthogonality, we also have the opposite inclusions, so every object of the heart of the original weight structure belongs to 
the retract-closure of $S$.

2. Choose $\kappa$ such that $S$ consists of $\kappa$-compact 
objects and that the weight structure restricts to $\cC^\kappa$. Applying 
\cite[Proposition~2.2.1(1,3)]{Bondarko_2019} we obtain a weight structure $w'$ on $\cC^\kappa$ whose 
heart consists of retracts of infinite sums of objects of $S$ and whose connective and coconnective parts are the smallest 
retract-closed extension-closed $\kappa$-coproduct closed subcategories of $\cC^\kappa$ containing non-negative 
(resp., non-positive) shifts of $S$. 
By construction, we have $\cC^{\kappa}_{w'\le 0} \subset \cC^{\kappa}_{w\le 0}$ and 
$\cC^{\kappa}_{w'\ge 0} \subset \cC^{\kappa}_{w\ge 0}$ and the opposite inclusions hold by orthogonality. Hence, the 
two weight structures coincide, and every $\kappa$-small object of the heart of the original weight structure is a retract 
of an infinite coproduct of objects of $S$. 

Since the above argument holds for any large enough $\kappa$, we see that 
every object of $\cC^{w\heartsuit}$ is a retract of a small coproduct of objects of $S$. 
\end{proof}

\begin{cor}\label{cor:conn_affine}
A morphism of $c$-categories $f:\cC \to \cD$ is 0-affine if and only if 
$f_!:\Ind(\cC)^{\mathrm{b}} \to \Ind(\cD)^{\mathrm{b}}$ is 0-affine.
\end{cor}
\begin{proof} 
A functor $f$ is $0$-affine if and only if $f_!$ generates the codomain under taking small colimits and shifts. 
Now by Lemma~\ref{lem:generation}(2) $f_!$ is 0-affine if and only if every object of $\Ind(\cD)^{w\heartsuit}$ is a 
retract of a small coproduct of objects of $f(\Ind(\cC)^{\heartsuit})$. Since $f(\Ind(\cC)^{\heartsuit})$ is already closed 
under small coproducts this happens if and only if every object of $\Ind(\cD)^{w\heartsuit}$ is a retract of an object of 
$f(\Ind(\cC)^{w\heartsuit})$. This is equivalent to $\Ind(\cC)^{\mathrm{b}} \to \Ind(\cD)^{\mathrm{b}}$  being 0-affine 
by Lemma~\ref{lem:generation}(1).
\end{proof}

\begin{rmk}\label{rmk:rings_induced}
Let $f:\cC \to \cD$ be a 0-affine weight exact functor between bounded weighted categories. 
Assuming $\cC$ is generated by a single object $G$ of the weight heart, $\cD$ is also generated by $f(G)$. 
In this case \cite[Theorem~2.1.23]{nilpotentextns} implies that we have a commutative diagram 
\[
\begin{tikzcd}
\cC \arrow[d,"f"]\arrow[r, "\simeq"] & \mathrm{Perf}(\mathrm{End}_\cC(G))\arrow[d]\\
\cD \arrow[r, "\simeq"] & \mathrm{Perf}(\mathrm{End}_\cD(f(G))) 
\end{tikzcd}
\]
whose horizontal maps are equivalences and the right vertical map is the base change along the map of connective ring spectra
$\mathrm{End}(G) \to \mathrm{End}(f(G)).$
\tqed \end{rmk}

Let $\cC$ be a $c$-category of size $< \kappa$. Recall from Construction~\ref{cnstr:resolution} 
that $\Ind(\cC)^\kappa$ admits a weight structure and the Yoneda embedding functors factor through 
$\cC \to \Ind(\cC)^{\kappa, \mathrm{b}}$.

\begin{lem}\label{lem:resolution_affinness}
Let $f:\cC\to \cD$ be a morphism in $\Cat^{\mathrm{conn}}_{n,\kappa}$. Then the following are equivalent:
\begin{enumerate}
\item $f$ is a 0-affine functor;
\item $\Ind(\cC)^{\kappa,\mathrm{b}} \to \Ind(\cD)^{\kappa,\mathrm{b}}$ is a 0-affine functor.
\end{enumerate}
\end{lem}
\begin{proof}
By Corollary~\ref{cor:conn_affine} the 0-affineness of $f$ is equivalent to 0-affineness of 
$\Ind(\cC)^{\mathrm{b}} \to \Ind(\cD)^{\mathrm{b}}$. 
By Lemma~\ref{lem:generation}(2) and Proposition~\ref{prop:hearts_cogen}  
$\Ind(\cC)^{\kappa,w\heartsuit}$ and $\Ind(\cD)^{\kappa,w\heartsuit}$ generate $\Ind(\cC)^{w\heartsuit}$ and 
$\Ind(\cC)^{w\heartsuit}$ under taking all small direct sums and retracts. In particular, $f:\cC \to \cD$ is 0-affine already when 
\[
\Ind(\cC)^{\kappa,w\heartsuit} \to \Ind(\cD)^{\kappa,w\heartsuit}
\]
is essentially surjective up to retracts. So we have (2) $\Rightarrow$ (1). 

To prove the converse, assume that $\Ind(\cC)^{w\heartsuit} \to \Ind(\cD)^{w\heartsuit}$ is surjective up to retracts. 
It suffices to show that $\Ind(\cC)^{\kappa,w\heartsuit} \to \Ind(\cD)^{\kappa,w\heartsuit}$ is surjective up to retracts. 
By assumption any object $x\in \Ind(\cD)^{\kappa,w\heartsuit}$ is a retract of 
$f(\bigoplus_{\alpha\in I} y_\alpha) \simeq \bigoplus_{\alpha\in I} f(y_\alpha)$ with 
$y_\alpha \in \Ind(\cC)^{\kappa, w\heartsuit}$. This means that the identity morphism $\id_x$ decomposes as 
\[
x \stackrel{i}\to \bigoplus_{\alpha\in I} f(y_\alpha) \stackrel{p}\to x
\]
But since $x$ is $\kappa$-compact, $i$ factors through a direct sum $\bigoplus_{\alpha\in I'} f(y_\alpha)$ with 
$I' \subset I$ of size less than $\kappa$. So we also have a decomposition 
\[
x \stackrel{i'}\to \bigoplus_{\alpha\in I'} f(y_\alpha) \hookrightarrow \bigoplus_{\alpha\in I} f(y_\alpha) \stackrel{p}\to x
\]
and $x$ is a retract of $\bigoplus_{\alpha\in I'} f(y_\alpha) \simeq f(\bigoplus_{\alpha\in I'} y_\alpha)$ which is the 
image of an object in $\Ind(\cC)^{\kappa,w\heartsuit}$. 
\end{proof}

\begin{cor}\label{cor:resolution_affine}
Let $f:\cC\to \cD$ be a 0-affine morphism in $\Cat^{\mathrm{conn}}_{n,\kappa}$. 
Then the functor $\cC^{(1)} \to \cD^{(1)}$ is also 0-affine. In particular, the resolution of 
Corollary~\ref{cor:resolution_full} is functorial in 0-affine maps. 
\end{cor}
\begin{proof}
By Lemma~\ref{lem:resolution_affinness} $\Ind(\cC)^{\kappa,\mathrm{b}} \to \Ind(\cD)^{\kappa,\mathrm{b}}$ is 0-affine. 
Hence the functor $\Ind(\cC)^{\kappa,\mathrm{b}} \to \cD^{(1)}$ is 0-affine and since 
$\Ind(\cC)^{\kappa,\mathrm{b}} \to \cC^{(1)}$ is surjective, so does the functor in question.
\end{proof}

\subsection{Nilpotent extensions and connective maps}
Recall that a map of connective rings $R \to S$ is a nilpotent extension if it is surjective on $\pi_0$ with nilpotent kernel. This notion naturally generalizes to any stably monoidal category endowed with a compatible t-structure, and in what follows we will introduce these notions in that generality.

\begin{dfn}\label{dfn:compatible_tstr}
Let $\cE$ be a stably monoidal category. We say that a t-structure $t$ on $\cE$ is compatible with the monoidal 
structure if the monoidal unit is connective, and the tensor product of connective objects is connective.  
\end{dfn}

\begin{dfn}\label{dfn:tstr_notions}
Assume $\cE$ is a stably monoidal category endowed with a compatible t-structure. Let $f:R \to R'$ be a map of connective algebras in $\cE$. 
\begin{enumerate}
\item 
We say that $f$ is a {\it $k$-connective map} if $\fib(R \to R')$ is $k$-connective in the t-structure.
\item
We say that $f$ is a {\it nilpotent extension} of order $\leq m$ if it is $0$-connective and the composite
\[
\fib(f)^{\otimes_Rm} \to R^{\otimes_Rm} \stackrel{\mathrm{mult}}\longrightarrow R \to \pi_0^{\heart}R
\]
is null. 
\end{enumerate}
\end{dfn}

Now let $\cC$ be a pre-$c$-category. The composition of functors endows the stable category  
$\End_{\PrL}(\Ind(\cC))$ with a presentably monoidal structure. Moreover, it also admits a natural t-structure, as follows.

\begin{cnstr}\label{cnstr:associated_tstr}
Consider the subcategory of $\End_{\Pr^L}(\Ind(\cC))$ functors that send weight connective objects to weight connective 
objects. This is a presentable subcategory closed under colimits and extensions, so by \cite[Proposition~1.4.4.11(1)]{HA} it extends to a t-structure which we call 
the \textit{associated $t$-structure} on $\End_{\Pr^L}\Ind(\cC)$. 
\end{cnstr}

\begin{rmk}
The associated t-structure is compatible with the monoidal structure in the sense of Definition~\ref{dfn:compatible_tstr}. 
\tqed \end{rmk}

To any functor $f:\cC\to \cD$ that induces a weight exact functor $f_!:\Ind(\cC) \to \Ind(\cD)$ one can associate the 
endofunctor $f^*f_! \in \End_{\PrL}(\Ind(\cC))$. It is a t-connective algebra object because both $f^*$ and $f_!$ 
preserve weight connectivity. 

\begin{dfn}\label{dfn:nilp_conn}
We say that a functor $f:\cC\to \cD$ of pre-$c$-categories that induces a weight exact functor 
$f_!:\Ind(\cC) \to \Ind(\cD)$ is a {\it $k$-connective} (resp. a {\it nilpotent extension}) if it is 0-affine and the associated algebra 
$f_!f^* \in \Alg(\End_{\Pr^L}(\Ind(\cC)))$ is $k$-connective (resp. a nilpotent extension). 
\end{dfn}

\begin{prop}\label{prop:weak_nilpotence}
Let $f:\cC \to \cD$ be an $0$-affine functor of pre-$c$-categories that induces a weight exact functor 
$f_!:\Ind(\cC) \to \Ind(\cD)$. 
\begin{enumerate}
\item
$f$ is $k$-connective for $k\in \NN$ if and only if for any $x,y \in \Ind(\cC)^{w\heartsuit}$ the map 
\[
\Map_{\Ind(\cC)^{w\heartsuit}}(x,y) \to \Map_{\Ind(\cD)^{w\heartsuit}}(f(x),f(y))
\]
is $k$-connective.
\item 
If $f$ is a nilpotent extension, then 
\[
\pi_0\Map_{\Ind(\cC)^{w\heartsuit}}(x,y) \to \pi_0\Map_{\Ind(\cD)^{w\heartsuit}}(f(x),f(y))
\]
is surjective and there exists $m\in \NN$ such that for any sequence of composable morphisms $g_1, \dots, g_m$ in $\cA$ with 
$f(g_i) \simeq 0$ we have that the composite $g_m\circ \dots \circ g_1 \simeq 0$. 
\end{enumerate}
\end{prop}
\begin{proof}
Set $I:=\fib(\id_{\Ind(\cC)} \to f^*f_!)$. 

First we show part (1). By adjunction, $k$-connectivity is equivalent to saying that 
for any $x, y\in \Ind(\cC)^{w\heartsuit}$ the map 
\[
\map_{\Ind(\cC)}(y,x) \to \map_{\Ind(\cC)}(y,f^*f_!(x))
\]
is $k$-connective. By Proposition~\ref{prop:hearts_gen} this is also equivalent to weight $k$-connectivity $I(x)$ for any 
$x\in \Ind(\cC)^{w\heartsuit}$. Since $I$ 
commutes with colimits and since the colimit closure of $\Ind(\cC)^{w\heartsuit}$ in $\Ind(\cC)$ is $\Ind(\cC)_{w\ge 0}$ 
(Proposition~\ref{prop:hearts_gen}), the last claim is also 
equivalent to $I$ sending weight connective objects to weight $k$-connective ones.

To prove part (2), observe that the triviality of the composite $I^{\circ m} \to \id_{\Ind(\cC)} \to \pi_0(\id_{\Ind(\cC)})$ 
implies that $I^{\circ m}\to \id_{\Ind(\cC)}$ factors through a 1-connective object $F$. Now, take any composable sequence of  
morphisms between objects of the heart 
\[
x_0 \stackrel{g_1}\to x_1 \stackrel{g_2}\to \dots \stackrel{g_m}\to x_m
\]
$g_i$ with $f_!(g_i) \simeq 0$, The composite $g_m\circ \dots \circ g_1$ lifts to a map
\[
x_0 \to I^{\circ m} x_m,
\]
but the composite $x_0 \to I^{\circ m} x_m \to x_m$ is then zero, since it factors through $Fx_m \in \Ind(\cC)_{w\ge 1}$.
\end{proof}

\begin{exm}\label{exm:iso_homotopy}
It follows from \Cref{prop:weak_nilpotence} that the $0$-connectivity condition amounts to surjectivity of the map 
$\pi_0\Map_\cA(x,y) \to \pi_0\Map_\cB(f(x),f(y))$ for all $x,y\in \cA$. 
For $k$-connective $f$ and $k>0$ this map is an isomorphism. 
\tqed \end{exm}
\begin{exm}
A map of connective ring spectra $R \to S$ is $k$-connective if and only if the induced map 
$\mathrm{Perf}(R) \to \mathrm{Perf}(S)$ is $k$-connective.
\tqed \end{exm}

\begin{dfn}\label{dfn:weak_nilpotence}
We say that a functor $f:\cC\to \cD$ of pre-$c$-categories that induces a weight exact functor 
$f_!:\Ind(\cC) \to \Ind(\cD)$ is a {\it weak nilpotent extension} (of order $\le m$) if it is $0$-connective 
and there exists $m\in \NN$ such that for any sequence of composable morphisms 
\[
x_0 \stackrel{g_1}\to x_1 \stackrel{g_2}\to \dots \stackrel{g_m}\to x_m
\] 
in $\cC$ with $x_0 \in \cC_{w\le 0}, x_m \in \cC_{w\ge 0}$ and $f(g_i) \simeq 0$ we have that the composite $g_m\circ \dots \circ g_1 \simeq 0$.
\end{dfn}

\begin{rmk}
Proposition~\ref{prop:weak_nilpotence}(2) implies that a nilpotent extension of pre-$c$-categories $f:\cC \to \cD$ 
is a weak nilpotent extension. We do not know whether the converse holds in general, however, it is true if both $\cC$ 
and $\cD$ are bounded weighted categories: in this case we have an equivalence 
$\Ind(\cC) \simeq \Fun_{\Sigma}(\cC^{\heartsuit, \mathrm{op}}, \Sp),$ and triviality of the map 
\[
I^{\otimes m} \to \id_{\Ind(\cC)} \to \pi_0(\id_{\Ind(\cC)})
\]
can be checked objectwise. In this case, the notion of nilpotent extension here agrees with that of \cite[Definition 5.1.1]{nilpotentextns}.
\tqed \end{rmk}
\begin{rmk}\label{rmk:nil_idempotents}
For a weak nilpotent extension the maps of rings
\[
\pi_0\mathrm{End}_{\cC}(x) \to \pi_0\mathrm{End}_{\cD}(f(x))
\]
are nilpotent extensions of rings. Thus, idempotents lift along these maps and in particular, a nilpotent extension 
between idempotent complete categories is essentially surjective on the nose.

Furthermore, the converse is true: if $f$ is essentially surjective up to retracts and for any 
$x \in \cC^\heartsuit$ the map on 
endomorphism rings as above is a nilpotent extension, then $f$ is a weak nilpotent extension. In other words, 
it suffices to check condition (2) in the case when $x=y$ and condition (3) when all domains and codomains of $g_i$ 
are equal. To see this, observe that 
$\pi_0\Map_{\cC^\heartsuit}(x,y) \to \pi_0\Map_{\cD^\heartsuit}(f(x),f(y))$ is a retract of 
$\pi_0\Map_{\cC^\heartsuit}(x\oplus y,x\oplus y) \to \pi_0\Map_{\cD^\heartsuit}(f(x\oplus y),f(x\oplus y))$ and that the morphisms $g_i$ and their 
composite can be viewed in the endomorphism ring of the coproduct of all the domains and codomains of $g_i$'s.
\tqed \end{rmk}

\begin{lem}\label{lem:example_nilp_weight0}
Let $f:\cC\to \cD$ be a weak nilpotent extension of order $\le n$ of pre-$c$-categories. 
Set $I:=\fib(\id_{\Ind(\cC)} \to f^*f_!)$. 
Then for any $x\in \cC_{w\ge 0}$ the canonical map $I^{\circ n}x \to x$ kills weight $0$. 
Furthermore, for any $k\in \NN$ the map $I^{\circ nk} x \to x$ kills weights $\le k$. 
\end{lem}
\begin{proof}
To show that $I^{\circ n}x \to x$ kills weight $0$, by Lemma~\ref{lem:weight0_morphism}, it suffices to 
construct a weight decomposition of $I^{\circ n}x$ such that the composite $w_{\le 0} I^{\circ n}x \to I^{\circ n}x \to x$ 
is zero. 
Fix any weight decomposition of $x$. By induction, we construct $x_0=w_{\le 0}x$ and $x_i\to Ix_{i-1}$ such that 
\[
x_i \to Ix_{i-1} \to \cof_i
\]
is a weight decomposition. In particular, $x_i$ belongs to the weight heart and $\cof_i \in \Ind(\cC)_{w\ge 1}$. 
The cofiber of the map $I^{\circ n}x_0 \to I^{\circ n}x$ also belongs to $\Ind(\cC)_{w\ge 1}$ since $I^{\circ n}$ 
preserves weight connectivity (Proposition~\ref{prop:weak_nilpotence}). 
Combining these two facts, we see that the cofiber of the composition
\[
\beta:x_n \to Ix_{n-1} \dots \to I^{\circ n}x_0 \to I^{\circ n}x
\]
is also in $\Ind(\cC)_{w\ge 1}$, so it gives a weight decomposition for $I^{\circ n}x$. Now the composite
$
x_n \stackrel{\beta}\to I^{\circ n}x \to x
$
is equivalent to the composite 
$
x_n \to x_{n-1} \to \dots \to x_0\to x.
$
By construction, each of the maps $x_i \to x_{i-1}$ is null after applying $f_!$, so by the condition in 
Definition~\ref{dfn:weak_nilpotence}, the whole composite is zero. 

We prove the furthermore part using induction on $k$. Assuming the claim is true for $k-1$, the map
$I^{\circ n(k-1)}x \to x$ decomposes as follows:
\[
I^{\circ n(k-1)}x \to w_{\ge k}(I^{\circ n(k-1)}x) \to x.
\]
By naturality, we have a commutative diagram
\[
\begin{tikzcd}
I^{\circ nk}x \arrow[r]\arrow[d] & I^{\circ n}(w_{\ge k}(I^{\circ n(k-1)}x))\arrow[d]\\
I^{\circ n(k-1)}x \arrow[r] & w_{\ge k}(I^{\circ n(k-1)}x)
\end{tikzcd}
\]
in which the right vertical map is the $k$-th suspension of a morphism killing weights $\le 0$. Hence, it kills weights $\le k$. Thus, the composite 
$I^{\circ nk} x \to w_{\ge k}(I^{\circ n(k-1)}x) \to x$ also kills weights $\le k$.
\end{proof}

\begin{cor}\label{cor:higher_nilpotence}
Fix the notations of Lemma~\ref{lem:example_nilp_weight0}. For any $k \in \NN$ and any sequence of composable morphisms 
\[
x_0 \stackrel{g_1}\to x_1 \stackrel{g_2}\to \dots \stackrel{g_{kn}}\to x_{kn}
\]
such that $f_!(g_i) \simeq 0$, $x_{kn}$ is connective and $x_{0}$ belongs to $\Ind(\cC)_{w\le k}$, 
the composition $g_{kn}\circ \dots \circ g_{1}$ is trivial.
\end{cor}
\begin{proof}
By assumption, each $g_i$ decomposes as
\[
x_{i-1} \stackrel{\bar{g}_1}\to I(x_i) \to x_i,
\]
so the total composition is equivalent to $x_0 \to I^{\circ kn}x_{kn} \to x_{kn}$. This kills weights $\le k$ by 
the furthermore part of Lemma~\ref{lem:example_nilp_weight0}. Since $x_0\in \Ind(\cC)_{w\le k}$, the composite is zero.
\end{proof}

\begin{rmk}\label{rmk:nilpotent_kernel}
\Cref{cor:higher_nilpotence} shows that for a nilpotent extension $f:\cC \to \cD$ of pre-$c$-categories and any 
$x\in \Ind(\cC)_{w\ge 0} \cap \Ind(\cC)_{w\le k}$ the map 
\[
\pi_0\mathrm{End}_{\Ind(\cC)}(x) \to \pi_0\mathrm{End}_{\Ind(\cD)}(f_!x)
\]
has nilpotent kernel. Note that this may not be true without the bounds on weights. 
\tqed \end{rmk}

\begin{thm}\label{thm:resolution_connective}
Let $f:\cC \to \cD$ be a $k$-connective functor of $c$-categories of size $< \kappa$. 
Then the induced functors 
\[
\bar{f}:\Ind(\cC)^{\kappa,\mathrm{b}} \to \Ind(\cD)^{\kappa,\mathrm{b}}
\]
and 
\[
f^{(1)}:\cC^{(1)} \to \cD^{(1)}
\]
are also $k$-connective. 
\end{thm}
\begin{proof}
We already checked in Lemma~\ref{lem:resolution_affinness} and Corollary~\ref{cor:resolution_affine} that these maps are 
0-affine. By Proposition~\ref{prop:weak_nilpotence}(1) combined with Lemma~\ref{lem:generation} $\bar{f}$ is $k$-connective. 

To show the claim for $f^{(1)}$ it suffices to check the equivalent condition of Proposition~\ref{prop:weak_nilpotence}(1). 
Recall that we have a commutative diagram 
\[
\begin{tikzcd}
\Ind(\Ind(\cC)^{\kappa,\mathrm{b}}) \arrow[r,"L_{\cC,!}"]\arrow[d, "{\bar{f}_!}"] & \Ind(\cC^{(1)})\arrow[d, "{f^{(1)}_!}"]\\
\Ind(\Ind(\cD)^{\kappa,\mathrm{b}}) \arrow[r, "L_{\cD,!}"] & \Ind(\cD^{(1)})
\end{tikzcd}
\]
in $\PrL$ which is also horizontally right adjointable (see Lemma~\ref{lem:bcres}). 
The fact that $\fib(x \to f^{(1),*}f^{(1)}_!x)$ belongs to $\Ind(\cC^{(1)})_{w\ge k}$ for any $x\in \Ind(\cC^{(1)})_{w\ge 0}$ 
follows from the fact that 
\[
L^*_\cC\fib(x \to f^{(1),*}f^{(1)}_!x) \simeq \fib(L^*_\cC x \to L^*_\cC f^{(1),*}f^{(1)}_!x) \simeq \fib(L^*_\cC x \to \bar{f}^{*}\bar{f}_!L^*_\cD x)
\]
belongs to $\Ind(\Ind(\cD)^{\kappa,\mathrm{b}})_{w\ge k}$ and the fact that $L^*_\cC$ detects weight connectivity.
\end{proof}

\begin{thm}\label{thm:nilpotent_resolution}
Let $f:\cC \to \cD$ be a weak nilpotent extension of rank $\leq k$ of $c$-categories of width $\le n$ of size 
$< \kappa$. Set $N = \mathrm{max}(n,1)$. Then the induced functors 
\[
\bar{f}:\Ind(\cC)^{\kappa,\mathrm{b}} \to \Ind(\cD)^{\kappa,\mathrm{b}}
\]
and 
\[
f^{(1)}:\cC^{(1)} \to \cD^{(1)}
\]
are also weak nilpotent extensions of ranks $\leq k$ and $\leq kn$, respectively.
\end{thm}
\begin{proof}
We already checked in Theorem~\ref{thm:resolution_connective} that these maps are 
$0$-connective. In this case $\bar{f}$ is a weak nilpotent extension of the same degree 
as $f$. It suffices to show the claim for $f^{(1)}$. Recall that we have a commutative diagram 
\[
\begin{tikzcd}
\Ind(\Ind(\cC)^{\kappa,\mathrm{b}}) \arrow[r,"L_{\cC,!}"]\arrow[d, "{\bar{f}_!}"] & \Ind(\cC^{(1)})\arrow[d, "{f^{(1)}_!}"]\\
\Ind(\Ind(\cD)^{\kappa,\mathrm{b}}) \arrow[r, "L_{\cD,!}"] & \Ind(\cD^{(1)})
\end{tikzcd}
\]
in $\PrL$ which is also horizontally right adjointable (see Lemma~\ref{lem:bcres}). 
Fix a sequence of composable morphisms 
\[
x_0 \stackrel{g_1}\to x_1 \stackrel{g_2}\to \dots \stackrel{g_{kn}}\to x_{kn}
\]
between objects of $\Ind(\cC^{(1)})^{w\heartsuit}$ such that $f^{(1)}_!(g_i) \simeq 0$. We want to show that the composite 
is $0$. We first prove the claim when $x_0$ is a retract of an object of the form $L_{\cC,!}(\bar{x}_0)$ where 
$\bar{x}_0 \in \Ind(\Ind(\cC)^{\kappa,\mathrm{b}})_{w\le k}$. 
By adjunction the composite is zero if and only if the composite 
\[
\bar{x}_0 \stackrel{\bar{g}_0}\to L^*_\cC x_1 \stackrel{L^*_{\cC}g_2}\to \dots \stackrel{L^*_{\cC}g_{kn}}\to L^*_\cC x_{kn}
\]
is zero. The right adjointability of the square 
\[
\begin{tikzcd}
\Ind(\Ind(\cC)^{\kappa,\mathrm{b}}) \arrow[r,"L_{\cC,!}"]\arrow[d, "{\bar{f}_!}"] & \Ind(\cC^{(1)})\arrow[d, "{f^{(1)}_!}"]\\
\Ind(\Ind(\cD)^{\kappa,\mathrm{b}}) \arrow[r, "L_{\cD,!}"] & \Ind(\cD^{(1)})
\end{tikzcd}
\]
ensures that the maps $\bar{f}_!(L^*_{\cC}g_i)\simeq L^*_{\cC}(f^{(1)}_!g_i) \simeq 0$. 
Furthermore, the map $\bar{f}_!(\bar{g}_0):\bar{f}_!\bar{x}_0 \to \bar{f}_!L^*_\cC x_1$ is zero: under the 
$(\bar{f}_!, \bar{f}^*)$-adjunction 
it corresponds to the map
\[
\bar{x}_0 \to \bar{f}^*\bar{f}_!L^*_{\cC} x_1 \simeq \bar{f}^*L_{\cD}^*f^{(1)}_! x_1 \simeq L_{\cC}^*f^{(1),*}f^{(1)}_! x_1
\]
which corresponds under the $(f^{(1)}_!L_{\cC,!}, L_{\cC}^*f^{(1),*})$-adjunction to the map $f^{(1)}_!(g_1):f^{(1)}_!L_{\cC,!}(\bar{x}_0) \to f^{(1)}_!(x_1)$ which is zero by assumption. 
Now the composite is zero by Corollary~\ref{cor:higher_nilpotence} and by assumption on $\bar{x}_0$.

Now we show that $x_0$ is also a retract of an object of the form $L_{\cC,!}(\bar{x}_0)$ for $\bar{x}_0 \in \Ind(\Ind(\cC)^{\kappa,\mathrm{b}})_{w\le k}$. 
Take any $\bar{x}'_0 \in \Ind(\Ind(\cC)^{\kappa,\mathrm{b}})$ with $L_{\cC,!}(\bar{x}'_0) \simeq x_0$. 
Consider a weight decomposition 
\[
w_{\le n} \bar{x}'_0 \to \bar{x}'_0 \to w_{\ge n+1} \bar{x}'_0.
\] 
Observe that $L_{\cC,!}(w_{\ge n+1} \bar{x}'_0)$ is 1-connective because $L_{\cC}^*L_{\cC,!}$ shifts weight connectivity by 
$-n$ and $L_{\cC}^*$ reflects weight connectivity. 
Since $L_{\cC,!}\bar{x}'_0$ is in the weight heart, the weight decomposition splits after applying $L_{\cC,!}$, 
and so $x_0$ is in fact a retract of $L_{\cC,!}(w_{\le n}\bar{x}'_0)$, so we may set $\bar{x}_0 = w_{\le n}\bar{x}'_0$.
\end{proof}

\begin{cor}\label{cor:resolution_nilpotent}
Let $f:\cC \to \cD$ be a $k$-connective map (resp. a weak nilpotent extension) of $c$-categories. 
Then for any choice of $\kappa_i$, the induced functors on the resolution of Corollary~\ref{cor:resolution_full} are 
$k$-connective maps (resp. weak nilpotent extensions) of bounded weighted categories. 
\end{cor}

\begin{rmk}\label{rmk:degree of nilpotence}
Consider a weak nilpotent extension $f:\cC\to \cD$ of rank $\leq k$ of $c$-categories of width $\le n$. 
The proof of Theorem~\ref{thm:nilpotent_resolution} shows that the ranks of nilpotence of 
functors in Corollary~\ref{cor:resolution_full} are bounded by $\frac{n!}{i!}k$ in degree $i$. 
In particular, the nilpotent extension of rings in the proof of Theorem~\ref{thm:truncating_invariants}(2) below is of degree 
$n!k$.
\tqed \end{rmk}

\subsection{Transfer and descent}

The reason for defining and studying the notion of a $c$-structure is the abundance of examples. 
In this subsection we discuss two ways of producing such examples:
\begin{enumerate} 
\item 
one can transfer (pushforward) $c$-structures along 0-affine functors of stable categories;
\item 
one can descend $c$-structures along the appropriately defined descendable morphisms.
\end{enumerate}

\begin{prop}[Transfer of connectivity]\label{prop:transfer_of_c}
Suppose that $\cC$ is a stable category such that $\Ind(\cC)$ has a compactly generated hypercomplete 
weight structure, and $f:\cC \to \cD$ is a functor such that 
\begin{enumerate}
\item $f$ is 0-affine,
\item $f^*f_!$ preserves weight connective objects. 
\end{enumerate}
Then there exists a unique hypercomplete compactly generated weight structure on $\cD$ such that $f_!$ is weight exact. 

Moreover, if the weight structure on $\Ind(\cC)$ is a $c$-structure of width $\le n$, then so is the induced weight 
structure on $\Ind(\cD)$.
\end{prop}
\begin{proof}
Consider a set of compact generators $\cC_{\le 0} \subset \cC$ for the weight structure on $\Ind(\cC)$. $f(\cC_{\leq0})$ generates $\cD$ because $f$ is $0$-affine, so by Theorem~\ref{thm:generated_ws} $f(\cC_{\le 0}) \subset \cD$ generates a unique hypercomplete weight structure on 
$\Ind(\cD)$. We then automatically have that $f_!$ preserves weight coconnective objects. 
By construction an object $x\in \Ind(\cD)$ belongs to $\Ind(\cD)_{w\ge 0}$ if and only if 
\[
\map(f_!y,x) \simeq \map(y, f^*x) 
\]
is connective for any $y\in \cC_{\le 0}$. In other words, $x\in \Ind(\cD)_{w\ge 0}$ if and only if 
$f^*x \in \Ind(\cC)_{w\ge 0}$. The second assumption then implies that $f_!$ preserves weight connective objects. 

Now we show the moreover part of the claim. By Proposition~\ref{prop:hearts_gen} $\Ind(\cC)$ is generated by objects of the 
heart, so the image of the heart also generates the heart $\Ind(\cD)^{w\heartsuit}$ under taking retracts. 
Hence, any object of $\Ind(\cD)^{w\heartsuit}$ is a filtered colimit of objects $f_!(x_{\alpha})$ for 
$x_{\alpha}\in \cC$ such that for each $x_{\alpha}$ all maps out of $x_{\alpha}$ in the diagram eventually 
kill weights $\le -n$. Since weight exact functors preserve morphisms killing weights, we get a $c$-structure of width $\le n$ on 
$\cD$.
\end{proof}


\begin{cor}[Transfer of heart structure]\label{thm:transferofconn}
Let $\cC$ be a bounded heart category of cohomological dimension $n$ and 
let $f:\cC \to \cD$ be a functor satisfying the same assumptions as in Proposition~\ref{prop:transfer_of_c}. 

Assume also that $\cD$ is idempotent complete. 
Then there is a heart structure on $\cD$ such that the connective and coconnective objects are 
generated by the image of those in $\cC$, and such that the map $\cC \to \cD$ is weight exact.
\end{cor}
\begin{proof}
Set $\cE$ to be the retract-closure of the extension-closure of $f(\cC^\heart)$ in $\cD$. 
Our goal is to show that $\cE$ is the heart of a heart structure on $\cD$. 
Note that if it exists, it is automatically cohomological dimension $\leq n$. 
Indeed, since $\cC$ is cohomological dimension $\leq n$, this implies that any $x\in \cC^\heart$ belongs to $\Ind(\cC)_{w\ge -n}$. 
Since $f^*f$ preserves weight connectivity, for any $x,y \in \cC^\heart$ we have that 
\[
\map(fx,fy) = \map(x, f^*fy)
\]
is $-n$-connective. Now observe that the subcategory of pairs $(s,t) \in \cD^{\mathrm{op}} \times \cD$ for which 
$\map(s,t)$ is $-n$-connective is closed under extensions and retracts, so it in particular contains 
$\cE^{\mathrm{op}} \times \cE$.

Consider the commutative square
\[
\begin{tikzcd}
\cD^\mathrm{b}(\cC^\heart) \arrow[d,"g"]\arrow[r, "\simeq"] & \cC\arrow[d,"f"]\\
\cD^\mathrm{b}(\cE) \arrow[r, "u"] & \cD.
\end{tikzcd}
\]
By \cite[Proposition~2.28]{winges2025presentable} it suffices to show that $u$ is an equivalence up to idempotent completion, 
so it remains to prove that 
\[
u_!:\Ind(\cD^\mathrm{b}(\cE)) \to \Ind(\cD)
\] 
is an equivalence. The images of both $g$ and $f$ generate the corresponding codomains under 
finite colimits, loops and retracts. 
Hence the images of $g_!$ and $f_!$ generate $\Ind(\cD^\mathrm{b}(\cE))$ and $\Ind(\cD)$ 
under all colimits and loops. 
Since $\cC$ is $-n$-connective, every object of $\cC^\heart$ is a finite extension of objects of the weight-heart 
in $\Ind(\cC)$. Hence both $\Ind(\cD^\mathrm{b}(\cE))$ and $\Ind(\cD)$ are generated by the images of the objects of 
the weight-heart $\Ind(\cC)^{w\heart}$. It suffices to show that the map 
\[
u_{x,y}:\map(g(x),g(y)) \to \map(f(x),f(y))
\]
is an equivalence for any $x,y \in \Ind(\cC)^{w\heart}$.

By \Cref{exm:minusone} (applied to the case $\cA = \cD$, $\cE=\cE$) the map 
\[
\map(g(x),g(y)) \to \map(f(x), f(y))
\]
is an equivalence on $-1$-connective covers for $x,y \in \cC^\heartsuit$. 
This is also true for $x,y \in \Ind(\cC^\heartsuit)$ 
since $\tau_{\ge -1}:\Sp \to \Sp_{\ge -1}$ preserves limits and filtered colimits. 
By Lemma~\ref{lem:wtheartforheartstr} we have an inclusion of subcategories 
$\Ind(\cC)^{w\heart} \subset \Ind(\cC^{\heart})$, so by the above observation 
$u_{x,y}$ is an equivalence on $-1$-connective covers for $x,y \in \Ind(\cC)^{w\heart}$. 
Note that in this case 
\[
\map(f(x), f(y)) \simeq \map(x, f^*f(y))
\]
is connective, so it suffices to show that 
$\map(g(x), g(y))$ is connective. To see this we show that $g(y)$ is in the heart of the canonical weight structure 
on $\Ind(\cD^\mathrm{b}(\cE))$. The fact that $g(y)$ belongs to 
$\Ind(\cD^\mathrm{b}(\cE))_{w\le 0}$ is automatic, so it suffices to prove weight connectivity. 

Using the above observation again, observe that for any $x\in \cC^\heartsuit$ we have 
\[
\pi_{-1}\map(g(x), g(y)) \simeq \pi_{-1}\map(g(x), g(y))\simeq \pi_{-1}\map(f(x), f(y)) \simeq 0.
\]
Since the subcategory of $\cE$ spanned by objects $x' \in \cE$ for which $\pi_{-1}\map(x', g(y)) \simeq 0$ is 
closed under extensions and 
retracts, and since $\cE$ is generated by the image of $\cC^{\heartsuit}$ under extensions and retracts, we see that 
$\pi_{-1}\map(x', g(y)) \simeq 0$ for any $x' \in \cE$. 
By Example~\ref{exm:special_subcategory} we see that for any $n>1$, any $x' \in \cE$ and
any map $x' \to \Sigma^n g(y)$ factors as $x \to \Sigma^{n-1} y' \stackrel{\alpha}\to \Sigma^n g(y)$ for some 
$y' \in \cE$. Since $\pi_0\map(\Sigma^{n-1}y', \Sigma^ng(y)) \simeq \pi_{-1}\map(y', g(y)) \simeq 0$, we obtain that 
$\alpha$ as well as the 
map $x' \to \Sigma^n g(y)$ are trivial and so $g(y)$ is indeed weight connective in the canonical weight structure on 
$\Ind(\cD^\mathrm{b}(\cE))$. 
\end{proof}

\begin{ntn}
Let $f:\cC \to \cD$ be a functor between small stable categories, and consider the induced adjunction 
\[
f_!:\Ind(\cC) \to \Ind(\cD):f^*
\]
 We use $M_f \in \Alg(\End_{\Pr^L}\Ind(\cC))$ to denote the monad associated to this adjunction.
\end{ntn}

For any $x\in \cC$ we then have a comparison map
\[
\begin{tikzcd}
x \to \lim\limits_{\Delta} (M_fx \arrow[r,shift left]\arrow[r,shift right] & M_f^{\otimes 2}x \arrow[r,shift left=2]\arrow[r]\arrow[r,shift right = 2] &\dots)
\end{tikzcd}
\]
in $\Ind(\cC)$. We denote the cosimplicial object giving this limit by $M_f^{\otimes \bullet +1}(x)$, so that the right hand side can be written as $\Tot(M_f^{\otimes \bullet +1}(x))$. Consider the full subcategories 
$\Delta_{\le n} \subset \Delta$ of nonempty totally ordered sets of size $\le n$. In what follows we are interested in studying the 
partial totalizations
\[
\Tot^n(M_f^{\otimes \bullet +1}(x)):=\lim_{\bullet\in \Delta_{\le n}} M_f^{\otimes \bullet +1}(x)_i.
\]

\begin{prop}\label{prop:descendability}
Let $f:\cC \to \cD$ be a functor between small stable categories. Let $I_f \to \id_{\Ind(\cC)}$ be the fiber of the unit map $\id_{\Ind(\cC)} \to M_f$. For $m\geq0$, the following are equivalent: 
\begin{enumerate}
\item[(1)] There is an endofunctor $X$ of $\Ind(\cC)$ such that $\id_{\Ind(\cC)}$ is a retract of $X$, and such that $X$ admits a filtration $0 = X_0 \to \dots \to X_{m+1} = X$ with associated graded $X_i/X_{i-1} \cong f^*(Y_i)$ for some functors $Y_i:\Ind(\cC) \to \Ind(\cD)$.
\item[(2)] For any $x\in \Ind(\cC)$ the canonical map $x \to \Tot^m(M_f^{\otimes \bullet +1}(x))$ admits a functorial splitting.
\item[(3)] The map $I_f^{\otimes m+1} \to \id_C$ is null.
\end{enumerate}
\end{prop}
\begin{proof}
	We first note that the tower of endofunctors $I_f^{\otimes \bullet+1}$ with transition maps given by the inclusion $I_f^{\otimes n}\otimes(I_f \to \id_C)$ can be identified with the fiber of the map\footnote{We use $M_f^{\otimes \bullet+1}$ to denote the cosimplicial object associated to the algebra $M_f$, with degeneracies given by the multiplication. $\Tot^n$ denotes the $n$th partial totalization.} $$x \to \Tot^{\bullet}(M_f^{\otimes \bullet+1})$$ by the proof of \cite[Proposition 2.14]{mathew2017nilpotence}\footnote{The cited proposition is for symmetric monoidal categories, but works equally well in the monoidal category of endofunctors of $\Ind(C)$, by replacing sets with linearly ordered sets in appropriate places.}. This shows that (3) is equivalent to (2). It also shows that (2) implies (1) since $$\Tot^m(M_f^{\otimes \bullet +1}) \to \Tot^{m-1}(M_f^{\otimes \bullet +1})\dots \to \Tot^0(M_f^{\otimes \bullet +1}) = M_f \to 0$$ gives a finite filtration on $\Tot^m(M_f^{\otimes \bullet +1})$ whose associated graded is $I^{\otimes n}_f\otimes M_f$, which is in the image of $f^*$ since $M_f$ is.
	
	We now prove (1) implies (3). Assume (1), so that it is enough to show that the composite $I_f^{\otimes m+1} \to \id \to X$ is null. By assumption, $X$ has an $m+1$-step filtration $0 = X_0 \to X_1 \to \dots \to X_{m+1} = X$, with $X_i/X_{i-1} = f^*(Y_i)$. By induction on $0\leq k\leq m+1$, we show that the composite $I^{\otimes k}_f \to \id \to X$ factors through $X_{m-k}$. For the inductive step, we assume the result for $k-1$. Producing the factorization is then equivalent to producing the map $g_k$ in the diagram making the square commute. Indeed, by taking vertical fibers, one then produces the desired factorization $h_k$.
	
	\[\begin{tikzcd}
		{I_f^{\otimes k-1}\otimes M_f} & {f^*(Y_{m-k+1})} \\
		{I_f^{\otimes k-1}} & {X_{m-k+1}} \\
		{I_f^{\otimes k}} & {X_{m-k}}
		\arrow["g_k",dashed, from=1-1, to=1-2]
		\arrow[from=2-1, to=1-1]
		\arrow[from=2-1, to=2-2]
		\arrow[from=2-2, to=1-2]
		\arrow[from=3-1, to=2-1]
		\arrow["h_k",dashed, from=3-1, to=3-2]
		\arrow[from=3-2, to=2-2]
	\end{tikzcd}\]
	
	The upper dashed line exists because of the adjunction.
\end{proof}

\begin{rmk}\label{rmk:phantom_descent}
	One may consider the pointwise versions of conditions (1) and (2), where the filtration and the splitting are only 
	required to exist objectwise and not functorially. These conditions are still equivalent to one another, but they are 
	strictly weaker than the existing conditions. See \Cref{exm:weaklydescendable} for an example. 
	However, for the proof of \Cref{thm:cdescent} even weaker condition suffices: that for every object $x$, there exists a 
	length $m+1$ filtration on an object $y$ with associated graded in the image of $f^*$ and a map $x \to y$ with phantom 
	fiber. 
\tqed
\end{rmk}

\begin{dfn}\label{dfn:descendability}
We say that a functor $f:\cC \to \cD$ be a functor between small stable categories is {\bf descendable} if it satisfies the 
equivalent conditions of Proposition~\ref{prop:descendability}. We say that $m$ is the exponent of $f$. 
\end{dfn}

\begin{rmk}\label{rmk:mathews_descendability}
Let $\cC$ be a rigid symmetric monoidal category and $R \in \CAlg(\cC)$. 
Then the functor 
\[
f:\cC \to \Mod_R(\cC)
\] 
is descendable if and only if $R$ is descendable in the sense of 
\cite[Definition~3.18]{mathew2016galois}. The fact that our definition implies Mathew's is 
\cite[Proposition~3.20]{mathew2016galois}. The other direction follows from the same proposition and the observation 
that there is a canonical equivalence of cosimplicial diagrams
\[
M_f^{\otimes \bullet +1}(x) \simeq x\otimes M_f^{\otimes \bullet +1}(\unit)
\]
coming from the projection formula $f^*(f_! x \otimes y) \simeq x \otimes f^*y$.
\tqed \end{rmk}

\begin{thm}[Descent of connectivity]\label{thm:cdescent}
Let $f:\cC \to \cD$ be a descendable functor between small stable categories of exponent $m$. 
Assume that $\cD$ admits a $c$-structure of width $\le n$ and $\Ind(\cC)$ admits a compactly generated weight structure such that 
$f_!$ 
preserves weight coconnective objects and sends weight connective objects to $-d$-connective objects for some 
$d \in \NN$. 
Then the weight structure satisfies the axioms for a $c$-structure of width $\le n+m(d+1)$ on $\cC$. 
\end{thm}

We note that the above theorem is not applicable in the following situation, where we wonder whether it is anyways possible to descend the $c$-structure:
\begin{qst}
	Suppose that $R \to S$ is a descendable map of $\EE_{1}$-rings such that $S^{\otimes_{R}i}$ is connective for each $i\geq 1$. Then is the pre-$c$-structure on $\Mod(R)^{\omega}$ generated by $R$ a $c$-structure?
\end{qst}

We can use the above theorem to deduce the example from the introduction:

\begin{exm}\label{exm:introexample}
	Let $G$ be a finite cohomological dimension pro-$p$-group, fix $n\geq1$, and consider the map of rings $(\ZZ/p^n)^{hG} \to \FF_p^{hG}$. Here we interpret the homotopy fixed points as $\colim_{\alpha}(-)^{hG_{\alpha}}$ where $G = \lim G_{\alpha}$ for $G_{\alpha}$ finite.
	
	We first claim that the conditions of \Cref{thm:cdescent} are satisfied for the horizontal maps in the square\begin{center}
		\begin{tikzcd}
			\mathrm{Perf}((\ZZ/p^n)^{hG}) \ar[r]\ar[d,"f"] &\mathrm{Perf}(\ZZ/p^n) \ar[d]\\
			\mathrm{Perf}(\FF_p^{hG})	\ar[r] & \mathrm{Perf}(\FF_p)
		\end{tikzcd}
	\end{center}
	where $\Ind$ of all categories are equipped with the compactly generated weight structure generated by the unit. The 
	horizontal functors are descendable by the assumption that $G$ has finite cohomological dimension, so it suffices to see 
	that the functors $- \otimes_{R^{hG}}R$ for $R = \ZZ/p^n\ZZ,\FF_p$ preserve connective objects. To do this, it suffices to 
	prove that $R$ is connective in the standard $t$-structure on $R^{hG}$-modules generated by the unit. Since $\FF_p$ is 
	connective in the $t$-structure on $\ZZ/p^n$-modules, and $\ZZ/p^n$ is built out of extensions from $\FF_p$, it is enough 
	to prove this for $\FF_p$. As an $\FF_p^{hG}$-module, $\FF_p$ is the colimit of $\FF_p^{hU_{\alpha}}$ indexed over all 
	open normal subgroups $U_{\alpha}\subset G$, so it suffices to show each of these is connective. But this follows from the fact that the function representation $\FF_p^{G/U_{\alpha}}$ is a unipotent continuous $G$-representation.
	
	Next, we must show that the map $f$ is a nilpotent extension. It is clearly $0$-affine and preserves weight coconnective 
	objects, and also preserves weight connective objects since $\FF_p$ is connective in the $t$-structure on 
	$\ZZ/p^n$-modules. It follows that the weight heart of the source generates that of the target, so it is enough to show 
	that the map on weight hearts is surjective on morphisms with nilpotent kernel. 
	Given $M \in \Mod((\ZZ/p^n)^{hG})$ that is in the weight heart, the map 
	$\pi_0\End(M) \to \pi_0\Hom(M,M\otimes_{\ZZ/p^n}\FF_p) \simeq \End(M\otimes_{\ZZ/p^n}\FF_p)$ 
	can be identified with the base change of a ring along $\ZZ/p^n \to \FF_p$. This is always a nilpotent extension. 
\end{exm}

Recall from Theorem~\ref{thm:adjacent_t} that a compactly generated weighted category $\Ind(\cC)$ is endowed with the 
adjacent t-structure such that 
\[
\Ind(\cC)_{t\ge 0} = \Ind(\cC)_{w\ge 0}.
\] 
To prove Theorem~\ref{thm:cdescent} we need the following two lemmas about this t-structure.

\begin{lem}\label{lem:adjacent_killing weights}
Let $\cC$ be a small stable category such that $\Ind(\cC)$ is endowed with a compactly generated weighted structure. 
Then a morphism $x\to y \in \cC$ kills weights $\le k$ if and only if for any $b \in \Ind(\cC)_{t\le k}$ and any map 
$y \to b$, the composite 
\[
x \to y \to b
\]
is $0$.
\end{lem}
\begin{proof}
One direction is easy: if $x\to y$ factors through an object of $\Ind(\cC)_{w\ge k+1} = \Ind(\cC)_{t\ge k+1}$, then the 
composite $x\to y \to b$ is trivial by the orthogonality axiom for t-structures. 
Conversely, assume that for any $b \in \Ind(\cC)_{t\le k}$ and any map $y \to b$, the composite $x \to y \to b$ is trivial. 
In particular, it is true for $b=t_{\le k}y$, Hence the map $x\to y$ factors through $t_{\ge k+1} y \in \Ind(\cC)_{w\ge k+1}$.
\end{proof}

\begin{lem}\label{lem:descendable_t-structure}
Let $\cC$, $\cD$ be stable categories with t-structure and $f:\cD \to \cC$ be a functor of t-amplitude $[-a,b]$, i.e. 
$f(\cD_{t\ge 0}) \subset \cC_{t\ge -a}$ and $f(\cD_{t\le 0}) \subset \cC_{t\le b}$. Let 
\[
	0=y_{-1} \to y_0 \to \dots \to y_m
\]
be a finite filtration in $\cC$ with an identification of associated graded $y_i/y_{i-1}$ with $f(x_i)$ for some $x_i \in \cD$. Then there exists a natural map of filtered objects 
\[
\begin{tikzcd}
	0=y_{-1} \arrow[r] &y_{0}\arrow[d] \arrow[r] & \dots \arrow[r] & y_m\arrow[d]\\
	0=\bar{y}_{-1} \arrow[r] &\bar{y}_{0}\arrow[r] & \dots \arrow[r] & \bar{y}_m
\end{tikzcd}
\]
such that $\bar{y}_i/\bar{y}_{i-1} \simeq f(\bar{x}_i)$ with $\bar{x}_i \simeq \tau_{\le a + i(a+b+1)} x_i$. 
In this case the fiber is in $\cC_{t\ge 1}$. 
\end{lem}
\begin{proof}
We construct the new filtration by induction on $m$. 
Assume the statement proved for all filtrations of length $\le m-1$. 
In particular, we have constructed the desired maps $y_i \to \bar{y}_i$ for $i\le m-1$. Define $y'_m$ to be the pushout 
$y_m \times_{y_{m-1}} \bar{y}_{m-1}$. By construction, we have $y'_m/\bar{y}_{m-1} \simeq y_m/y_{m-1}$. 
Consider the fiber sequence 
\[
	\bar{y}_{m-1} \to y'_m \to f(x_m).
\]
Since $\bar{y}_{m-1}$ is an extension of objects of $\cC_{t \le (i+1)(a+b+1)}$ for $0\le i\le m-1$, it belongs to 
$\cC_{t \le m(a+b+1)}$. Now, because $f(\tau_{\ge a + m(a+b+1) +1})$ is in $\cC_{\ge m(a+b+1) +1}$, the map $f(x_m) \to \Sigma y_{m-1}$ factors through $f(\tau_{\le a + m(a+b+1)})$. This allows us to define $\bar{y}_m$ to be the fiber of that map. 
\end{proof}

\begin{lem}\label{lem:adjacent_descendable}
Fix the assumptions of Theorem~\ref{thm:cdescent}. 
The following statements are valid:
\begin{enumerate}
\item any object $y\in \Ind(\cC)_{t\le 0}$ is a retract of an object $y_{m}$ that admits a finite filtration $0=y_{-1} \to y_0 \to y_1 \dots \to y_m$ where $y_i/y_{i-1} \cong f^*x_i$ where $x_i \in \Ind(\cD)_{t\le i(d+1)}$, where $t$ stands for the adjacent t-structure;
\item any object $y\in \Ind(\cC)_{t\ge 0}$ is a retract of an object $y_m$ that admits a finite filtration $0 = y_{-1} \to y_0 \to \dots \to y_m$ where $y_i/y_{i-1} \cong f^*(x_i)$ where $x_i \in \Ind(\cD)_{t\geq -(d+1)(m-i+1)+1}$.
\end{enumerate}
\end{lem}
\begin{proof}
Observe that $f^*:\Ind(\cD) \to \Ind(\cC)$ preserves t-connectivity and sends t-coconnective objects to $d$ t-coconnective 
objects. 
To show that it preserves t-connectivity, observe that for any $x\in \Ind(\cD)_{t\ge 0} =\Ind(\cD)_{w\ge 0}$ and any 
$a\in \Ind(\cC)_{w\le 0}$ the spectrum 
\[
\map(a,f^*(x)) \simeq \map(f_!(a),x)
\]
is connective, since $f_!$ preserves weight coconnectivity. Similarly, for $x\in \Ind(\cD)_{t\le 0}$ and 
$a \in \Ind(\cC)_{t\ge 0}=\Ind(\cC)_{w\ge 0}$ we have
\[
\pi_0\map(a,f^*(\Omega^d x)) \simeq \pi_0\map(f_!(\Sigma^d a),x) \simeq 0 
\]
since $f_!$ 
sends weight $d$-connective objects to weight connective objects. 

Now both statements follow from \Cref{lem:descendable_t-structure} applied to $f^*$ and its dual. 
Note that for a fiber sequence 
\[
	x \to y_m \to \bar{y}_m
\]
with $x$ in $\Ind(\cC)_{t\ge 1}$ the retraction diagram $y\to y_m \to y$ induces a retraction diagram $y \to \bar{y}_m \to y$.
\end{proof}

\subsubsection{Proof of {Theorem~\ref{thm:cdescent}}}
Consider $x \in \Ind(\cC)_{w\ge 0}$. Choose any filtered diagram $I \to \cC$ such that 
\[
\colim_I x_i \simeq x.
\] 
It suffices to show that for any $i\in I$ there exists $i\to N \in I$ such that $x_i \to x_{N}$ kills weights 
$\le -n - m(d + 1) -1$. 

Since $f_!x \in \Ind(\cD)_{w\ge 0}$ and $\cD$ is a $c$-category of width $\le n$, there exists a map $i\to t(i) \in I$ such that 
$f(x_i) \to f(x_{t(i)})$ kills weights $\le -n - 1$. This works for any $i\in I$, so we iterate this procedure $m$ times and 
set $N = t^{\circ m+1}(i)$. 
Now we will show that the composite  
\begin{equation}\label{eq:composite}
x_i \to x_{t(i)} \to x_{t(t(i))} \to \dots \to x_{t^{\circ m+1}(i)}
\end{equation}
kills weights $\le -n - m(d + 1) -1$. 
By Lemma~\ref{lem:adjacent_killing weights} it suffices to show that for any 
$b\in \Ind(\cC)_{t\le -n - m(d + 1) - 1}$  
and any map $x_{t^{\circ m}(i)} \to b$, the composite $x_i \to x_{t^{\circ m}(i)} \to b$  
is trivial.

By Lemma~\ref{lem:adjacent_descendable}(1),  $b$ is a retract of an object $b'$ that admits an $m+1$-step filtration $F_jb'$ 
for $1\le j \le m+1$ whose graded pieces are of the form 
$f^*a_j$ for $a_j \in \Ind(\cD)_{t\le -n-1}$. We set $F_0b' = 0$. 
It suffices to prove via decreasing induction on $0\le j\le m+1$ that the composite $x_i \to x_{t^{\circ m+1}(i)} \to b'$ 
factorizes as
\[
x_i \to x_{t^{\circ j}(i)} \to F_jb' \to b'.
\]
For $j=0$ this would show that the composite (\ref{eq:composite}) kills weights $\le -n - m(d + 1) -1$. 
For $j=m+1$ the statement is tautological. Now if we already have a factorization for a given $j$, consider the commutative 
diagram 
\[
\begin{tikzcd}
x_{t^{\circ j-1}(i)} \arrow[r]& x_{t^{\circ j}(i)}\arrow[d]\\
F_{j-1}b' \arrow[r] & F_jb' \arrow[r] & f^*a_j.
\end{tikzcd}
\]
The composite $x_{t^{\circ j-1}(i)} \to x_{t^{\circ j}(i)} \to f^*a_j$ is $0$, since 
its adjoint is the composite $f_!x_{t^{\circ j-1}(i)} \to f_!x_{t^{\circ j}(i)} \to a_j$ where the first map kills weights
$\le -n -1$ by construction. The bottom row is a fiber sequence, so there exists a way to complete the commutative diagram above:
\[
\begin{tikzcd}
x_{t^{\circ j-1}(i)}\arrow[d,dotted] \arrow[r]& x_{t^{\circ j}(i)}\arrow[d]\\
F_{j-1}b' \arrow[r] & F_jb'.
\end{tikzcd}
\]
This finishes the proof of the inductive step and the theorem. \qed

We conclude the section with a similar theorem for descending heart structures from bounded weighted categories:

\begin{thm}[Adams type descent of heart structures]\label{thm:cdescentheart}
	Let $f:\cC \to \cD$ be a descendable map of exponent $\leq n$ where $\cD$ is a bounded weighted category, such that $f^*f_!$ 
	sends objects of $f^{-1}(\cD^{\heart})$ to filtered colimits of objects in $f^{-1}(\cD^{\heart})$, and 
	$f^{-1}(\cD^{\heart})$ generates $\cC$. 
	
	Then $f^{-1}\cD^{\heart}$ is the heart of a heart structure of cohomological 
	dimension $n$ on $\cD$, and $f_!$ is weight exact.
\end{thm}
\begin{proof}
	$f^{-1}D^{\heart}$ is an extension closed subcategory of $\cC$, and hence admits a natural exact structure. Let $\cC'$ be $\mathcal{D}^\mathrm{b}(f^{-1}\cD^{\heart})$.
	We have a commutative diagram
	
	\begin{center}
		\begin{tikzcd}

			\mathcal{D}^\mathrm{b}(f^{-1}\cD^{\heart})\ar[r,"D^b(f)"]\ar[d,"{f'}"] & \ar[d]D^b(\cD^{\heart})\\
			
						\cC \ar[r,"f"] &\cD
		\end{tikzcd}
	\end{center}
	where the right vertical map is an equivalence. It suffices to show that the map $f'$ is an equivalence. By assumption $f'$ generates $\cC$, so we will show that it is fully faithful.
	
	We first observe that for $x \in f^{-1}(\cD^{\heart})$ the cofiber sequence $x \to f^*f_!x \to \cof$ is a filtered colimit 
	of cofiber sequences of objects in the heart and in particular $\cof$ belongs to $\Ind(f^{-1}(\cD^{\heart}))$. 
	Indeed, we can write $f^*f_!x$ as a filtered colimit of objects in 
	$f^{-1}\cD^{\heart}$ by assumption, and by compactness of $x$, $x$ maps to a cofinal system of such objects. We claim that 
	the cofiber of each such map $x \to x'$ is in $f^{-1}\cD^{\heart}$. The map $f_!x\to f_!x'$ admits a 
	retraction given by the map $f_!x' \to f_!f^*f_!x \to f_!x$ using the counit of the adjunction. Since $\cD^{\heart}$ 
	is closed under fibers of retractions in $\cD$ (see \cite[Theorem~1.6]{saunier-winges}), the cofiber of $f_!x \to f_!x'$ 
	is in $\cD^{\heart}$. 

	By assumption $f^{-1}\cD^{\heart}$ generates $\cC$, so by the Keller's criterion \Cref{thm:Kellers_criterion} it suffices to show that the embedding $f^{-1}\cD^{\heart} \to f^{-1}(\cD_{\ge 0})$ is left special (note that 
	$\mathcal{D}^\mathrm{b}(f^{-1}(\cD_{\ge 0})) \simeq \mathrm{SW}(f^{-1}(\cD_{\ge 0})) \simeq \cD$). 
	We use Proposition~\ref{prop:speciality}(3) to check specialty. 
	Consider $x,y\in f^{-1}\cD^{\heart}$, $n\ge 2$ and a map $x \stackrel{\alpha}\to \Sigma^n y$. Using decreasing induction 
	on $n$, it suffices to show that $\alpha$ factorizes as 
	\[
		x \to \Sigma^{n-1} \tilde{y} \to \Sigma^n y.
	\]
	for some $\tilde{y} \in f^{-1}\cD^{\heart}$. 
	Using the fiber sequence $y \to f^*f_!y \to \cof$ and the fact that $\map_\cC(x,f^*f_!y) \simeq \map(f_!x,f_!y)$ is 
	connective, we obtain a map $x \to \Sigma^{n-1} \cof$ such that the composite $x \to \Sigma^{n-1} \cof \to y$ is $\alpha$. 
	Since $x$ is compact and since $\Sigma^{n-1} \cof \in \Ind(\Sigma^{n-1} f^{-1}\cD^{\heart})$, we may further factorize the map as
	$x \to \Sigma^{n-1}\tilde{y} \to \Sigma^{n-1}\cof$ for some $\tilde{y} \in f^{-1}\cD^{\heart}$ which finishes the proof.
\end{proof}

%
%

\section{Localizing invariants of \texorpdfstring{$c$-categories}{c-categories}}\label{sec:localizing_invariants_of_c_categories}

Theorem~\ref{thm:resolution} and its interpretation in Corollary~\ref{cor:resolution_full} become especially powerful in the 
context of localizing invariants. 
In this section we demonstrate some immediate applications of the theory of $c$-categories in this context. 
We start by recalling the basic definitions following \cite{BGT,Hoyois_2017,Land_2019}. 

\begin{dfn}\label{dfn:localizing_invariant}
A functor $F:\Cat^\mathrm{st} \to \cE$ valued in a stable category $\cE$ is said to be a {\bf localizing invariant} if 
for any localization sequence $\cA \to \cB \to \cC$ of small stable categories the sequence 
\[
F(\cA) \to F(\cB) \to F(\cC)
\]
in $\cE$ is a fiber sequence.
\end{dfn}

\begin{exm}
Topological Hochschild homology, (non-connective) algebraic $\K$-theory, topological cyclic homology are examples of localizing invariants.
\tqed \end{exm}

\begin{prop}\label{prop:motivic_resolution}
Let $\cC$ be a $c$-category of width $\le n$ and let $E$ be a localizing invariant. 
Then for $N = \mathrm{max}(n,1)$ there exist equivalences 
\[
\Sigma^{N} E(\cC) \simeq E(\mathrm{Perf}(R_\cC))
\]
for a connective ring spectrum $R_\cC$ constructed as the endomorphism spectrum of a generator of an object of the weight 
heart of $\cC^{(N)}$ (see Corollary~\ref{cor:resolution_full}). 
Moreover, for a 0-affine functor  $\cC \to \cD$ of $c$-categories of width $\le n$ we have a commutative diagram
\[
\begin{tikzcd}
\Sigma^N E(\cC) \arrow[d,"f"]\arrow[r, "\simeq"] & E(\mathrm{Perf}(R_\cC))\arrow[d]\\
\Sigma^N E(\cD) \arrow[r, "\simeq"] & E(\mathrm{Perf}(R_\cD)).
\end{tikzcd}
\]
where the right horizontal map is the base change along a map of rings $R_\cC \to R_\cD$.
\end{prop}
\begin{proof}
We are going to employ the localization sequences provided by Theorem~\ref{thm:resolution}. 
First observe that $\Ind^\kappa(\cC)$ has infinite sums, so by the Eilenberg swindle, any localizing invariant evaluates to $0$ on it. 
Now we prove the claim using induction on $n$. If $\cC$ is a $c$-category of width $\le 0$ or $\le 1$, 
pick $\kappa$ bigger than the size of $\cC$. Then $\Ind(\cC)^{\kappa, \mathrm{b}}$ is generated by a single object 
which we may also assume to be the object of the weight heart (by taking the sum of graded pieces of its weight 
filtration). Using this along with Theorem~\ref{thm:resolution} we 
see that $(\Ind(\cC)^{\kappa, \mathrm{b}}/\cC)^\mathrm{ic}$ is also generated by a single object in the weight heart. 
Since this generator belongs to the weight heart, its endomorphism ring spectrum $R$ is connective. Theorem of Schwede and 
Shipley (see \cite[Theorem~7.1.2.1]{HA}, \cite[Theorem~3.1.1]{Schwede2003}) implies that in this case we have an 
equivalence 
\[
(\Ind(\cC)^{\kappa, \mathrm{b}}/\cC)^\mathrm{ic} \simeq \mathrm{Perf}(R).
\]
Using the fact that $E$ is localizing we obtain an equivalence 
\[
\Sigma E(\cC) \simeq E(\mathrm{Perf}(R)).
\]
To show the induction step we again use the localization sequence 
\[
\cC \to \Ind^\kappa(\cC) \to (\Ind^\kappa(\cC)/\cC)^{\mathrm{ic}},
\]
Theorem~\ref{thm:resolution} and Proposition~\ref{prop:Karoubi-completion}. In this case they yield an equivalence 
$\Sigma E(\cC) \simeq E(\cC_1)$ for $\cC_1$ a $c$-category of width $\le n-1$. 

The construction above is functorial, so the moreover part follows from Lemma~\ref{lem:resolution_affinness} and 
Remark~\ref{rmk:rings_induced}.
\end{proof}

\begin{dfn}
Let $E:\Cat^{\mathrm{st}} \to \cE$ be a localizing invariant. 

\begin{enumerate}
\item When $\cE$ admits a t-structure, we say that $E$ is {\bf connected} if for a $k$-connective map of 
connective ring spectra $R\to S$ for $k\geq 1$, the induced map $E(\mathrm{Perf}(R))\to E(\mathrm{Perf}(S))$ is $k+1$-connective.

\item We say that $E$ is {\bf truncating} (see {\cite{Land_2019}}\label{dfn:truncating_invariant}) if for any ring spectrum 
$R$ such that the induced map $E(\mathrm{Perf}(R))\to E(\mathrm{Perf}(\pi_0(R)))$ is an equivalence.
\end{enumerate}
\end{dfn}

\begin{thm}\label{thm:truncating_invariants}
Let $f:\cC \to \cD$ be a $k$-connective functor of $c$-categories of width $\le n$ and let $E$ be a 
localizing invariant. Set $N=\mathrm{max}(n,1)$. 
\begin{enumerate}
\item If $E$ is a connected invariant and $f$ is $k$-connective then the map $E(\cC) \to E(\cD)$
is $k-N+1$-connective.
\item  If $E$ is a truncating invariant and $f$ is a nilpotent extension, then the map 
$E(\cC) \to E(\cD)$ is an equivalence.
\end{enumerate}
\end{thm}
\begin{proof}
By Proposition~\ref{prop:motivic_resolution} we have a commutative diagram 
\[
\begin{tikzcd}
\Sigma^N E(\cC) \arrow[d,"f"]\arrow[r, "\simeq"] & E(\mathrm{Perf}(R_\cC))\arrow[d]\\
\Sigma^N E(\cD) \arrow[r, "\simeq"] & E(\mathrm{Perf}(R_\cD))
\end{tikzcd}
\]
whose horizontal maps are equivalences. Moreover, the map $\mathrm{Perf}(R_\cC) \to \mathrm{Perf}(R_\cC)$ is by construction 
equivalent to the $N$-th map in the resolution of Corollary~\ref{cor:resolution_full}: $f^{(N)}:\cC^{(N)} \to \cD^{(N)}$. 

Under the assumptions of (1), $f^{(N)}$ is a $k$-connective functor of weighted categories by 
Corollary~\ref{cor:resolution_nilpotent}, so in particular the maps $R_\cC \to S_\cC$ are $k$-connective. 
Thus the vertical maps in the commutative diagram above are $k+1$-connective. Desuspending the left vertical map, we 
prove the first claim. 

Under the assumptions of (2), by Theorem~\ref{thm:nilpotent_resolution} $f^{(N)}$ is a nilpotent extension of weighted 
categories, so in particular it induces a a surjection $\pi_0R_\cC \to \pi_0R_\cD$ with nilpotent kernel. Thus, the vertical 
maps in the diagram are also equivalences. 
\end{proof}

\begin{cor}\label{cor:width_n_thh}
For any $c$-category $\cC$ of width $\le n$ the topological Hochschild homology of $\cC$ is $-n$-connective. 
\end{cor}
\begin{proof}
This follows from the fact that $\mathrm{THH}$ of connective ring spectra is connective. 
\end{proof}

\begin{cor}[DGM]\label{cor:DGM}
Let $\cC \to \cD$ be a nilpotent extension of $c$-categories. Then the commutative square
\[
	\begin{tikzcd}
		\K(\cC) \arrow[r]\arrow[d] & \TC(\cC)\arrow[d]\\
		\K(\cD) \arrow[r] & \TC(\cD)
	\end{tikzcd}
\]
is a pullback square.
\end{cor}
\begin{proof}
Apply Theorem~\ref{thm:truncating_invariants}(2) to $\fib(\K(-) \to \mathrm{TC}(-))$, which is a truncating invariant by \cite{dundas2012local}.
\end{proof}

\begin{cor}\label{cor:kthyconnectivity}
Let $\cC \to \cD$ be a $k$-connective map between $c$-categories of width $\le n$. 
Then the map 
\[
	\K_i(\cC) \to \K_i(\cD)
\]
is an isomorphism for $i < -n + k + 1$.
\end{cor}
\begin{proof}
Follows from Theorem~\ref{thm:truncating_invariants}(1), and the fact that $K$-theory is a connected localizing invariant \cite[Lemma 2.4]{Land_2019}.
\end{proof}

\begin{rmk}
Corollary~\ref{cor:width_n_thh} implies in particular that there are not as many weakly approximable triangulated 
categories. Take $R = \mathrm{C}^\star(X,\mathbb{Z})$ for a simply connected space $X$. 
In this case Loday showed an equivalence
\[
\mathrm{THH}(R) \simeq \mathrm{C}^*(LX, \mathbb{Z})
\]
where $LX$ denotes the space of unpointed loops in $X$. Whenever $LX$ has nontrivial cohomology in infinitely many degrees 
(for instance, when $X=S^2$), this spectrum is not $n$-connective for any $n$. Hence, $\mathrm{Perf}(R)$ does not admit a 
$c$-structure. 

Since weakly approximable categories admit $c$-structures by \Cref{cor:approximable_ccategory}, we learn that $\mathrm{Perf}(R)$ is not weakly approximable, answering the question formulated before Remark 3.3 in \cite{canonaco2024}. 
\tqed \end{rmk}

\section{Applications to algebraic stacks}\label{sec:stacks}

For an affine scheme $X$ the category $\mathrm{Perf}(X)$ admits a bounded weight structure 
(Example~\ref{exm:weight_conn_ring}). More generally, this also works for $X$ being an affine quotient stack 
$[\mathrm{Spec} R/G]$ for a linearly reductive group $G$ (\cite[Theorem~2.2.21]{nilpotentextns}). Using this, 
Elmanto and the second author showed an equivariant version of DGM for such stacks. However, for general schemes and stacks 
there is no bounded weight structure on $\mathrm{Perf}(X)$. 
Nevertheless, in this section we show that for many stacks $\mathrm{Perf}(-)$ admit a natural $c$-structure 
(Theorem~\ref{thm:stacky_cstr_2}). 
Moreover, the nilpotent extensions of stacks induce nilpotent extensions of $c$-categories 
(Theorem~\ref{thm:morphisms_of_stacks}), which allows us to show DGM and similar statements for such stacks. 

We define $\mathrm{Perf}(X)_{\le 0}$ to be the subcategory of perfect complexes of non-positive tor-amplitude.  
For $X$ of finite cohomological dimension this is a pre-$c$-structure on $\mathrm{Perf}(X)$, 
and we can show that it is a $c$-structure under stronger assumptions. These assumptions are satisfied in many natural 
examples (see Example~\ref{exm:connective_perfect_generation}).

\subsection{Generality on quasi-geometric stacks}\label{subsec:stacks}

Given a functor $\mathcal{F}:\mathrm{CAlg}^{\mathrm{cn}} \to \Spc$, one may define the category 
$\mathrm{QCoh}(\mathcal{F})$ via right Kan-extending $\mathrm{QCoh}$ along 
\[
\Spec:\mathrm{CAlg}^{\mathrm{cn}} \to \Fun(\mathrm{CAlg}^{\mathrm{cn}}, \Spc)^\mathrm{op}.
\] 
Similarly, one may define the subcategories $\mathrm{QCoh}(X)^{\mathrm{cn}}$ and $\mathrm{Perf}(X)$. 
When $X$ is a {\it quasi-geometric stack}, these categories contain a lot of information about $X$. 
For instance, applying a localizing invariant to $\mathrm{Perf}(X)$, one obtains an interesting cohomology theory of 
$X$---this way $\K$-theory and Hochschild homology of an algebraic stack can be constructed. The goal of this section is to 
apply the abstract results of previous sections to obtain some new results about $\mathrm{Perf}(X)$ and its localizing 
invariants. 


\subsubsection{Morphisms of quasi-geometric stacks}\label{subsubsec:stacks}
For a natural transformation of functors from $\mathrm{CAlg}^{\mathrm{cn}}$ to $\Spc$ 
\[
f:\mathcal{G} \to \mathcal{F}
\] 
one can make sense of schematic properties such as being a closed immersion, an open immersion, affine, quasi-affine, locally 
of finite presentation, or faithfully flat, by imposing that the given property is 
satisfied for the pullback along any map $\Spec R \to \mathcal{F}$. More precisely, we say that $f$
is representable if in the pullback diagram
\[
\begin{tikzcd}
X \arrow[r]\arrow[d] & \mathcal{G}\arrow[d]\\
\Spec R \arrow[r] & \mathcal{F}
\end{tikzcd}
\]
$X$ is representable by a qcqs algebraic space; we say that $f$ satisfies a 
property P if $f$ is representable and the map of algebraic spaces $X \to \Spec R$ satisfies the property $P$. 

\begin{dfn}\label{dfn:stk}
A functor $\mathcal{F}:\mathrm{CAlg}^{\mathrm{cn}} \to \Spc$ is called an {\bf quasi-geometric stack} if it is a sheaf in the 
fpqc topology, its diagonal $\mathcal{F} \to \mathcal{F}\times \mathcal{F}$ is quasi-affine, and there exists an 
fppf\footnote{representable and faithfully flat locally of finite presentation.} map $\Spec R \to \mathcal{F}$. A map of 
quasi-geometric stacks is a natural transformation of functors. We say that $\mathcal{F}$ is a geometric stack, if the 
diagonal is furthermore affine. 

We say that a quasi-geometric stack is {\bf classical} if it admits an fppf cover by an affine classical scheme. 
\end{dfn}

\begin{wrn}
In the definition of a quasi-geometric stack from \cite{SAG} the map $\Spec R \to \mathcal{F}$ is only required to be
representable and faithfully flat.
\end{wrn}

\begin{rmk}\label{rmk:limit_descroption}
The fppf map $\Spec R \to \mathcal{F}$ for a quasi-geometric stack in particular gives a surjection in the fpqc topology. 
\tqed \end{rmk}

\begin{rmk}
For a quasi-geometric stack $X$ a faithfully flat map $\Spec R \to X$ is automatically quasi-affine. 
\tqed \end{rmk}

\begin{rmk}\label{rmk:stk_tstruct}
For $X$ a quasi-geometric stack $\mathrm{QCoh}(X)$ is a presentable category equivalent to the stabilization of the prestable 
category $\mathrm{QCoh}(X)^{\mathrm{cn}}$. In particular, there is a canonical t-structure on $\mathrm{QCoh}(X)$  such that 
for an fppf map $\Spec R \to X$, the induced functor  $\mathrm{QCoh}(X) \to  \mathrm{QCoh}(\Spec R)$ is t-exact 
(\cite[Corollary~9.1.3.2]{SAG}). This t-structure is right and left complete and is compatible with filtered colimits. 
Furthermore, $\mathrm{QCoh}(X)$ admits a presentable symmetric monoidal structure such that 
the tensor product of connective objects is connective and such that the pullback functor is strongly monoidal 
(see \cite[Section~6.2.6]{SAG}).
\tqed \end{rmk}

\begin{rmk}\label{rmk:descent_qcoh}
The category $\mathrm{QCoh}(X)$ satisfies fpqc descent \cite[Proposition~6.2.3.1]{SAG}. 
In particular, for an fppf map $U \stackrel{p}\to X$ we have 
\[
\mathrm{QCoh}(X) \simeq \lim_{i\in \Delta} \mathrm{QCoh}(U^{\otimes_X i}).
\]
denote the projection maps $U^{\otimes_X i} \to X$ by $p_n$. 
Specializing further, given any quasi-coherent sheaves $F,G \in \mathrm{QCoh}(X)$, the have an equivalence 
\[
\map_{\mathrm{QCoh}(X)}(F,G) \simeq \lim_{i\in \Delta} \map_{\mathrm{QCoh}(U^{\otimes_X i})}(p_n^*F,p_n^*G).
\]
In general, one may consider the category $\mathrm{Flat}_{/X}$ of flat morphisms over $X$, and view any quasi-coherent 
sheaf $F$ as an {\it fpqc sheaf} of spectra on $X$ via setting $F^{\mathrm{fpqc}}(V/X) := \map(\mathcal{O}_V,F)$. 
Denoting the category of fpqc-sheaves on $X$ by $\mathrm{Sh}_{\mathrm{fpqc}}(X; \Sp)$, this yields a functor 
\[
	(-)^{\mathrm{fpqc}}:\mathrm{QCoh}(X) \to \mathrm{Sh}_{\mathrm{fpqc}}(\mathrm{Flat}_{/X}; \Sp).
\]
\tqed \end{rmk}

By construction, a map of quasi-geometric stacks $f:Y \to X$ induces an adjunction 
\[
\begin{tikzcd}
\mathrm{QCoh}(X) \arrow[r, shift left, "f^*"] & \mathrm{QCoh}(Y).\arrow[l, shift left, "f_*"]
\end{tikzcd}
\]
In the next proposition we show various important compatibility properties of these functors, the monoidal structure on 
$\mathrm{QCoh}(-)$ and the t-structure of Remark~\ref{rmk:stk_tstruct}. 

\begin{prop}\label{prop:basechange}
The following statements are valid.
\begin{enumerate}
\item Let $f:Y \to X$ be a representable map of quasi-geometric stacks. Then for any $x\in \mathrm{QCoh}(X)$ and 
$y\in \mathrm{QCoh}(Y)$ the canonical map $y \otimes f^*x \to f^*(y \otimes f_*x)$ is an equivalence.
\item Given a pullback diagram of quasi-geometric stacks
\[
\begin{tikzcd}
Y' \arrow[r,"{\bar{g}}"]\arrow[d, "{\bar{f}}"] & Y\arrow[d, "f"]\\
X' \arrow[r, "g"] & X
\end{tikzcd}
\] 
with $f:Y \to X$ representable, the square 
\[
\begin{tikzcd}
\mathrm{QCoh}(X) \arrow[d,"g^*"]\arrow[r, "f^*"] & \mathrm{QCoh}(Y)\arrow[d, "{\bar{g}}^*"]\\
\mathrm{QCoh}(X')\arrow[r,"{\bar{f}}^*"] & \mathrm{QCoh}(Y')
\end{tikzcd}
\]
is horizontally right adjointable. In other words, the canonical map $g^*f_* \to \bar{g}^*\bar{f}_*$ is an equivalence.
\item For a representable map $f:Y \to X$ of quasi-geometric stacks there exists an integer $d$ such that the pushforward map 
$f_*$ sends $\mathrm{QCoh}(Y)^{\mathrm{cn}}$ to $\Omega^d\mathrm{QCoh}(X)^{\mathrm{cn}}$. Moreover, if $f$ is affine, we may 
choose $d=0$, so $f_*$ is right t-exact. 
\end{enumerate}
\end{prop}
\begin{proof}
The first two claims are \cite[Corollary~6.3.4.3]{SAG} and \cite[Proposition~6.3.4.1]{SAG} (see also 
\cite[Proposition~A.1.5]{Halpern_Leistner_2023}). 

To prove the third claim, consider an fppf cover $g:U \to X$ where $U$ is an affine scheme. 
In the pullback square
\[
\begin{tikzcd}
Y' \arrow[r,"{\bar{f}}"]\arrow[d,"{\bar{g}}"] & U\arrow[d, "g"]\\
Y\arrow[r,"f"] & X
\end{tikzcd}
\]
$Y'$ is quasi-compact, so the map $\bar{f}_*$ sends connective objects to $-d$-connective 
objects for some $d$. Moreover, if $f$ is affine, we may choose $d=0$. 

Now to show that $f_*x$ is $-d$-connective for $x$ connective it suffices to show that $g^*f_*x$ is connective. 
By the base change formula proved before we have $g^*f_*x \simeq \bar{f}_*\bar{g}^*x$. 
We know that the latter is $-d$-connective thanks to the above observation about $\bar{f}_*$ and t-exactness 
of the pullback to a flat cover (see Remark~\ref{rmk:stk_tstruct}).
\end{proof}

\begin{lem}\label{lem:qaff_tstr}
Let $Y$ be a quasi-affine scheme. Then $\mathrm{QCoh}(Y)^{\mathrm{cn}}$ is equal to the minimal subcategory containing 
$\mathcal{O}_Y$ and closed under colimits, extensions, and retracts. 
\end{lem}
\begin{proof}
The claim is true for affine schemes. 
Find an open immersion $i:Y \to X$ such that $X$ is an affine scheme. Now the functor 
$i^*:\mathrm{QCoh}(X) \to \mathrm{QCoh}(Y)$ is a localization and t-exact. 
Consider a new t-structure on $\mathrm{QCoh}(Y)$ such that $\mathrm{QCoh}(Y)_{\bar{t} \ge 0}$ is generated by 
$i^*\mathcal{O}_X = \mathcal{O}_Y$. It suffices to show that $\bar{t} = t$. 
We now have inclusions
\[
\mathrm{QCoh}(Y)_{\bar{t} \ge 0} \subset \mathrm{QCoh}(Y)_{t \ge 0} \quad \quad \quad \quad \mathrm{QCoh}(Y)_{t \le 0} \subset \mathrm{QCoh}(Y)_{\bar{t} \le 0}.
\]
Observe that $i^*$ is also t-exact with respect to $\bar{t}$. 
Now for any $y\in \mathrm{QCoh}(Y)^{\mathrm{cn}}$ take $x$ with $i^*x = y$. 
We see that $\bar{t}_{\le -1}y = i^*t_{\le -1}x = t_{\le -1} y =0$, hence $y\in \mathrm{QCoh}(Y)_{\bar{t} \ge 0}$.
\end{proof}

\begin{dfn}\label{dfn:cohomological_dimension}
For a representable morphism $f:Y \to X$ the cohomological dimension $\mathrm{cd}(f)$ is defined to be the least non-negative 
integer $d$ such that the property $(3)$ of Proposition~\ref{prop:basechange} holds.

The proof of the proposition shows, $cd(f)$ is bounded from above by the coherent cohomological dimension of the qcqs 
algebraic space $Y'$. If $f$ is 
quasi-affine, then $Y'$ is a quasi-affine scheme and Lemma~\ref{lem:qaff_tstr} gives: $\mathrm{cd}(Y')$ is the greatest 
integer $i$ 
such that the group $\pi_{-i}\mathrm{R}\Gamma(Y',\mathcal{O}_{Y'})$ is non-trivial.
\end{dfn}

\subsubsection{Compact generation on stacks}

\begin{dfn}
Let $X$ be a quasi-geometric stack.

\begin{enumerate}
\item We say that $X$ is {\bf of finite cohomological dimension} or {\bf fcd} if the functor
\[
	\mathrm{R}\Gamma(X,-) : \mathrm{QCoh}(X) \to \Sp
\]
sends $\mathrm{QCoh}(X)^{\mathrm{cn}}$ to $\Sp_{\ge -d}$ for some $d\in \NN$.
\item We say that $X$ satisfies {\bf perfect generation} if it is concentrated and $\mathrm{QCoh}(X)$ is compactly generated. 
It follows that we have an equivalence $\Ind(\mathrm{Perf}(X)) \simeq \mathrm{QCoh}(X)$.
\end{enumerate}
\end{dfn}

\begin{rmk}
For a quasi-geometric stack $X$ the following are equivalent (see \cite[Remark~4.6]{Hall_2017} for the case of classical 
stacks; the general case is the same):
\begin{enumerate}
\item 
all perfect complexes in $\mathrm{QCoh}(X)$ are compact; 
\item
$\mathcal{O}_X$ is a compact object;
\item 
$X$ is of finite cohomological dimension.
\end{enumerate}
\tqed \end{rmk}

\begin{ntn}\label{ntn:dual}
Perfect complexes in $\mathrm{QCoh}(X)$ form the subcategory of dualizable objects. 
We denote by $P^\vee$ the dual of $P \in \mathrm{Perf}(X)$. 
\end{ntn}

\begin{exm}\label{exm:char_0}
It follows from \cite[Theorem~0.3.2]{Drinfeld2013} that any quasi-geometric stack with affine stabilizers in characteristic 
$0$ is fcd. It still remains an open question whether there are examples of fcd stacks that do no satisfy perfect 
generation (\cite[Remark~9.4.0.1]{SAG}). 
\tqed \end{exm}

\begin{exm}\label{exm:char_p}
An ultimate example of quasi-geometric stack that does not satisfy perfect generation 
is given by $X=B\mathbb{G}_a$ in prime characteristic. In this case the only compact objects in $\mathrm{QCoh}(X)$ are the 
zero objects (\cite[Proposition~3.1]{Hall_2018}). 
In \cite[Theorem~1.12]{Alper2022ArtinAF} the following conditions were proved to be equivalent in characteristic $p$ 
under the assumption that $X$ is a classical quasi-geometric stack and has affine diagonal:
\begin{enumerate}
\item $X$ is fcd;
\item $X$ satisfies perfect generation;
\item the stabilizers at closed points of $X$ are linearly reductive;
\item the stabilizers at all points of $X$ are linearly reductive.
\end{enumerate} 
\tqed \end{exm}

\begin{dfn}\label{dfn:resolution_prop}
We say that a quasi-geometric stack $X$ satisfies the {\bf resolution property} if for any $n\in\NN$ and any 
$F \in \mathrm{QCoh}(X)_{t\le n}$, there exist vector bundles $\cE_i$ and a family of maps $\cE_i \to F$ 
such that the induced map $\bigoplus_i \pi_0(\cE_i) \to \pi_0(F)$ is an epimorphism. 
\end{dfn}

\begin{exm}\label{exm:BG}
For any separated regular noetherian scheme $S$ of Krull dimension $\le 1$ and any affine group scheme $G$ over $S$ the stack 
$B_SG$ satisfies the resolution property (\cite[Lemma~2.4]{Thomason1987}). 
\tqed \end{exm}

\begin{exm}\label{exm:VmodG}
If $f:Y \to X$ is a quasi-affine map of quasi-geometric stacks and $X$ satisfies the resolution property, then 
$Y$ also satisfies the resolution property (\cite[Proposition~9.3.3.8]{SAG}). 
For a quasi-affine scheme $V$ endowed with an action of an affine group scheme $G$ over $S$, 
the map $[V/G] \to B_SG$ is quasi-affine. In particular, employing Example~\ref{exm:BG} we see that the quotient 
stack $[V/G]$ satisfies the resolution property at least if $S$ is separated regular noetherian of Krull dimension $\le 1$.
\tqed \end{exm}

\begin{rmk}\label{rmk:perf_comp}
On an fcd stack $X$ a sheaf $F\in \mathrm{QCoh}(X)$ is a perfect complex if and only if it is compact 
(see \cite[Proposition~9.1.5.3]{SAG}). In particular, any vector bundle on an fcd stack is perfect. 
\tqed \end{rmk}

\begin{exm}\label{exm:resprop_cg}
Any fcd stack satisfying the resolution property satisfies perfect generation (see \cite[Proposition~9.3.3.7]{SAG}). 
\tqed \end{exm}

\begin{exm}\label{exm:descent_cg}
It follows from \cite[Theorem~C]{Hall_2017} that any classical quasi-geometric stack $X$ satisfies perfect generation as 
soon as it admits a representable separated quasi-finite and faithfully flat map locally of finite presentation 
$X' \to X$ such that $X'$ is fcd and satisfies the resolution property. 
\tqed \end{exm}

\begin{dfn}
We say that a quasi-geometric stack $X$ satisfies {\bf connective perfect generation} if it is fcd and the prestable 
category $\mathrm{QCoh}(X)^{\mathrm{cn}}$ is compactly generated by connective perfect complexes. 
\end{dfn}

\begin{exm}\label{exm:connective_perfect_generation}
While in general we do not have a good description of a general class of stacks satisfying connective perfect generation, 
this property is satisfied in many cases listed below
\begin{enumerate}
\item Any qcqs spectral algebraic space satisfies connective perfect generation (\cite[Proposition~9.6.1.2]{SAG}).
\item Any finite cohomological dimension stack satisfying the resolution property also satisfies connective perfect 
generation (\cite[Proposition~9.3.3.7]{SAG}). 
For instance, in characteristic 0 this holds for any quotient stack $[V/G]$ for $V$ a quasi-projective scheme and any 
affine $G$ (see \Cref{exm:VmodG}). 
\item A stronger property of approximation by compact complexes introduced in \cite{Lipman2007} and proved for all 
qcqs schemes. In \cite[Theorems~A,B]{hall2025compactapproximationdescentalgebraic} it was also shown for a large class of 
classical stacks: those having quasi-finite and separated diagonal, Deligne-Mumford stacks in characteristic 0, and those 
admitting a good moduli space.
\item Let $X$ be a noetherian quasi-geometric stack $X$ that satisfies perfect generation and that is additionally 
regular, i.e. it admits an fppf cover by a regular scheme $\pi:\Spec R \to X$ 
(cf. \cite[Definition~2.15(ii)]{sosnilo2021regularity}). Then $X$ also satisfies connective perfect generation. 
To see this, first observe that for any $P \in \mathrm{Perf}(X)$ its truncations $\tau_{\ge 0}P$ are also perfect 
complexes (since locally they are). Consequently, the t-structure on $\mathrm{QCoh}(X)$ restricts to a t-structure on 
$\mathrm{Perf}(X)$. Now by \cite[Corollary~C.6.3.3]{SAG} it suffices to show that for any non-zero 
$F\in \mathrm{QCoh}^\mathrm{cn}$, there is a non-zero map $P \to F$ for $P$ a connective perfect complex. 
Since $X$ satisfies perfect generation, there is a filtered diagram $P_{(-)}:I \to \mathrm{Perf}(X)$ with 
\[
\colim_I P_i \simeq F.
\]
Now since the t-structure is compatible with filtered colimits and since $\tau_{\le -1} F = 0$, we also have an 
equivalence 
\[
\colim_I \tau_{\ge 0}P_i \simeq F
\]
proving connective perfect generation.
\end{enumerate}

We expect that one has connective perfect generation in the generality of Example~\ref{exm:descent_cg} and also that 
the classicality assumption can be dropped there.
\tqed \end{exm}

\subsection{The local \texorpdfstring{$c$-structure}{c-structure} for stacks}

\begin{cnstr}\label{cnstr:weight_structure_onstacks}
Let $X$ be a quasi-geometric stack satisfying perfect generation. Denote by $\mathrm{Perf}(X)_{\le 0}$ the subcategory of 
perfect complexes of Tor-amplitude $\le 0$. This generates a weight structure on $\mathrm{QCoh}(X)$ called the {\bf local 
weight structure}.
\end{cnstr}

Since $X$ is of finite cohomological dimension, we have the inclusion
\[
	\mathrm{Perf}(X) \subset \bigcup\limits_{n \in \NN} \mathrm{QCoh}(X)_{w\ge -n},
\]
so in particular $\mathrm{Perf}(X)_{\le 0}$ gives a pre-$c$-category structure for $\mathrm{Perf}(X)$. 
Our goal is to show that in certain situations $\mathrm{Perf}(X)_{\le 0} \subset \mathrm{Perf}(X)$ is a $c$-structure. 
If $X$ is an affine scheme, we already know that it is, and in general we want to apply Theorem~\ref{thm:cdescent} to 
descend this $c$-structure along the pullback of an fppf cover $\Spec R \to X$. 
For this we need to check two things: that $\mathrm{Perf}(X) \to \mathrm{Perf}(R)$ is descendable in the sense of 
Definition~\ref{dfn:descendability} and that the weight connectivity assumption is satisfied. The first one is 
given by Proposition~\ref{prop:descendability_stacks} below and the second one by 
Proposition~\ref{prop:stacky_cstructure}.


\begin{prop}\label{prop:descendability_stacks}
Let $f:Y \to X$ be a quasi-affine fppf map of quasi-geometric stacks satisfying perfect generation. 
Then the functor 
\[
f^*: \mathrm{Perf}(X) \to \mathrm{Perf}(Y) 
\]
is descendable of exponent $\mathrm{cd}(Y)+3$.
\end{prop}
\begin{proof}
The fact is proved in \cite[Theorem~7.1]{hall2022remarks} for classical algebraic stacks. Although the argument is pretty 
similar to the one in there (specifically, to the ``proof of (4)'' on page 19), we include it here for the reader's 
convenience. 

Under the perfect generation assumption, we have equivalences
\[
\Ind(\mathrm{Perf}(Y)) \simeq \mathrm{QCoh}(Y) \quad\quad  \Ind(\mathrm{Perf}(X)) \simeq \mathrm{QCoh}(X).
\]
Since $f$ is quasi-affine, by \cite[Proposition~6.3.4.6]{SAG} $f^*$ is equivalent to 
\[
\mathrm{QCoh}(X) \to \Mod_{f_*\mathcal{O}_Y}(\mathrm{QCoh}(X)). 
\]
It now suffices to show that the algebra $f_*\mathcal{O}_Y$ is descendable in the sense of Mathew 
(Remark~\ref{rmk:mathews_descendability}). Consider the fiber sequence 
\[
K \to \mathcal{O}_X \to f_*\mathcal{O}_Y.
\]
By \cite[Lemma~11.20]{Bhatt2016} it suffices to show that there exists $m$ such that 
\[
K^{\otimes m} \to \mathcal{O}_X
\]
is null-homotopic. 

Consider a flat map $p:\Spec R \to X$. 
Since $p^*f_*\mathcal{O}_Y$ is a finite limit of flat $R$-algebras, we see that 
$p^*f_*\mathcal{O}_Y$ has tor-amplitude $\le 0$. This implies in particular that 
$p^*\Sigma K$ has tor-amplitude $\le 1$, but it in fact has tor-amplitude $\le 0$: 
indeed, since $f$ is faithfully flat, the map 
\[
R \otimes M \to p^*f_*\mathcal{O}_Y \otimes M
\]
is a $\pi_0$-inclusion for any $M \in \Mod(R)^{\heartsuit_t}$. This also implies that $p^*(\Sigma K)^{\otimes k}$ has tor-amplitude $\le 0$ for any positive integer $k$. 
Since $f$ is finitely presented, $p^*f_*\mathcal{O}_Y$, $p^*\Sigma K$ and $p^*(\Sigma K)^{\otimes k}$  
are countably presented modules. By Lazard theorem, we have that $p^*(\Sigma K)^{\otimes k}$ is equivalent to a countable 
colimit $\colim P_i$ with $P_i \in \mathrm{Perf}(R)_{w\le 0}$. Such a countable colimit itself belongs to 
$\mathrm{Mod}(R)_{w\le 1}$. 
This means that $\map(p^*(\Sigma K)^{\otimes k}, p^*\mathcal{O}_X)\simeq \Sigma^k\map(p^*K^{\otimes k}, p^*\mathcal{O}_X)$ is 
$k-1$-connective. Our goal is to show that $\map(K^{\otimes k}, \mathcal{O}_X)$ is also $1$-connective, so the map 
$K^{\otimes k} \to \mathcal{O}_X$ is forced to be zero. We choose $k = 3 + \mathrm{cd}(X)$. 

Using countable presentability of $K^{\otimes k}$ and perfect generation of $X$, 
choose a countable sequence $P_{(-)}:I \to \mathrm{Perf}(X)$ such that $\colim_I P_i \simeq K^{\otimes k}$, so that 
we have an equivalence 
\[
\tau_{\le 0}\map(K^{\otimes k}, \mathcal{O}_X) \simeq \tau_{\le 0}\lim_I \map(P_i, \mathcal{O}_X).
\]
Since $I$ is countable, $\lim$ sends sequences of $2$-connective objects to $1$-connective objects we have an equivalence
\[
\tau_{\le 0}\lim_I \map(P_i, \mathcal{O}_X) \simeq \tau_{\le 0}\lim_I \tau_{\le 1}\map(P_i, \mathcal{O}_X).
\]
We can further rewrite it as 
\begin{gather*}
\tau_{\le 0}\lim_I \tau_{\le 1}\mathrm{R}\Gamma(X,(P_i^\vee)^\mathrm{fpqc}) \simeq \tau_{\le 0}\lim_I \tau_{\le 1}\mathrm{R}\Gamma(X,\tau_{\le 1 + \mathrm{cd}(X)}(P_i^\vee)^\mathrm{fpqc}) \\
\simeq \tau_{\le 0}\lim_I \mathrm{R}\Gamma(X,\tau_{\le 1 + \mathrm{cd}(X)}(P_i^\vee)^\mathrm{fpqc})
\simeq \tau_{\le 0}\mathrm{R}\Gamma(X,\lim_I \tau_{\le 1 + \mathrm{cd}(X)}(P_i^\vee)^\mathrm{fpqc})
\end{gather*}
in the notation of Remark~\ref{rmk:descent_qcoh} and Notation~\ref{ntn:dual}
\footnote{
note that $\lim_I \tau_{\le 1 + \mathrm{cd}(X)}(P_i^\vee)^\mathrm{fpqc}$ is computed in fpqc sheaves and not in 
quasi-coherent sheaves, because the equivalence $\map(K^{\otimes k}, \mathcal{O}_X) \simeq \lim_I \map(P_i, \mathcal{O}_X)$ 
only extends to an equivalence in fpqc sheaves.
}, since 
$\mathrm{R}\Gamma(X,\tau_{\ge 2 + \mathrm{cd}(X)}F)$ is $2$-connective for any quasi-coherent sheaf $F$. 
Now it suffices to show that sheaf $\lim_I \tau_{\le 1 + \mathrm{cd}(X)}(P_i^\vee)^\mathrm{fpqc}$ vanishes. For this 
we just need to check that locally over a flat morphism $p:V=\mathrm{Spec} R \to X$. 
\[
\mathrm{R}\Gamma(V,\lim_I \tau_{\le 1 + \mathrm{cd}(X)}(P_i^\vee)^\mathrm{fpqc}) \simeq 
\lim_I \tau_{\le 1 + \mathrm{cd}(X)}\map(p^*P_i,\mathcal{O}_V).
\]
We have a fiber sequence 
\[
\lim_I \tau_{\ge 2 + \mathrm{cd}(X)}\map(p^*P_i,\mathcal{O}_V) \to \lim_I \map(p^*P_i,\mathcal{O}_V) \to 
\lim_I \tau_{\le 1 + \mathrm{cd}(X)}\map(p^*P_i,\mathcal{O}_V) 
\]
where the middle term is equivalent to $\map(p^*(K^{\otimes k}),\mathcal{O}_V)$ and so is $2+\mathrm{cd}(X)$-connective by 
the observation in the first part of the proof. On the other hand, the left term in the sequence is 
$1 + \mathrm{cd}(X)$-connective. Hence the right term is both $2 + \mathrm{cd}(X)$-connective and 
$1 + \mathrm{cd}(X)$-truncated, and is thus zero.
\end{proof}

\begin{lem}\label{lem:description_positivew}
Let $X$ be a quasi-geometric stack satisfying perfect generation. The connective part of the weight structure on 
$\mathrm{QCoh}(X)$ admits the following description:
\[
\mathrm{QCoh}(X)_{w\ge 0} = \{ x \in \mathrm{QCoh}(X) | \text{for any } y\in\mathrm{Perf}(X)_{\ge 0} \map(\mathcal{O}_X, y \otimes x) \text{ is connective } \}.
\]
Here $\mathrm{Perf}(X)_{\ge 0}$ denotes the subcategory of perfect complexes of Tor-amplitude $\ge 0$.
\end{lem}
\begin{proof}
For an object $x\in \mathrm{QCoh}(X)$ to be weight connective means that for all $y\in \mathrm{Perf}_{\le 0}$ the spectrum
\[
\map(y,x)
\] 
is connective. By duality and since $(\mathrm{Perf}(X)_{\le 0})^\vee = \mathrm{Perf}(X)_{\ge 0}$, this is equivalent to 
connectivity of the spectrum 
\[
\map(\mathcal{O}_X, y\otimes x)
\]
for all $y\in \mathrm{Perf}(X)_{\ge 0}$. 
\end{proof}

\begin{cor}\label{cor:alternative_description_positivew}
Let $X$ be a quasi-geometric stack satisfying connective perfect generation. Then the connective part of the weight 
structure on $\mathrm{QCoh}(X)$ admits the following description:
\[
\mathrm{QCoh}(X)_{w\ge 0} = \{ x \in \mathrm{QCoh}(X) | \text{for any } y\in\mathrm{QCoh}(X)_{\ge 0} \map(\mathcal{O}_X, y \otimes x) \text{ is connective. } \}
\]
\end{cor}
\begin{proof}
Since $\mathcal{O}_X$ is compact and the tensor product commutes with colimits, the condition of connectivity of 
$\map(\mathcal{O}_X, y \otimes x)$ for all $y\in \mathrm{Perf}_{\ge 0}$ automatically implies (and is equivalent to) the 
connectivity of $\map(\mathcal{O}_X, y \otimes x)$ for all $y\in \Ind(\mathrm{Perf}_{\ge 0})$. If $X$ satisfies connective 
compact generation this is exactly the condition in the statement.
\end{proof}

\begin{prop}\label{prop:stacky_cstructure}
Let $X$ be a quasi-geometric stack satisfying perfect generation. 
Let $f:Y \to X$ be a quasi-affine morphism such that additionally 
$f_*$ sends $\mathrm{Perf}(Y)_{\ge 0}$ to $\Ind(\mathrm{Perf}(X)_{\ge -d})$ for some $d \in \mathbb{N}$. Then then the base change functor 
\[
f^*:\mathrm{QCoh}(X) \to \mathrm{QCoh}(Y)
\]
sends weight connective objects to weight $-d$-connective objects. In particular, if $d=0$ then then $f^*$ is weight exact. 

The additional assumption is automatic if $X$ satisfies connective perfect generation, and in this case one can choose 
$d=0$ whenever $f$ is affine. 
\end{prop}
\begin{proof}
We use the description of the weight connective objects from Lemma~\ref{lem:description_positivew}. 
In particular, it yields that for any $x \in \mathrm{QCoh}(X)_{w\ge 0}$ and any $t \in \Ind(\mathrm{Perf}(X)_{\ge -d})$ 
\[
\map(\mathcal{O}_X,t \otimes x)
\]
is $-d$-connective. What we need to show is that for any 
$y\in \mathrm{Perf}(Y)_{\ge 0}$ the spectrum
\[
\map(\mathcal{O}_Y, y \otimes f^*x)
\]
is connective. Consider the commutative diagram 
\[
\begin{tikzcd}
Y \arrow[r,"f"]\arrow[dr,"\pi_Y"] & X\arrow[d,"\pi_X"]\\
& \Spec(\SP).
\end{tikzcd}
\]
Using the description of the global sections as the pushforward along the structure map and the projection formula, 
we have equivalences:
\[
(\pi_X)_*(f_*y \otimes x) \simeq (\pi_X)_*f_*(y \otimes f^*x) \simeq (\pi_Y)_*(y \otimes f^*x) \simeq \map(\mathcal{O}_Y, y \otimes f^*x).
\]
It suffices to show that 
\[
\map(\mathcal{O}_X, f_*y \otimes x) \simeq (\pi_X)_*(f_*y \otimes x)
\]
is $-d$-connective. This follows since $f_*y \in \Ind(\mathrm{Perf}(X)_{\ge -d})$. 

Note that there exists $d$ (which can be chosen to be $0$ if $f$ is affine) such that $f_*y$ is $-d$-connective for 
$y \in \mathrm{QCoh}(Y)^{\mathrm{cn}}$ by Proposition~\ref{prop:basechange}(3), so if $X$ satisfies connective compact 
generation then the extra assumption is redundant.
\end{proof}


\begin{thm}\label{thm:stacky_cstr_2}
Let $X$ be a quasi-geometric stack that satisfies connective perfect generation. 
Then 
$\mathrm{Perf}(X)$ with its pre-$c$-structure of \Cref{cnstr:weight_structure_onstacks} is a $c$-category. 
\end{thm}
\begin{proof}
Consider an fppf cover by an affine scheme $\Spec R \to X$. 
Since $X$ has quasi-affine diagonal, this map is quasi-affine. 
Now the base change functor 
\[
\Ind(\mathrm{Perf}(X)) \to \Ind(\mathrm{Perf}(\Spec R))
\]
preserves weight coconnectivity and by Proposition~\ref{prop:stacky_cstructure} sends weight connective objects to 
$-d$-connective objects where $d$ is the cohomological dimension of $f$. 
Additionally, it is descendable of exponent $m = \mathrm{cd}(X) + 3$ by 
\Cref{prop:descendability_stacks} (see also \cite[Theorem~7.1]{hall2022remarks}). 
Since $\mathrm{Perf}(\Spec R)$ is a bounded weighted category (Example~\ref{exm:weight_conn_ring}), 
it is also a $c$-category of width $\le 0$, so we can apply Theorem~\ref{thm:cdescent}.
\end{proof}

\begin{rmk}\label{rmk:width_stk_ccat}
Pick an fppf cover $f:\Spec R \to X$. The $c$-structure on $\mathrm{Perf}(X)$ in Theorem~\ref{thm:stacky_cstr_2} is of width 
$\le m(d+1)$, where $d$ is the cohomological dimension of $f$ and $m$ is the descendability exponent of $f$.
\tqed \end{rmk}

Our next goal is to compare the locally defined properties of morphisms with the respective notions for $c$-categories. 
We say that an affine morphism $f:Y\to X$ of stacks is $k$-connective or a nilpotent extension if it remains $k$-connective 
after pulling back along any fppf cover $\Spec R \to X$ (see the explanation in the beginning of 
Subsection~\ref{subsec:stacks}).

\begin{lem}\label{lem:tensor_qcoh}
Let $X$ be a quasi-geometric stack satisfying connective perfect generation. 
For any $x \in \mathrm{QCoh}(X)^{\mathrm{cn}}$ the functor 
\[
x\otimes -: \mathrm{QCoh}(X) \to \mathrm{QCoh}(X)
\]
preserves weight connective objects.
\end{lem}
\begin{proof}
This follows immediately from the description of Corollary~\ref{cor:alternative_description_positivew}.
\end{proof}


\begin{thm}\label{thm:morphisms_of_stacks}
Let $f:Y \to X$ be a quasi-affine morphism of quasi-geometric stacks that satisfy connective perfect generation. 
\begin{enumerate}
\item If $f$ is affine, then $f^*: \mathrm{Perf}(X) \to \mathrm{Perf}(Y)$ is a 0-affine morphism of $c$-categories.
\item If $f$ is $k$-connective and affine for some integer $k$, then $f^*: \mathrm{Perf}(X) \to \mathrm{Perf}(Y)$ is a 
$k$-connective functor of $c$-categories.
\item If $f$ is a nilpotent extension, then $f^*: \mathrm{Perf}(X) \to \mathrm{Perf}(Y)$ is a weak nilpotent extension of 
$c$-categories.
\end{enumerate}
\end{thm}
\begin{proof}
Claim $(1)$ is the statement of Proposition~\ref{prop:stacky_cstructure}. 

To prove the other two claims choose an fppf cover $\pi:U=\Spec R \to X$ and consider the diagram 
\begin{equation}\label{eq:adjointable}
\begin{tikzcd}
\mathrm{QCoh}(X) \arrow[r,"\pi^*"]\arrow[d, "f^*"] & \mathrm{QCoh}(U)\arrow[d,"\bar{f}^*"]\\
\mathrm{QCoh}(Y) \arrow[r, "\bar{\pi}^*"] & \mathrm{QCoh}(Y \times_X U).
\end{tikzcd}
\end{equation}
By Proposition~\ref{prop:basechange}(2) the square is vertically right adjointable, so in particular we have an equivalence
\[
\pi^*\mathcal{O}_Y \simeq \mathcal{O}_{Y \times_X U}
\]
of quasi-coherent sheaves on $U$. 
In particular, since connectivity in $\mathrm{QCoh}(X)$ can be checked locally, the map of $\mathrm{QCoh}(X)$-algebras 
\[
\mathcal{O}_X \to \mathcal{O}_Y
\] 
is $k$-connective if $f$ is. 
From this and Lemma~\ref{lem:tensor_qcoh} it follows that the endofunctor 
\[
\mathcal{O}_Y\otimes -: \mathrm{QCoh}(X) \to \mathrm{QCoh}(X)
\]
is a $k$-connective algebra in the associated t-structure (see Construction~\ref{cnstr:associated_tstr}). 

We will now prove assertion (3). 
Note that if the map $Y \to X$ is a nilpotent extension of rank $\le k$, then by the above argument the functor 
$\mathrm{Perf}(X) \to \mathrm{Perf}(Y)$ is at least $0$-connective. Denote by $m$ the exponent of descendability of $\pi$ and 
by $d$ the cohomological dimension of $\pi$. 
We are going to show that it is a nilpotent extension of rank $\le k(md+m+d)(m+1)$. For the ease of notation set $r=md+m+d$. Consider a composable sequence of morphisms
\begin{equation}\label{eq:composite2}
x_0 \stackrel{g_1}\to x_1 \stackrel{g_2}\to \dots \stackrel{g_{kr(m+1)}}\to x_{kr(m+1)}
\end{equation}
such that $x_0 \in \mathrm{QCoh}(X)_{w\le 0}$, $x_{kr(m+1)} \in \mathrm{QCoh}(X)_{w\ge 0}$ and $f^*(g_i) \simeq 0$ for all $i$. It suffices to show that the composite is $0$. 
Recall from Lemma~\ref{lem:adjacent_descendable}(2) that $x_{kr(m+1)}$ is a summand of an object $z$ admitting an 
$m+1$-step filtration $F_jz$ with $F_0z=0$ $1\le j\le m+1$ 
whose graded pieces are of the form $\pi^*y_j$ with 
$y_j \in \mathrm{Qcoh}(U)_{w\ge -md-m +(d+1)(j-1) -d} \subset \mathrm{Qcoh}(U)_{w\ge -r}$. We have 
$F_0x_{kr(m+1)} = 0$. 
Using decreasing induction on $j$ we will show that the composite $x_{krj} \to x_{kr(m+1)}$ in the sequence 
(\ref{eq:composite2}) factors as   
\[
x_{krj}\to F_jx_{kr(m+1)} \to x_{kr(m+1)}. 
\]
The case $j=0$ then yields the desired claim. The base of the induction is tautological. 
For the inductive step assume that we know the claim for a given $j$, consider the commutative diagram
\[
\begin{tikzcd}
x_{kr(j-1)} \arrow[r, "g_{kr(j-1)+1}"] & x_{kr(j-1)+1} \arrow[r,"g_{kr(j-1)+2}"] & \dots\arrow[r, "g_{krj}"] & x_{krj}\arrow[d]\\
&& F_{j-1}x_{kr(m+1)}\arrow[r] & F_jx_{kr(m+1)}\arrow[r] & \pi^*y_j.
\end{tikzcd}
\]
Our goal is to construct a map $x_{kr(j-1)} \to F_{j-1}x_{kr(m+1)}$ making the diagram commutative. 
The bottom row is a fiber sequence, so it suffices to show that the composite 
\[
x_{kr(j-1)} \stackrel{g_{kr(j-1)+1}}\to x_{kr(j-1)+1} \stackrel{g_{kr(j-1)+2}}\to \dots \to x_{krj}\to \pi^*y_j
\]
is trivial. By adjunction it suffices to show that the composite of $kr+1$ maps
\[
\pi^*x_{kr(j-1)} \stackrel{\pi^*g_{kr(j-1)+1}}\to \pi^*x_{kr(j-1)+1} \stackrel{\pi^*g_{kr(j-1)+2}}\to \dots \stackrel{\pi^*g_{krj}}\to \pi^*x_{krj}\to y_j
\]
is trivial. Since the square (\ref{eq:adjointable}) is commutative we obtain  that $\bar{f}^*(\pi^*g_i) \simeq 0$ for any 
$kr(j-1)+1 \le i < krj$. Taking the composite of the last two maps in the sequence, we obtain a composable 
sequence of $kr$ morphisms satisfying the assumptions of Corollary~\ref{cor:higher_nilpotence} (up to a shift of the 
sequence; note that $y_j \in \mathrm{QCoh}(U)_{w\ge -r}$ and $\pi^*x_{kr(j-1)} \in \mathrm{QCoh}(U)_{w\le 0}$). 
Hence, the composite is $0$ as desired. 
\end{proof}

\begin{rmk}\label{rmk:general_nilpotence_descent}
One might be tempted to ask whether Theorem~\ref{thm:morphisms_of_stacks}(3) holds in the generality of arbitrary $\Ind$-right 
adjointable diagrams of $c$-categories
\[
\begin{tikzcd}
	\cC \arrow[r, "{\pi}"]\arrow[d,"f"] & \cU \arrow[d,"{\bar{f}}"] \\
	\cD \arrow[r, "{\bar{\pi}}"] & \cV.
\end{tikzcd}
\]
where $\pi$ and $\bar{\pi}$ are descendable functors satisfying the assumptions of \Cref{thm:cdescent} and 
$\bar{f}$ is a nilpotent extension. Indeed, the proof of the nilpotence condition works in this generality, however, 
the proof of $0$-connectivity uses the specific properties of the local weight structure on stacks. 
In particular, in a similar situation from \Cref{exm:introexample} we have to employ a slightly different strategy to 
prove that $\mathrm{Perf}((\ZZ/p^n)^{hG}) \to \mathrm{Perf}((\FF_p)^{hG})$ is a nilpotent extension. 
\tqed \end{rmk}

\begin{rmk}\label{rmk:fpqc_covers_stacks}
Theorem~\ref{thm:stacky_cstr_2} remains valid if one only assumes the map 
$\Spec R \to \mathcal{F}$ to be representable and faithfully flat in Definition~\ref{dfn:stk}. 
We briefly sketch the argument here. As discussed in \Cref{rmk:phantom_descent}, 
\Cref{thm:cdescent} holds under a weaker assumption that $f$ is a {\it phantom} descendable functor 
(see \Cref{dfn:phantom_descendable}). On the other hand, any faithfully flat map of rings is phantom descendable, so arguing as in 
\Cref{prop:descendability_stacks} one can show that any faithfully flat quasi-affine pullback of stacks 
is descendable. The rest of the argument goes through in exactly the same way. Similarly, the proofs of 
\Cref{thm:morphisms_of_stacks}(1,2) generalize to this setting. 

On the contrary, the proof of \Cref{thm:morphisms_of_stacks}(3) does not work in this generality and we do not know whether 
the statement is true. 
\end{rmk}

\subsection{Localizing invariants of algebraic stacks}

We are now in a good position to apply the results of Section~\ref{sec:localizing_invariants_of_c_categories} to 
the $c$-categories constructed in Theorem~\ref{thm:stacky_cstr_2}. 

For instance, an immediate consequence of Corollary~\ref{cor:width_n_thh} is the following. 

\begin{cor}\label{cor:thh_bound}
Let $X$ be a quasi-geometric stack satisfying connective perfect generation. 
Then $\mathrm{THH}(X)$ is bounded below.
\end{cor}

Theorem~\ref{thm:truncating_invariants} together with Theorem~\ref{thm:morphisms_of_stacks}(3) imply the following 

\begin{cor}\label{cor:dgm_general}
Let $f:Y \to X$ be a nilpotent extension of quasi-geometric stacks satisfying connective perfect generation and $E$ a truncating localizing 
invariant. Then the map 
\[
E(X) \stackrel{f^*}\to E(Y)
\]
is an equivalence.
\end{cor}

In particular, applying the previous corollary to $E$ being the fiber of the cyclotomic trace
\[
	\K^{\mathrm{inv}}(-):= \fib(\K(-) \to \TC(-)) 
\]
we obtain the following variant of the DGM theorem for stacks. 

\begin{cor}\label{cor:dgm}
Let $f:Y \to X$ be a nilpotent extension of quasi-geometric stacks satisfying connective perfect generation. Then 
the square
\[
\begin{tikzcd}
\K(X)\arrow[r]\arrow[d, "{f^*}"] & \TC(X)\arrow[d, "{f^*}"]\\
\K(Y)\arrow[r] & \TC(Y)
\end{tikzcd}
\]
is a pullback square.
\end{cor}

We may also apply Corollary~\ref{cor:dgm_general} to $E=\mathrm{KH}$ and $E=\mathrm{HP}$ to get the following. 
\begin{cor}\label{cor:HP}
The functors $\mathrm{HP}(-)$ and $\mathrm{KH}(-)$ are nilinvariant on quasi-geometric stacks satisfying connective perfect generation. 
\end{cor}

\addtocontents{toc}{\protect\setcounter{tocdepth}{1}}
\section{Counterexamples}\label{sec:counterexamples}
The goal of this section is to provide some examples and questions about the boundaries of some of the notions studied in this paper. To highlight some of the most interesting examples:

\begin{enumerate}
	\item In \Cref{cvsheart}, we give an example of a $c$-category whose weight structure doesn't come from a heart category.
	\item In \Cref{exm:weaklydescendable}, we give an example of a functor that satisfies the conditions of descendability 
	object-wise, but isn't descendable.
	\item In \Cref{thm:localizingsubcat}, we show that the localizing motive of the ring $\QQ^{S^2}$ is not in the (stable) 
	localizing subcategory generated by motives of connective rings, and that the DGM theorem does not hold for 
	the map $\QQ^{S^2}[x_2] \to \QQ^{S^2}$.
\end{enumerate}

We begin by giving an example of a nilpotent extension of heart categories that is an isomorphism for all truncating 
invariants, but is not covered by Theorem \ref{thm:truncating_invariants}. 

\begin{exm}\label{exm:truncatingnotcovered}
	 Considering $\FF_p$ as a module over $\FF_p^{hC_p}$, we can form the trivial square zero extension $\Mod(\FF_p^{hC_p} \oplus \FF_p) \to \Mod(\FF_p^{hC_p})$. To see that this map is an equivalence on truncating invariants, we have a map of localization sequences
	 
	 \begin{center}
	 	\begin{tikzcd}
	 			 \ar[d]	\Mod(\FF_p[C_p]\oplus \FF_p[C_p]^{\oplus p})\ar[r]&\ar[d] \Mod(\FF_p^{hC_p}\oplus \FF_p )\ar[r]&\ar[d]\Mod(\FF_p^{tC_p})\\
	 	\Mod(\FF_p[C_p])	\ar[r] &\Mod(\FF_p^{hC_p}) \ar[r] & \Mod(\FF_p^{tC_p})

	 	\end{tikzcd}
	 \end{center}
	 
	 Since the left vertical map is a square-zero extension and the right vertical map is an equivalence, the middle vertical map is an isomorphism on all truncating invariants.
\tqed \end{exm}

It may be interesting to try to extend Theorem \ref{thm:truncating_invariants} to include examples such as that in 
\Cref{exm:truncatingnotcovered}. We note that Dotto \cite{dotto2018dundas} proves a DGM theorem for certain trivial 
square-zero extensions of exact categories. We wonder whether this is reflective of there being a stronger notion of 
truncating invariant which most examples of truncating satisfy which include these extensions of exact categories, or whether 
these maps are isomorphisms on all truncating invariants.

\subsection{Localizations of bounded weighted categories}


Some subtlety in the theory of $c$-categories comes from the fact that examples of $c$-categories in practice are a lot more special than generic examples.

Since $c$-categories arise as kernels of localizations between cconnective categories, wild behaviour of $c$-categories can arise from unusual examples of localizations.

Recall that for any connective ring spectrum and any set $S$ of morphisms between finitely generated projective modules over 
it one may define the localization $R \to R[S^{-1}]$ which universally inverts those morphisms, and $R[S^{-1}]$ is again going 
to be a connective ring spectrum. More generally, this works in the context of bounded weighted categories: given a 
subcategory of a bounded weighted category $\cD \subset \cC$ generated under finite limits and colimits by a set of objects of 
weight amplitude $[0,1]$, the localization $\cC/\cD$ admits a weight structure, such that the functor $\cC \to \cC/\cD$ is weight 
exact. 
One may wonder if the converse holds. If it did, studying $c$-categories (of width $\le 1$) would be reduced to studying such 
subcategories. The next example shows that it can fail: we construct a weight exact localization, whose kernel is 
not generated by objects of weight amplitude $[0,n-1]$.


\begin{exm}\label{exm:infinitecoordaxes}
Let us use $\Mod^{\ZZ^n}(-)$ to denote the category of $\ZZ^n$-graded modules over a $\ZZ^n$-graded ring spectrum. 
A slight generalization of Example~\ref{exm:weight_conn_ring} shows that $\Mod^{\ZZ^n}(A)^\omega$ is naturally a bounded 
weighted category when $A$ is a connective ring spectrum with the heart generated by graded shifts of $A$ as a module over 
itself. 

	Let $k$ be a field, and consider the $\ZZ^n$-graded ring $R_n = k[x_1,\dots,x_n]/(x_ix_j,i\neq j)$, where $k$ is a field and $x_i$ has grading generating the $i$-th coordinate. 
	By localizing away from the tensor ideal generated by $\otimes_{1}^n\cof(x_i)$, we obtain $\Mod^{\ZZ^n}(\prod_1^nk[x_i^{\pm1}])$. The ring $\prod_1^nk[x_i^{\pm1}]$ is connective, so we obtain a weight exact localization functor 
	\[
	L:\Mod^{\ZZ^n}(R_n)^\omega \to \Mod^{\ZZ^n}(\prod_1^nk[x_i^{\pm1}])^\omega.
	\]
	We claim that any connective object $M$ in the kernel of this localization with weight amplitude $[0,n-1]$ vanishes. This means in particular that $L$ is a weight exact localization of categories with bounded weight structure such that the kernel is not generated by objects of weight amplitude $[0,n-1]$. By taking a direct sum of these localizations over all $n\geq0$, we obtain such a localization where the kernel is not generated by objects of uniformly bounded weight amplitude.
	
	We now justify our claim. Because the ring $R_n$ is an iterated pullback, $\Mod^{\ZZ^n}(R_n)^{\omega}$ is a weight exact fully faithful subcategory of the iterated pullback category $$P_n=\Mod^{\ZZ^n}(k[x_1])^{\omega}\times_{\Mod^{\ZZ^n}(k)^{\omega}}\Mod^{\ZZ^n}(k[x_2])^{\omega}\times\dots \times_{\Mod^{\ZZ^n}(k)^{\omega}}\Mod^{\ZZ^n}(k[x_n])^{\omega}.$$
	Let $M$ be a connective nonzero object in $P_n$ that is $x_i$-torsion for each $i$. Let us use $M'$ to denote the image of $M$ in $\Mod^{\ZZ^n}(k)$, which is nonzero because $M$ is nonzero $x_i$-complete for each $i$. To finish the proof, it is enough to show that $\pi_kM'\neq0$ for some $k\geq n$.
	
	An $x_i$-torsion object of $\Mod^{\ZZ^n}(k[x_i])^{\omega}$ can be written as a finite direct sum of $k[x_i]/x_i^k$ for some $k$, shifted into some multigrading. It follows that for any nonzero homogeneous element $z \in \pi_*M'$, and any $1 \leq i \leq n$, $z$ is either the target of a $x_i$-Bockstein differential, from a class that lifts to a nonzero class in $\pi_{*+1}M'$, or the source of a  $x_i$-Bockstein differential to a class that lifts to a nonzero class in $\pi_{*+1}M'$.
	
	Starting with a nonzero homogeneous element $z_0 \in \pi_*M'$, we follow an algorithm which at each step takes in a nonzero homogeneous class $z_i\in\pi_*M'$ and outputs a nonzero homogeneous class $z_{i+1} \in \pi_{*\pm1}M'$ if $*<n$. We will show that the algorithm terminates by outputting a nonzero homogeneous element in $\pi_kM'$ for $k\geq n$, which will complete the proof.
	
	At each step of the algorithm, if our class $z_i$ is in $\pi_jM'$, for $j\geq n$, the algorithm terminates. Otherwise, we can assume that $0\leq j < n$. If $z_i$ is the source of a $x_{j+1}$-Bockstein differential, $z_{i+1}$ be a representative of the differential, and otherwise let it be a representative of the target of the differential.
	
	
	We can see this algorithm terminates as follows: we lexicographically order the multidegrees in $\ZZ^n$, and note that whenever the homotopy degree of the $z_{i+1}$ is higher than that of $z_{i}$, the multigrading with respect to this ordering is higher than for all $z_j$ for $j<i$. Thus the algorithm must terminate since there are only finitely many nonzero multigradings for which the homotopy groups of $M'$.
\tqed \end{exm}

Another fact about localizations of commutative connective $\EE_1$-rings is that they are flat. 
In the absence of commutativity, this fails. Many examples were constructed in \cite{NeemanRanickiSchofield2004}. Without 
claiming originality, we will also give a simple example below.

\begin{exm}\label{exm:notflat}
	Here we give an example of a localization of connective $\EE_1$-rings that is not flat.
	Let $k$ be a field, and consider the symmetric monoidal $k$-linear category $\cC$ which is defined as the pullback
	
	\begin{center}
		\begin{tikzcd}
			\cC	\ar[r]\ar[d] &\Mod(k[x_2]) \ar[d]\\
		\Mod(k[x_{2}^{-1}])	\ar[r] & \Mod(k[x_2^{\pm1}] \cong k[x_{-2}^{\pm1}]).
		\end{tikzcd}
	\end{center}
	$\cC$ should be thought of as `$\mathrm{Perf}(\PP^1_{\pm2})$', where $\PP^1_{\pm2}$ is like $\PP^1$, but built out of an affine line in topological degree $2$. 
	There is an invertible object $\cO(n) \in \cC$ that is given by the triple 
	\[
		(k[x_2],x_{-2}^{-n}k[x_{-2}],k[x_2^{\pm1}]]\cong x_{-2}^{-n} k[x^{\pm}_{-2}]),
	\]
	where the isomorphism sends $x_2$ to $x_{-2}^{-1}$. There is a natural map $y:\cO \to \cO(1)$ which is given by the inclusion $k[x_{-2}] \to x_{-2}^{-1}k[x_{-2}]$, and a map $x:\Sigma^2\cO \to \cO(1)$ which is given by multiplication by $x_2$ and $x_{-2}^{-1}$. One can check that $\cO$ and $\cO(1)$ generate the category, that $\End(\cO\oplus \cO(1))$ is the matrix ring 
	
	\begin{center}
	\[
		\begin{pmatrix}
			k & k \oplus \Sigma^2 k\\
			0 & k
		\end{pmatrix}
	\]
	\end{center}
	We can consider the localization away from $\cof(y)$. In this localization, $\cO$ is sent to the colimit 
	$\cO \to \cO(1) \to \cO(2) \to\dots$, where each map is given by multiplication by $y$. This colimit is the triple $
	(k[x_2],k[x_{-2}^{\pm1}], \cong)$, whose endomorphism ring is easily seen to be $k[x_2]$. It follows that the localization 
	of $R$ away from $\cof(y)$ is $M_2(k[x_2])$, which is connective, but clearly not flat over $R$.

	In particular, if we consider the kernel $K$ of the localization $\cC \to \Mod(k[x_2])$, with inclusion $i:K \to \cC$, 
	then the functor $ii^R$ doesn't send weight coconnective objects to weight bounded above objects.
\tqed
\end{exm}
%

We finally pose two questions, which we both expect to have negative answers to, but do not have examples.

\begin{qst}\label{question:kerlocccat}
	Does there exist a compactly generated kernel of a localization of bounded weighted categories that doesn't admit a $c$-structure?
\end{qst}

\begin{qst}\label{question:width0}
	Does the weight structure on $\Ind$ of a $c$-category of width 0 restrict to compact objects?
\end{qst}

\subsection{Comparison with fcd heart categories}

Another notion related to that of a $c$-category, is that of an fcd heart category. One may hope to build our whole theory using 
them. However, it turns out that their many drawbacks in such approach: 
\begin{itemize}
	\item many examples naturally arise as $c$-categories, but not heart categories (see \Cref{thm:stacky_cstr_2});
	\item the notion of morphisms of heart categories are not compatible with resolutions of Theorem~\ref{cor:resolution_full} (see \Cref{wrn:preserving_injectives}).
\end{itemize}

We can use \Cref{exm:infinitecoordaxes} to write a $c$-category whose weight structure does not come from a heart structure. 

\begin{exm}\label{cvsheart}
	Let $K_n$ be the kernel of the localization $L$ from \ref{exm:infinitecoordaxes}. We claim that $K_n$ has a $c$-structure 
	of width $1$, so that the infinite sum $\oplus_n K_n$ also does. To see this, we claim the pre-$c$-structure generated by 
	the dual of $\otimes_{1}^n\cof(x_i)$ is such a $c$-structure. By construction, an object $M \in K_n$ of the weight 
	structure is connective iff $M\otimes \otimes_{1}^n\cof(x_i)$ is connective.
	
	Since the dual of $\otimes_{1}^n\cof(x_i)$ is coconnective in the weight structure on $K_n$, it follows that the right 
	adjoint $i^*$ of the inclusion $i_{!}:\Ind(K_n) \to \Mod^{\ZZ^n}(R_n)$ preserves connective objects. It is enough to show 
	that $i^*$ applied to connective objects generates the connective objects of $\Ind(K_n)$, since $i_!i^*$ sends connective 
	objects to $-1$-connective objects, which shows that the pre-$c$-structure is a $c$-structure of width $1$.
	To show this, we note that the double right adjoint $i_*$ is fully faithful with image the $(x_1,\dots,x_n)$-complete 
	objects, and connectivity on such objects can be checked after tensoring with $\otimes_{1}^n\cof(x_i)$. It follows that it 
	sends connective objects to connective objects: for a connective object $M \in \Ind(K_n)$ we can compute $i_*M$ as the 
	limit $\lim_{(i_1,\dots, i_n) \in \ZZ^n} M/(x_1^{i_n},\dots, x_n^{i_n})$, and we know that $M/x_i^n$ to be connective in 
	$\Mod^{\ZZ^n}(R_n)$ (in particular, the system is also Mittag-Leffler on $\pi_0$). Then the result follows since $i^*i_*=\id_{\Ind(K_n)}$, and both functors preserve connectivity. 
	
	Finally, we claim that the weight structure on $\oplus_nK_n$ does not come from a fcd heart structure. Indeed, if there 
	was, then the compact objects would be generated by objects that are $[-m,0]$ bounded in the weight structure for some 
	$m$, but this is not true by \Cref{exm:infinitecoordaxes}.
\tqed
\end{exm}


 Given any fcd heart category, we may saturate it using \Cref{prop:cstr_heartstructure} and get a different 
heart structure on the same category. One may hope that exact functors between such saturated heart categories induce weight 
exact functors, but this also turns out to be false.

\begin{exm}
	We give here an example of a heart exact but not weight exact functor between heart categories coming from 
	\Cref{prop:cstr_heartstructure}. Let $k$ be a field, and consider the map of rings 
	$k\oplus \Sigma^{-1}k\epsilon_1 \to k\oplus \Sigma^{-1}(k\epsilon_1 \oplus k\epsilon_2)$, i.e the map of trivial 
	square-zero extensions in degree $-1$. The standard $t$-structure on the categories of modules over these rings restricts 
	to compact objects and is bounded by \cite[Proposition 3.9]{kcoconn}, and Postnikov towers split since the category is 
	cohomological dimension $1$. The heart of the $t$-structure is generated under extensions from the unit. From this, it is 
	easy to see that the heart structures arising from \Cref{exm:minusone} have the property that the nonnegative objects are 
	exactly the weight $\geq-1$ compact objects, so that it arises from \Cref{prop:cstr_heartstructure}.
	
	The weight hearts are generated under extensions by the unit, and so the base change functor is heart exact. It is not 
	weight exact, since the augmentation $k$ as a $k\oplus \Sigma^{-1}k\epsilon_1$-module is in the weight heart, but its base 
	change to $k\oplus \Sigma^{-1}(k\epsilon_1 \oplus k\epsilon_2)$ has $\epsilon_2$ acting nontrivially on the element $1$, 
	so it is not weight connective.
\tqed
\end{exm}

\subsection{Weakly phantom descendable functors}

The notion of a descendable functor in Definition~\ref{dfn:descendability} can be weakened in several different ways. 
One version is to require the filtrations in \Cref{prop:descendability} to be defined objectwise, but not functorially, and 
another is to require the filtrations to exist up to phantom extensions. Finally, we can require the filtrations to exist 
non-functorially and up to phantom extensions.

\begin{dfn}\label{dfn:phantom_descendable}
Let $f:\cC \to \cD$ be a functor between small stable categories. 
\begin{enumerate}
	\item We say that $f$ is {\bf weakly descendable} if for every object $x$, there exists a 
	length $m+1$ filtration on an object $y$ with associated graded in the image of $f^*$ such that $x$ is a retract of $y$.
	\item We say that $f$ is {\bf phantom descendable} if
	there is an endofunctor $X$ of $\Ind(\cC)$ that admits a length $m+1$ filtration with associated graded of the form 
	$f^*(Y_i)$ for some functors $Y_i:\Ind(\cC) \to \Ind(\cD)$ and such that there is a map $X \to Y$ in 
	$\End_{\PrL}(\Ind(\cC))$ with phantom fiber. 
	\item We say that $f$ is {\bf weakly phantom descendable} if for every object $x$, there exists a 
	length $m+1$ filtration on an object $y$ with associated graded in the image of $f^*$ and a map $x \to y$ with phantom 
	fiber. 
\end{enumerate}
\end{dfn}

A descendable functor satisfies all of the above conditions, and both weak descendability and phantom descendability imply 
weak phantom descendability. An example of a phantom descendable map which is not descendable was constructed in 
\cite[Theorem~D]{aoki2024cohomologylocallyprofinitesets}.

\begin{exm}\label{exm:weaklydescendable}
	We give an example of a weakly descendable functor with $m=0$ which is not descendable for any $m\geq0$.
	
	Consider the functor $\Mod(\FF_p) \to \Mod(\End_{\mathbb{\mathrm{BP}}}\FF_p)$ coming from the connective cover ring map, where $\mathrm{BP}$ is the Brown--Petersen spectrum. This functor has the property that every object is a retract of the image of the right adjoint, but we claim it is not descendable in the sense of \Cref{dfn:descendability}. To see this, we note that $\End_{\mathrm{BP}}\FF_p$ is the colimit of $\End_{\mathrm{BP}\langle n \rangle}\FF_p$ as $n \to \infty$, so it suffices to show that the exponent of nilpotence of the maps $\Mod(\FF_p) \to \Mod(\End_{\mathrm{BP}\langle n \rangle}\FF_p)$ tends to $\infty$.  $\Mod(\End_{\mathrm{BP}\langle n \rangle}\FF_p)$ is in the thick subcategory of $\FF_p$-bimodules generated by $\FF_p$, and so are its tensor powers. Thus we can think of these as living in the category of $\mathrm{THC}(\FF_p)$-modules, and $\End_{\mathrm{BP}\langle n \rangle}\FF_p$ becomes the object $\Hom_{\mathrm{Bimod}_{\FF_p}}(\FF_p,\End_{\mathrm{BP}\langle n \rangle}\FF_p)$. $\mathrm{THC}(\FF_p)$ is a divided power algebra over $\FF_p$ on a class $x$ in degree $-2$, and this object is the tensor product of the cofiber of $x^{(p^i)}$ for $0 \leq i\leq n$. In particular, the ideal of $\pi_*\mathrm{THC}(\FF_p)$ that acts by $0$ on this module has exponent of nilpotence $p^{n+1}$, so at least $p^{n+1}$ such modules are required to build the unit, which implies the claim. 
\tqed 
\end{exm}

As noted in Remark~\ref{rmk:phantom_descent} our main descent result (Theorem~\ref{thm:cdescent}) holds for weakly phantom 
descendable maps. However, we do not know much about other properties of these weak notions. 

\begin{qst}
	Does there exist a weakly descendable (resp. weakly phantom descendable) functor of small stable categories such that the map on opposite categories is not weakly descendable (resp. weakly phantom descendable)? 
\end{qst}

\subsection{Motives generated by \texorpdfstring{$c$-categories}{c-categories}}
We have already seen that not all small stable categories admit a $c$-structure (see \Cref{cor:width_n_thh}). However, if we are 
only interested in studying finitary localizing invariants we could hope to use our results for any small stable category 
$\cC$ whose localizing motives belongs to the full subcategory of $\Mloc$ generated by connective ring spectra. 
Here we give an example to show that it is not possible. 

Recall \cite{BGT} that $\Mloc$ is the universal category with a filtered colimit preserving localizing invariant from $\Cat^{\mathrm{perf}}$ that is denoted $\Uloc$. For a ring $R$, we will sometimes write $R$ instead of $\Uloc(R)$ or $\mathrm{Perf}(R)$, and freely use the fact that $\Uloc(-\otimes R)$ is a localizing invariant \cite[Lemma 3.12]{Land_2019}.

\begin{thm}[Efimov]\label{thm:localizingsubcat}
	$\Uloc(\QQ^{S^2})$ is not in the localizing subcategory of $\Mloc$ generated by $\Uloc(R)$ for connective rings $R$. Moreover, the fiber of the map $\K^{\inv}(\Uloc(\QQ^{S^2}[x_2]) \to \Uloc(\QQ^{S^2}))$ doesn't vanish.
\end{thm}

\begin{rmk}
	The following part of the above result about $\Uloc(\QQ^{S^2})$ not being in the localizing subcategory generated by motives of connective rings was first proven by Efimov (to appear). We present our own proof here which is similar in spirit to Efimov's proof.
\tqed \end{rmk}

\begin{rmk}
	It is much easier to prove the statement that $\Uloc(\QQ^{S^2})$ is not in the thick subcategory of $\Mloc$ generated by connective rings: this follows from the fact that $\THH(\QQ^{S^2}) = \QQ^{LS^2}$ is not bounded below.
\tqed \end{rmk}

Our strategy to prove \Cref{thm:localizingsubcat} is to construct a colimit preserving functor $\Mloc \to \Sp$ that vanishes on $\Uloc(R)$ for $R$ connective, but not on $\QQ^{S^2}$. 

\begin{prop}\label{prop:Gfunctor}
	Let $F = \K^{\mathrm{inv}}(-\otimes \QQ)$. Then $G=\fib F(C\otimes \SP[x_2]) \to F(C)$ for $C \in \Mloc$ is a functor $\Mloc \to \Sp$ that preserves colimits and vanishes on $R$ a connective $\EE_1$-ring.
\end{prop}

\begin{proof}
	It follows from the fact that $\K^{\inv}$ is a truncating invariant that $G$ vanishes on $\Uloc(R)$ for $R$ connective. So it remains to show that it preserves filtered colimits. Let $C_{\alpha}$ be a filtered diagram in $\Mloc$, so that we need to show that $$\fib(\colim G(C_{\alpha})\to G(\colim C_{\alpha}))=0$$
	
	Note that the tensor product in $\Mloc$ commutes with filtered colimits, as does $\K(-\otimes \QQ)$ and $\THH(-\otimes \QQ)_{h\mathbb{T}}$. Then since there are cofiber sequences\footnote{We use here that when the input ring is rational, $\TC$ agrees with $\TC^{-}$.} 
	$$\K^{\mathrm{inv}}(-\otimes \QQ) \to \K(-\otimes \QQ) \to \TC^-(-\otimes \QQ)$$
	$$\THH(-\otimes \QQ)_{h\mathbb{T}} \to \TC^{-}(-\otimes \QQ) \to \TP(-\otimes \QQ),$$
	Since the fiber of the assembly map of an exact functor is exact as the functor varies, we have that the fiber of interest is the desuspension of the total fiber of the square
	
	\begin{center}
		\begin{tikzcd}
			{\colim \TP(C_{\alpha}\otimes \QQ[x_2])} \ar[r]\ar[d] &\TP(\colim C_{\alpha}\otimes \QQ[x_2]) \ar[d]\\
		\colim \TP(C_{\alpha}\otimes \QQ)	\ar[r] & \TP(\colim C_{\alpha}\otimes \QQ).
		\end{tikzcd}
	\end{center}
	But the vertical maps are now equivalences since $\TP$ is $\AA^1$-invariant rationally, and the map $\QQ[x_2] \to \QQ$ is an $\AA^1$-equivalence.
%
\end{proof}

We can now proceed with the proof of \Cref{thm:localizingsubcat}:

\begin{proof}[Proof of \Cref{thm:localizingsubcat}]
By \Cref{prop:Gfunctor}, it suffices to show that $G(\QQ^{S^2}) \neq 0$. 
By applying \cite[Example 4.9]{kcoconn} we have a pullback square 
	\begin{equation}
		\begin{tikzcd}\label{eqn:pullbacksquare}
			\Uloc(\QQ^{S^2}) \ar[r]\ar[d] & \Uloc(\QQ) \ar[d]\\
		\Uloc(\QQ)	\ar[r] & \Uloc(\QQ[x_{-1}]),
		\end{tikzcd}
	\end{equation} 
so since $G$ vanishes on connective rings, it suffices to check that $G(\QQ[x_{-1}]) \neq 0$.	
	
%
	
Now there is a localization sequence 
\begin{equation}\label{eqn:locseq}
	\Uloc(\QQ\langle \epsilon_{0}\rangle) \to \Uloc(\QQ[x_{-1}]) \to \Uloc(\QQ[x^{\pm}_{-1}]).
\end{equation}
where $\QQ\langle \epsilon_{0}\rangle$ is the trivial square-zero extension of $\QQ$. 
Since $\QQ\langle\epsilon_0\rangle$ is connective, it is enough to show that $G(\QQ[x_{-1}^{\pm1}])\neq 0$, i.e that 
\[
	\K^{\mathrm{inv}}(\QQ[x_{-1}^{\pm}][t_2]) \to \K^{\mathrm{inv}}(\QQ[x_{-1}^{\pm}])
\] 
is not an equivalence. 
Note that this map is equivalent to the map 
$\K^{\mathrm{inv}}(\QQ[x_{-1}^{\pm}][t_0]) \to \K^{\mathrm{inv}}(\QQ[x_{-1}^{\pm}])$, 
by replacing $t_2$ with $t_0=t_2x_{-1}^2$.
Using the localization sequence (\ref{eqn:locseq}) again, it is enough to show that the total fiber of the square below is not 
zero:
\begin{center}
	\begin{tikzcd}
		\K^{\mathrm{inv}}(\QQ\langle \epsilon_0\rangle[t_0])\ar[r]\ar[d] & \K^{\mathrm{inv}}(\QQ[x_{-1}][t_0])\ar[d]\\
		\K^{\mathrm{inv}}(\QQ\langle \epsilon_0\rangle)\ar[r] & \K^{\mathrm{inv}}(\QQ[x_{-1}]).
	\end{tikzcd}
\end{center}
Since $\QQ\langle \epsilon_0\rangle[t_0]$ and $\QQ\langle \epsilon_0\rangle$ are connective and $\K^{\inv}$ is truncating, this is isomorphic to the total fiber of the square
\begin{center}
	\begin{tikzcd}
		\K^{\mathrm{inv}}(\QQ[t_0])\ar[r]\ar[d] & \K^{\mathrm{inv}}(\QQ[x_{-1}][t_0])\ar[d]\\
		\K^{\mathrm{inv}}(\QQ)\ar[r] & \K^{\mathrm{inv}}(\QQ[x_{-1}]).
	\end{tikzcd}
\end{center}	
For the associated square on $\K$-theory (instead of $\K^{\mathrm{inv}}$), the vertical maps are isomorphisms by 
$\AA^1$-invariance (see \cite[Proposition 5.1]{kcoconn}).
Thus it is enough to see that the total fiber of the square
\begin{center}
	\begin{tikzcd}
		\TC(\QQ[t_0])\ar[r]\ar[d] & \TC(\QQ[x_{-1}][t_0])\ar[d]\\
		\TC(\QQ)\ar[r] & \TC(\QQ[x_{-1}]) 
	\end{tikzcd}
\end{center}
does not vanish.
	
Note that the composite functor $$\Mod(\QQ) \to \Mod(\QQ\langle \epsilon_{0}\rangle) \to  \Mod(\QQ[x_{-1}])$$ is the cofiber 
of the natural endomorphism $x_{-1}$ on the functor induced by the ring map $\QQ \to \QQ[x_{-1}]$, so it is twice this map at 
the level of $\Uloc$. Since we work rationally, we may replace the map with the one coming form the map of rings. 
Using (\ref{eqn:pullbacksquare}) again, it is enough to show that the square
	
	\begin{center}
		\begin{tikzcd}
			\TC(\QQ^{S^2}[t_0])\ar[r]\ar[d] & \TC(\QQ[t_0])\ar[d]\\
			\TC(\QQ^{S^2})\ar[r] & \TC(\QQ) 
		\end{tikzcd}
	\end{center}
	is not a pullback square. Now this is a standard calculation in cyclic homology. For example, if $\epsilon_{-2}$ is the class in $\pi_{-2}$ of $\QQ^{S^2}$, then the class $dt_0d\epsilon_{-2} \in \pi_0\TC(\QQ^{S^2}[t_0])$ lifts to the total fiber of the square.
\end{proof}

\subsection{Jumps of nilpotence ranks}
In the theory of nilpotent extensions of $c$-categories some unintuitive jumps of the nilpotence rank happen under the 
constructions we perform. For instance, in Theorem~\ref{thm:nilpotent_resolution}(3) the nilpotence rank of a functor jumps 
from $k$ to $kn$ where $n$ is the width of the $c$-category. 

Furthermore, in the proof of Theorem~\ref{thm:morphisms_of_stacks}(3) we showed that for 
a nilpotent extension of rank $\le k$ of quasi-geometric stacks satisfying connective perfect generation $f:Y \to X$ induces 
a nilpotent extension of $c$-categories 
\[
f^*:\mathrm{Perf}(X) \to \mathrm{Perf}(Y)
\]
of rank $\le k(md+m+d)(m+1)$, where $d=cd(\pi)$ and 
$m$ is the descendability exponent of $\pi$ for an fppf cover $\pi:\Spec R \to X$. 
Combining this with Remark~\ref{rmk:degree of nilpotence} and we see that motivically the functor $f^*$ is equivalent 
to a (desuspension) of a nilpotent extension of ring spectra of rank not exceeding a very unwieldy number:
\[
(m(d+1))!k(md+m+d)(m+1).
\]
For motivic applications the precise bounds in these results are not important: all localizing invariants are 
nilinvariant as long as they are invariant under nilpotent extensions of rank $1$ (see \cite[Corollary~3.5]{Land_2019}). 
However, in order to better understand these constructions, it would be nice to understand whether our nilpotence 
estimates are optimal. 

\begin{qst}
Can the bounds of nilpotence rank in Theorem~\ref{thm:nilpotent_resolution}(3) and Theorem~\ref{thm:morphisms_of_stacks}(3) 
be improved? Can the optimal bounds be realized in examples?
\end{qst}

\addtocontents{toc}{\protect\setcounter{tocdepth}{2}}
\appendix
\section{Weakly approximable categories}
In \cite{neeman2018triangulated} the notion of a weakly approximable category was introduced to prove certain   
representability theorems. Many categories of geometric origin were shown to be weakly approximable, which gave rise to many 
interesting consequences. In particular, in \cite{neeman2024boundedtstructurescategoryperfect} a proof of 
the conjecture \cite[Conjecture~1.5]{antieau2018ktheoretic} on non-existence of t-structures about 
perfect complexes on singular schemes was given. 
In this section we discuss the relation between weak approximability and the existence of c-structures on the subcategory of 
compact objects.

We note that similar observations to ones in this section have been made in Saunier's PhD thesis 
\cite[Section~11.3.2]{Saunier2025PhD}. In particular, Proposition~\ref{prop:approx_heart_str} is essentially the same as 
\cite[Theorem~11.15]{Saunier2025PhD}. 

\begin{ntn}
Assume $\cC$ is a stable category with small coproducts. 
Let $S$ be a collection of objects of $\cC$ and let $n, m \in \ZZ \cup \{-\infty, +\infty\}$. 

We denote by $\langle S \rangle^{[n,m]}$ the smallest full subcategory of $\cC$ containing the set 
\[
\bigcup\limits_{n\le i\le m} \Sigma^i S
\]
and closed under extensions and retracts. We also denote by $\overline{\langle S \rangle}^{[n,m]}$ the smallest full 
subcategory of $\cC$ containing the set 
\[
\bigcup\limits_{n\le i\le m} \Sigma^i S
\]
and closed under extensions, small coproducts, and retracts.
\end{ntn}

\begin{dfn}\label{dfn:approximable}
Let $\cC$ be a stable presentable category endowed with a t-structure and assume $\cC$ is generated by a compact 
object $G$. 
We say that $\cC$ is {\bf weakly approximable} if there exists an integer $A$ such that the following two conditions are 
satisfied:
\begin{enumerate}
\item[(i)] $G$ belongs to $\cC_{t\ge -A}$ and $\pi_0\Map(G,y) \simeq 0$ for any $y \in\cC_{t\le A}$;
\item[(ii)] For every object $F \in \cC_{t\ge 0}$ there exists a fiber sequence 
\[
E \to F \to D
\]
such that $E \in \overline{\langle G \rangle}^{[-A,A]}$ and $D \in \cC_{t\ge 1}$.
\end{enumerate}
\end{dfn}

\begin{wrn}
Most papers on the subject of weakly approximable categories use the cohomological convention for t-structures. In that 
convention $\cC_{t\ge i}$ becomes $\cC^{t\le -i}$, and $\cC_{t\le i}$ becomes $\cC^{t\ge -i}$. 
The existence of these two different notational schemes is one of the main sources of mistakes for many researchers in the 
field (including the second author), 
so we decided to stick with the homological convention throughout the paper, which is 
more commonly used in homotopy theory.
\end{wrn}

\begin{rmk}
Definition~\ref{dfn:approximable} is originally formulated in the setting of arbitrary triangulated categories 
(that do not have to be homotopy categories of stable categories).
\tqed \end{rmk}

\begin{rmk}\label{rmk:connect_vanish_appr}
The vanishing $\pi_0\Map(G,y) \simeq 0$ in Definition~\ref{dfn:approximable} is equivalent to saying that the 
mapping spectrum $\map(G,y)$ is 1-connective for any $y \in\cC_{t\ge A}$.
\tqed \end{rmk}

\begin{prop}\label{prop:properties_approximability}
Let $\cC$ be a weakly approximable stable presentably category, and let $G,A$ be as in \Cref{dfn:approximable}.
Then the following statements are valid:
\begin{enumerate}
\item For any $F\in \cC_{t\ge 0} \cap \cC^\omega$ there exists a fiber sequence 
$E \to F \to D$
such that $E \in \langle G \rangle^{[-A,A]}$ and $D \in \cC_{t\ge 1}$;
\item for any integer $N$ and any $F\in \cC_{t\ge -N+1}$ there exists a fiber sequence 
$E \to F \to D$
such that $E \in \overline{\langle G \rangle}^{[-A-N,A]}$ and $D \in \cC_{t\ge 1}$;
\item for any integer $N$ and any $F\in \cC_{t\ge -N+1}\cap \cC^\omega$ there exists a fiber sequence 
$E \to F \to D$ 
such that $E \in \langle G \rangle^{[-A-N,A]}$ and $D \in \cC_{t\ge 1}$;
\item for any compact generator $H \in \cC$ conditions $(i)$ and $(ii)$ of Definition~\ref{dfn:approximable} are satisfied 
for $H$ in place of $G$, upon increasing $A$ if necessary.
\end{enumerate}
\end{prop}
\begin{proof}
	Assertion~$(1)$ is \cite[Lemma~2.12]{neeman2018triangulated}. 
	
	We now prove assertion $(2)$, and $(3)$ is proved via the same argument. Using condition $(2)$ of 
	Definition~\ref{dfn:approximable} construct a sequence 
	\[
	F=D_0 \to D_1 \to \dots \to D_{N}
	\]
	of maps such that $D_i \in \cC_{t\ge -N + i + 1}$ and the fiber of $D_i \to D_{i+1}$ is in 
$\overline{\langle G \rangle}^{[-A-N+i,A-N+i]}$. The fiber of the composite is an extension of the fibers, so it belongs 
to $\overline{\langle G \rangle}^{[-A-N,A]}$. 

Assertion~$(4)$ is \cite[Proposition~2.6]{neeman2018triangulated}.
\end{proof}

\begin{prop}\label{prop:approx_heart_str}
	Let $\cC$ be a weakly approximable category. Then $\cC^\omega$ admits a bounded heart structure with 
	\[
	\cC^\omega_{\le 0} = \langle G \rangle^{[-\infty,A]}\quad\quad \text{ and } \quad\quad\cC^\omega_{\ge 0} = \cC_{t\ge 0} \cap \cC^\omega.
	\]
	where $G,A$ are as in \Cref{dfn:approximable}.
\end{prop}
\begin{proof}
	The formulas in the statement give extension closed subcategories of $\cC^{\omega}$. Condition~$(i)$ of 
	Definition~\ref{dfn:approximable}, together with the fact that $G$ is a generator, implies that any $x\in \cC^{\omega}$ 
	also belongs to $\bigcup\limits_n \cC^{\omega}_{\ge -n} \cap \cC^{\omega}_{\le n}$, so it suffices to prove that there 
	are heart decompositions. 
	Since any $x \in \cC^\omega$ belongs to $\cC_{t\ge N}$ for some $N$ this follows from 
	Proposition~\ref{prop:properties_approximability}(3).
\end{proof}

\begin{cor}\label{cor:approximable_ccategory}
	Let $\cC$ be a weakly approximable category. Then for any choice of $G,A$ as in \Cref{dfn:approximable}, the 
	heart structure constructed in Proposition~\ref{prop:approx_heart_str} on $\cC^\omega$ is fcd. In particular, $\cC^\omega$ is 
	a $c$-structure of width $\le 2A$, where $A$ is as in Definition~\ref{dfn:approximable}.
\end{cor}
\begin{proof}
	It suffices to observe that for any $x\in \cC^\omega_{\le 0}$ and any 
	$y\in \cC^\omega_{\ge 2A}$ the spectrum $\map_\cC(x,y)$ is $1$-connective. 
	Condition~$(i)$ of Definition~\ref{dfn:approximable} yields that this is true for $x=\Sigma^i G$ for any $i<A$. 
	But this connectivity property is closed under extensions and retracts, so it is in fact true for any 
	$x\in \cC^\omega_{\le 0}= \langle G \rangle^{[-\infty,A]}$.
\end{proof}

\begin{dfn}\label{dfn:equivalent_tstr}
Let $\cC$ be a stable category. We say that two t-structures $t_1$ and $t_2$ on $\cC$ are equivalent if there 
exist integers $n,m$ such that 
\[
\cC_{t_1\ge n} \subset \cC_{t_2\ge 0}\subset \cC_{t_1\ge m} 
\]
or, equivalently 
\[
\cC_{t_1\le -m} \subset \cC_{t_2\le 0} \subset \cC_{t_1\le -n}.
\]
\end{dfn}

\begin{prop}\label{prop:equiv_t_str_appr}
Let $\cC$ be a stable presentable category compactly generated by an object $G$. Assume $t_1$ and $t_2$ are  
two equivalent t-structures on $\cC$. If the conditions $(i)$ and $(ii)$ in Definition~\ref{dfn:approximable} are satisfied 
for $t_1$ then they are satisfied for $t_2$, after increasing $A$ if necessary.
\end{prop}
\begin{proof}
This is \cite[Lemma~2.5]{neeman2018triangulated}. 
\end{proof}

\begin{thm}[{\cite[Theorem~2.4.2(II.2)]{Bondarko_2021}, \cite[Theorem~3.2.3]{bondarko2024tstructures}}]\label{thm:adjacent_t}
On any compactly generated weighted category $\cC$ there exists a t-structure such that 
$\cC_{t\ge 0} = \cC_{w\ge 0}$.
\end{thm}

\begin{prop}\label{prop:weight_appr}
Any t-structure satisfying the conditions $(i)$ and $(ii)$ is equivalent to the t-structure provided by 
Theorem~\ref{thm:adjacent_t}.
\end{prop}
\begin{proof}
Assume $\cC$ satisfies conditions $(i)$ and $(ii)$ of Definition~\ref{dfn:approximable} for 
a t-structure $t$. We have that $\map(G,\cC_{t\ge A})$ is 1-connective, so $\cC_{t\ge A} \subset \cC_{w\ge 0}$. 
It suffices to prove that we have an inclusion $\cC_{w\ge 0} \subset \cC_{t\ge 0}$. Consider the heart structure of 
Proposition~\ref{prop:approx_heart_str}. The weight structure associated to this heart structure 
(see Construction~\ref{cnstr:t-str-and-weight-str}) is equal to the weight structure generated by $G$. 
Now by Corollary~\ref{cor:weight_heart_relation} we have inclusions 
\[
\cC_{w\ge 0} \subset \Ind(\cC^\omega_{\ge 0}) \subset \cC_{t \ge 0}.
\]
\end{proof}

\begin{rmk}
Choosing the t-structure of Theorem~\ref{thm:adjacent_t} to check weak approximability of $\cC$, 
we can see that the conditions $(i)$ and $(ii)$ correspond to the conditions for $\cC^\omega$ to be a $c$-category. 
Condition~$(i)$ says that any object of $\cC^\omega$ is weight bounded (so, $\cC^\omega$ is a pre-$c$-category) and 
condition~$(ii)$ ensures that for any object of the heart $F \in \cC^{w\heartsuit}$ there is a fiber sequence 
\[
E \to F \to D
\]
where $D \in \cC_{w\ge 1}$ and $E \in \overline{\langle G \rangle}^{[-A,A]}$. In this case the second map is 0, and the first map is a retraction, so $F$ belongs to $\overline{\langle G \rangle}^{[-A,A]}$. But any object of 
$\overline{\langle G \rangle}^{[-A,A]}$ is a colimit of objects of $\langle G \rangle^{[-A,+\infty]} \subset \cC^\omega \cap \cC_{w\ge -2A}$.
\tqed \end{rmk}

\section{\texorpdfstring{$-n$}{-n}-connective embeddings}

A $c$-category $\cC$ of width $\le n$ has the property that it has a strongly continuous (i.e with colimit preserving right 
adjoint) fully faithful embedding of $\Ind(\cC)$ into $\Ind$ of a boundedly weighted category, such that the induced comonad 
on the target sends objects of the weight heart to $-n$-connective objects. One may ask whether the property of admitting such 
an embedding characterizes those categories admitting $c$-structures. The goal of this section is to partially answer this, by 
giving an internal characterization of when a category admits such an embedding.

We begin by preparing to define the notion of interest in \Cref{dfn:-nconnemb}, which we call a $-n$-connective embedding.

\begin{dfn}\label{dfn:indprowtstr}
	Given a small category $\cD$ with weight structure, we define the induced weight structure on $\Ind(\cD)$ to be the  weight structure compactly generated\footnote{It is not necessary here that the weight structure be bounded, and the resulting weight structure may not be hypercomplete} by objects of $\cD_{\leq 0}$. Similarly one defines a weight structure on $\Pro(\cD)$ by writing $\Pro(\cD) = \Ind(\cD^{op})^{op}$.
\end{dfn}


\begin{lem}\label{lem:wtstrind}
The heart of $\Ind(\cD)$ consists of retracts of infinite sums of objects in the heart of $\cD$. 
Moreover, $\Ind(\cD)_{w\geq0}$ is generated under (filtered) colimits by $\cD_{w\geq0}$.
\end{lem}
\begin{proof}
We start by verifying the first claim. It is easy to see that such objects are in the heart. Conversely, given an object $a$ 
in the heart, $\map(b,a)$ is connective for each $b$ in $\cD^{w\heartsuit}$. One can take $b'$ to be the direct sum over all 
homotopy classes of such maps, and consider $c$, the cofiber of the map $b' \to a$. We claim that $c$ is weight connected. To see this, 
given a map $z \to c$ with $z \in \cD_{\leq0}$, the map factors through $z_{\geq0} \in \cD^{w\heartsuit}$ since $c$ is weight 
connective, so we can assume $z$ is in the weight heart of $\cD$. The composite $z \to c \to \Sigma b'$ is null, so our map 
from $z$ factors through $a$. But then it factors through $b'$, so the original map must be null. It then follows that $a$ is 
a retract of $b'$.
	
Now we check the second claim. Objects that are colimits of those in $\cD_{w\geq0}$ are clearly in $\Ind(\cD)_{w\geq0}$. 
Conversely, consider the minimal subcategory $T$ of $\Ind(\cD)$ containing $\cD_{w\geq0}$ and closed under filtered 
colimits. Note that $T$ has all finite colimits and also contains the weight heart by the first claim. 
Now consider an arbitrary object $x$ in $\Ind(\cD)_{w\ge 0}$ and consider the fiber sequence 
\[
	\bar{x} := \colim\limits_n \Omega w_{\ge n+1} x \to \colim\limits_n w_{\le n} x \to x.
\]
The object on the middle belongs in $T$, since it is bounded in the weight structure, so is a filtered colimit of objects that are finite colimits of objects of the 
weight heart. So it suffices to show that the object on the left does. Consider the stable presentable subcategory 
$\bigcap\limits_n \Ind(\cD)_{w\ge n}$. Any object $y$ in there admits a non-zero map from some $u\in\cD$, 
and by orthogonality this factors through each step of the weight decomposition of $u$.
This allows us to construct a non-zero map $\colim w_{\ge n} u \to y$. In particular, $\bigcap\limits_n \Ind(\cD)_{w\ge n}$ is 
generated by objects of the form $\colim\limits_n w_{\ge n} u$ for $u\in \cD$. 
\end{proof}

\begin{lem}\label{lem:eqcondnconnemb}
	Suppose that $f:\cC \to \cD$ is a fully faithful embedding of small stable categories, where $\cD$ has a weight structure.
	Then the following conditions are equivalent:
	
	\begin{enumerate}
		\item For each $d \in \cD_{w\geq0}$, $f^*(d)$ is a filtered colimit of compact objects of $\Ind(\cC)$ such that for every compact object $x_{\alpha}$, all maps out of $x_{\alpha}$ in the diagram eventually kill weights $\le -n-1$.
		\item The endofunctor $f_!f^*$ of $\Ind(\cD)$ sends connective objects to $-n$-connective objects.
	\end{enumerate}
	
	The following two conditions are also equivalent to each other, and become equivalent to $(1)$ and $(2)$ under the assumption that the image of each object of $\cC$ lands inside bounded above objects of $\cD$.
	
	\begin{enumerate}[start=1,label={(\arabic*'):}]
		\item For each $d \in \cD^{w\heart}$, $f^*(d)$ is a filtered colimit of compact objects of $\Ind(\cC)$ such that for every compact object $x_{\alpha}$, all maps out of $x_{\alpha}$ in the diagram eventually kill weights $\le -n-1$.
		\item The endofunctor $f_!f^*$ of $\Ind(\cD)$ sends the weight heart to $-n$-connective objects.
	\end{enumerate}	
	
	Finally, if the image of $\cC$ lands inside bounded below objects, then $(1')$ and $(2')$ are equivalent to condition $(3)$:
	\begin{enumerate}[start=3,label={(\arabic*):}]
		\item If $L:\cD \to \cD/\cC$ is the quotient map, then for each $d \in \cD^{w\heart}$, the map $\End(d) \to \End(Ld)$ is $-n$-connective.
	\end{enumerate} 
\end{lem}

\begin{proof}
	We first show that $(2)$ implies $(1)$, an identical argument shows $(2')$ implies $(1')$. Suppose $f_!f^*$ of $\Ind(\cD)$ 
	sends connective objects to $-n$-connective objects. Then for each $d \in \cD_{w\geq0}$, $f_!f^*(d)$ is $-n$-connective, 
	which by \Cref{lem:wtstrind} means that $f_!f^*(d)$ is a filtered colimit of compact $-n$-connective objects $y_{\beta} \in \cD$. $f^*(d)$ is also a 
	filtered colimit of compact objects $x_{\alpha} \in \cC$, so one learns that in the diagram for each $x_{\alpha}$ the maps eventually factor through some $y_{\beta}$ (viewed in $\cD$), and thus eventually kill weights $\le -n-1$.
	
	Next, we show that $(1)$ implies $(2)$, an identical argument showing that $(1')$ implies $(2')$. In order to show $(2)$, 
	by \Cref{lem:wtstrind}, it suffices to show that objects $d \in \cD_{w\geq0}$ are sent to $\Ind(\cD)_{w\geq-n}$ under 
	$f_!f^*$. In other words, we need to show that $f_!f^*(d)$ has no nonzero homotopy classes of maps from 
	$x \in \cD_{w\leq-n-1}$. However, any map $x \to f_!f^*(d)$ factors through a compact object, and so by assumption $(1)$ 
	it also factors through a $-n$-connective object, thus the map is null.
	
	To see that $(2)$ and $(2')$ are equivalent under the assumption that $\cC$ lands in bounded above objects, we note that 
	$(2')$ is equivalent to the claim that $f_!f^*$ sends connective bounded objects to $-n$-connective objects. But if the 
	image of $\cC$ is bounded above, then $f^*$ agrees on $\colim_i w_{\leq i}d$ and $d$, showing that $(2)$ and $(2')$ are 
	equivalent.
	
	We next assume $(2')$ and show $(3)$. Then we can identify the fiber of $\End(d) \to \End(Ld)$ with the fiber of 
	$\map(d,d) \to \map(d,L^*L_!d)$, which is $\map(d,f_!f^*d)$, which is $-n$-connective because of $(2')$. Conversely, 
	assuming $\map(d,f_!f^*d)$ to be $-n$-connective for each $d \in \cD^{\heartsuit}$, and that the image of $\cC$ to be 
	bounded below, we must show that $d'=f_!f^*d$ is in $\Ind(\cD)_{w \geq-n}$. The assumption implies that $\map(a,d')$ is $-n$-connective 
	whenever $a$ is in the weight heart, by replacing $d$ with $a\oplus d$. Now it suffices to show that for each compact 
	object $c$ in $\cC$ mapping to $d'$, the map factors through a $-n$-connective object. Indeed, this is because then $d'$ is a filtered colimit along maps that eventually kill weights $\leq -n-1$, so is $-n$-connective. But $c$ is bounded below, so 
	$w_{\leq -n-1}c$ is bounded, so we find that the composite $w_{\leq -n-1}c \to c \to d'$ is null since $w_{\leq -n-1}c$ is 
	a finite extensions of objects in shifts of the heart by $-n-1$ or less.
\end{proof}


\begin{rmk}
	In condition $(3)$, if $n\geq1$, the condition that $\End(d) \to \End(Ld)$ is $-n$-connective is equivalent to saying that $\End(Ld)$ is $1-n$-connective.
\tqed \end{rmk}

\begin{dfn}\label{dfn:-nconnemb}
	Let $\cD$ be a stable category with a weight structure. For a small category $\cC$, we say that an $-n$-connective embedding $f:\cC \to \cD$ is a fully faithful embedding such that:
	
	\begin{enumerate}
		\item The image of $\cC$ lands inside bounded objects of $\cD$.
		\item Any of the equivalent conditions of \Cref{lem:eqcondnconnemb} are satisfied.
	\end{enumerate}
\end{dfn}

\begin{lem}\label{lem:opp}
	The embedding $f:\cC \to \cD$ is an $-n$-connective embedding iff $f^{op}:\cC^{op} \to \cD^{op}$ is an $-n$-connective embedding.
\end{lem}

\begin{proof}
	Condition $(1)$ of \Cref{dfn:-nconnemb} is closed under passing to the opposite category, so it suffices to show that the 
	condition $(3)$ of \Cref{lem:eqcondnconnemb} is closed under passing to the opposite functor. But the sequence 
	$\cC^{op} \to \cD^{op} \to (\cD/\cC)^{op}$ is a localization sequence, and $\End(d) \to \End(Ld)$ is clearly 
	$-n$-connective iff $\End(d)^{op} \to \End(Ld)^{op}$ is.
\end{proof}

\begin{lem}\label{lem:subcat}
	Let $\cC \to \cD$ be an $-n$-connective embedding, and let $\cD' \subset \cD$ be a stable weighted subcategory\footnote{this means that the inclusion into $\cD$ is weight exact.} of $\cD$ containing the image of $\cC$. Then $\cC \to \cD'$ is an $-n$-connective embedding.
\end{lem}{}
\begin{proof}
	Follows from Lemma~\ref{lem:eqcondnconnemb}(1) and the fact that the right adjoint to the inclusion $\Ind(\cD') \subset \Ind(\cD)$ preserves weight connectivity. 
\end{proof}

\begin{lem}\label{lem:ind}
	If the embedding $f:\cC \to \cD$ is an $-n$-connective embedding, then $\cC \to \cD \to \Ind(\cD)^{\kappa}$ and $\cC \to \cD \to \Pro(\cD)^{\kappa}$ are too.
\end{lem}
\begin{proof}
	By \Cref{lem:opp}, it suffices to show the result for $\Ind(\cD)^{\kappa}$. By \Cref{lem:wtstrind}, objects of the heart are (retracts of) infinite sums of objects of the heart of $\cD$, from which one sees that condition $(3)$ is satisfied.
\end{proof}

\begin{lem}\label{lem:changewtstr}
	Suppose $\cC \xrightarrow{f} \cD \xrightarrow{g} \cD'$ are fully faithful functors, and suppose $\cD$, $\cD'$ have weight structures such that the image of $\cC$ is bounded. Suppose further that $g$ sends connective objects to $a$-connective objects and $g_*$ sends connective objects to $b$-connective objects. Then:
	
	\begin{itemize}
		\item If $f$ is an $-n$-connective embedding, then $g\circ f$ is an $-n+a+b$-connective embedding.
		\item if $g\circ f$ is an $-n$-connective embedding, then $f$ is an $-n+a+b$-connective embedding.
	\end{itemize}
\end{lem}
\begin{proof}
	Suppose $f$ is a $-n$-connective embedding. Given $d \in \cD'_{\geq0}$, we can first apply $g^*$ to get an object in $\Ind(\cD)_{\geq b}$, and then apply $f_!f^*$ to get a $-n+a$-connective object. We can then apply $g_!$ to get a $-n+a+b$-connective object, showing that $g\circ f$ is a $-n+a+b$-connective embedding.
	
	Suppose that $g\circ f$ is an $-n$-connective embedding. Given $d \in \cD_{\geq0}$, $g(d)$ is $a$-connective, so $(gf)_!(gf)^*g(d)$ is $-n+a$-connective, and so $f_!(gf)^*g(d)$ is $-n+a+b$-connective, showing that $f$ is an $-n+a+b$-connective embedding.
\end{proof}

\begin{thm}\label{thm:weakimpliesstrong}
	Suppose that $f:\cC \to \cD$ is an $-n$-connective embedding. Then there is a natural weight structure on $\Pro(\Ind(\cC)^{\kappa})^{\kappa'}$ for all sufficiently large $\kappa$, $\kappa'$ depending on $\kappa$, such that the canonical embedding $\cC \to \Pro(\Ind(\cC)^{\kappa})^{\kappa'}$ is a $-2n$-connective embedding.
\end{thm}

\begin{proof}
	For large enough $\kappa$, $f^*$ restricts to $\kappa$-compact objects. For large enough $\kappa'$, we can generate a 
	weight structure on $\Pro(\Ind(\cC)^{\kappa})^{\kappa'}$ by using $f^*(\cD_{w\geq0})$ as cocompact generators for the 
	nonnegative objects. Clearly $\Pro(f^*)^{\kappa'}:\Pro(\Ind(\cD)^{\kappa})^{\kappa'} \to \Pro(\Ind(\cC)^{\kappa})^{\kappa'}$ 
	sends connective objects to connective objects in this weight structure. We claim that the left 
	adjoint\footnote{Since $\Pro(-)$ is a $2$-functor, it preserves adjunctions.} $\Pro(f_!(-))^{\kappa'}$ sends connective 
	objects to $-n$-connective objects. To see this, its image is generated by $f_!f^*(\cD_{w\ge 0})$ 
	under operations that $-n$-connective objects are closed under.
	
	Now we have the diagram $\cC \to \Pro(\Ind(\cC)^{\kappa})^{\kappa'} \to \Pro(\Ind(\cD)^{\kappa})^{\kappa'}$. The embedding 
	into $\Pro(\Ind(\cD)^{\kappa})^{\kappa'}$ is an $-n$-connective embedding by applying \Cref{lem:ind}, and the image of 
	$\cC$ is bounded. It follows by applying \Cref{lem:changewtstr} that the embedding into 
	$\Pro(\Ind(\cC)^{\kappa})^{\kappa'}$ is $-2n$-connective.
\end{proof}

\nocite{}
\printbibliography

\end{document}